\documentclass[a4paper]{article}

\usepackage{pgf,tikz}
\usepackage[multiple]{footmisc}
\usetikzlibrary{arrows}

\usepackage{amsthm,amssymb,amsmath,enumerate,graphicx,epsf,mathtools,multicol,xfrac,color}

\newtheorem{proposition}{Proposition}[section]
\newtheorem{definition}[proposition]{Definition}
\newtheorem{theorem}[proposition]{Theorem}
\newtheorem{corollary}[proposition]{Corollary}

\newtheorem{lemma}[proposition]{Lemma}

\newtheorem*{cla}{Claim}
\newenvironment{poc}{\begin{proof}[Proof of claim]}{\end{proof}}

\numberwithin{equation}{section}

\newcommand{\var}{\operatorname{Var}}

\newcommand{\pr}[1]{\mathbb{P}\!\left(#1\right)}
\newcommand{\ex}[1]{\mathbb{E}\!\left[#1\right]}
\newcommand{\pru}[2]{\mathbb{P}_{#1}\!\left(#2\right)}
\newcommand{\exu}[2]{\mathbb{E}_{#1}\!\left[#2\right]}

\newcommand{\mbfs}{\operatorname{MBFS}}
\newcommand{\eps}{\ensuremath{\epsilon}}

\newcommand{\ca}{\ensuremath{\mathcal A}}
\newcommand{\cb}{\ensuremath{\mathcal B}}
\newcommand{\cc}{\ensuremath{\mathcal C}}
\newcommand{\cd}{\ensuremath{\mathcal D}}
\newcommand{\ce}{\ensuremath{\mathcal E}}
\newcommand{\cf}{\ensuremath{\mathcal F}}
\newcommand{\cg}{\ensuremath{\mathcal G}}
\newcommand{\ch}{\ensuremath{\mathcal H}}
\newcommand{\ci}{\ensuremath{\mathcal I}}
\newcommand{\cj}{\ensuremath{\mathcal J}}

\newcommand{\cl}{\ensuremath{\mathcal L}}
\newcommand{\cm}{\ensuremath{\mathcal M}}

\newcommand{\er}{Erd\H{o}s-R\'{e}nyi}
\newcommand{\epn}{\varepsilon_n}
\newcommand{\pepn}{\left(1+\BO{\varepsilon_n}\right)}
\newcommand{\mepn}{\left(1-\BO{\varepsilon_n}\right)}
\newcommand{\BO}[1]{O\!\left(#1\right)} 
\newcommand{\lo}[1]{o\!\left(#1\right)}

\title{Random Walk Hitting Times and Effective Resistance in Sparsely Connected Erd\H{o}s-R\'{e}nyi Random Graphs}
\allowdisplaybreaks 
\usepackage[margin=1.25 in]{geometry}
\author{John Sylvester
\medskip\\
  {\small University of Glasgow, Glasgow, United Kingdom}\\
  {\small \texttt{firstname.lastname@glasgow.ac.uk} }}
\providecommand{\key}[1]
{  {\small 	\textbf{Keywords: } #1}
}
\providecommand{\ams}[1]
{  {	\small \textbf{AMS Subject Classification: } #1}
}

\date{}
\begin{document}
\maketitle

\begin{abstract}
We prove a bound on the effective resistance $R(x,y)$ between two vertices $x,y$ of a connected graph which contains a suitably well-connected sub-graph. We apply this bound to the Erd\H{o}s-R\'{e}nyi random graph $\mathcal{G}\left(n,p\right)$ with $np=\Omega(\log n)$,  proving that $R(x,y)$ concentrates around $1/d(x) + 1/d(y)$, that is, the sum of reciprocal degrees. We also prove expectation and concentration results for the random walk hitting times, Kirchoff index, cover cost, and the random target time (Kemeny's constant) on $\mathcal{G}\left(n,p\right)$ in the sparsely connected regime $\log n + \log\log \log n \leq np < n^{1/10}$.       	
\end{abstract}
\key{Random Graph, Random Walk, Effective Resistance, Hitting Time, Kirchoff Index}\\
\ams{Primary 05C80, 05C81, 60C05; Secondary 05C40, 60J85, 60J45.}

\section{Introduction \& Results} 
The effective resistance $R(x,y)$ between two vertices $x,y$ of a graph $G=(V,E)$ is the energy dissipated by a unit current flow from $x$ to $y$ when all edges have unit resistances. That is 
\begin{equation}\label{resistance}R(x,y) =\inf_{\theta} \left\{\sum_{e\in E}\theta(e)^2 : \theta \text{ is a unit flow from $x$ to $y$}\right\},  \end{equation}see Section \ref{prelim} for a complete mathematical formulation. The effective resistance has connections to Markov chain theory, in particular for infinite graphs the transience or recurrence of a random walk is determined by the resistance from the origin to cut sets at arbitrary distance from the origin \cite{doyle1984random}. In finite graphs the resistances determine hitting times of random walks \cite{tetali1991random} and are related to the eigenvalues of the Laplacian \cite{lovasz1993random}.

We prove a new bound on effective resistance for graphs $G$ containing a sub-graph $H$ with good connectivity properties. The result, Theorem \ref{resbdd}, may be stated (very) informally as 
\[ R(x,y)\leq \frac{1}{d(x)} + \frac{1}{d(y)} + \mathrm{Error}_{x,y}(G,H), \] 
where $d(\cdot )$ is the degrees of a vertex and $\mathrm{Error}_{x,y}(G,H)$ is an error term. Depending on the graph $G$ and the sub-graph $H$ chosen the error may be insignificant compared to the other terms. In that case our bound is essentially tight as $1/(d(x)+1) + 1/(d(y)+1)$ is always a lower bound on $R(x,y)$, see Section \ref{ressec} for a full statement of Theorem \ref{resbdd}. Although this bound holds for any connected graph, the $\mathrm{Error}_{x,y}(G,H)$ term may dominate; our bound works well for graphs with strong expansion properties. We apply this bound to the random graph $\mathcal{G}\sim \mathcal{G}(n,p)$, that is the simple graph $\cg$ on $n$ vertices with law $\cg(n,p)$ given by sampling each edge independently with probability $p$. The random graph $\cg(n,p)$ has been extensively studied \cite{bollobasrandom,Frbook,JSL} and so it is a natural question to determine the effective resistance for such a fundamental graph distribution. We remark that throughout all $\log$'s are base $e$ and we define $\epn:=\epn(n,p)$ to be the function
\begin{equation}\label{edef}
\varepsilon_n : = \frac{\log  n}{ np\log (np)}.
\end{equation}

\begin{theorem}\label{resconc}For any $c>0$ let $\mathcal{G}\sim  \mathcal{G}(n,p)$ with $c \log n\leq np  \leq n^{1/10}$.  Then for a fixed $ i, j\in V$ where $i\neq j$, 
	\[\pr{\left|R(i,j) - \left(\!\frac{1}{{d}(i)} + \frac{1}{{d}(j)}\!\right)\!\right| > 9\frac{2+\epn{d}(i)}{{d}(i)^2} + 9\frac{2+\epn{d}(j)}{{d}(j)^2} }=o\!\left(e^{-np/4} \right)+ o\!\left(n^{-3}\right) .\]
\end{theorem}
 Notice if $np = \Omega(\log n )$ then $\epn = \lo{1}$. Theorem \ref{resconc} shows that with high probability (w.h.p.) the main contribution to the effective resistance $R(i,j)$ between vertices $i,j \in V$ comes from the flow through edges connecting $i$ and $j$ to their immediate neighbours.

 From the definition \eqref{resistance} of $R(x,y)$ one observes that the contribution to the resistance from each edge in the graph is quadratic in the amount of flow passing through that edge. The main idea of Theorem \ref{resbdd} is to show that if a graph contains a sub-graph from a certain family well connected graphs then there are many paths between the neighbours of $x$ and $y$ which become edge disjoint away from $x$ and $y$. If this is the case then the flow can be divided up evenly between the edges close to $x$ and $y$, further away we use the edge disjoint paths to carry the flow. In a graph with good expansion there should be many paths such paths and so the flow through these edges should be negligible compared to that through edges close to $x$ and $y$. However since we balanced the flow evenly over edges close to $x$ and $y$ the contribution to $R(x,y)$ from these edges is close to optimal and matches a simple lower bound up to lower order terms. Aside from $\cg(n,p)$ our bound on resistance, Theorem \ref{resbdd}, may potentially be applied to other random graph models such as binomial random intersection graphs \cite[\S 11]{Frbook} and Chung-Lu graphs \cite{Chung} in certain regimes. These regimes where this bound may be effective are those where there is constant minimum degree, the average degrees is large, and it is hard to get good enough control on the spectral statistics to apply spectral methods to obtain estimates on the resistances or hitting times with the correct leading constant.

We also consider expected hitting times $h(i,j)$ of a random walk. Let $\mathbf{P}^G_i$ be the law of a simple random walk (SRW) $X_t$ on $G$, that is the random process which at each step moves to a uniformly chosen neighbour of the current vertex, then $ h(i,j) :=\mathbf{E}^G\left[ \tau_j \mid X_0=i \right]$ where $\tau_j:=\inf \left\{t:X_t=j\right\}$. Hitting times are well studied in Markov chain theory \cite{aldousfill,peresmix}. They also feature in randomised algorithms, for example the run time of the original $\mathsf{LOGSPACE}$ algorithm for undirected complexity \cite{lovasz1993random}, and are a popular tool in machine learning to analyze the structure of graphs \cite{von2010hitting}. Tetali's formula \cite{tetali1991random} relates hitting times to resistances:
\begin{align}
h(i, j) &=  |E(G)|\cdot R(i, j) + \sum\limits_{u\in V} \frac{{d}(u)}{2}\left[R(j,u)-R(u, i)\right].\label{eq:hit}
\end{align}Using Tetali's formula we derive results for hitting times and related qualities via controlling resistances. We shall focus on the following regime for $\cg(n,p)$ which we call sparsely connected: 
\begin{equation} \label{sparsecon}\log n + \log\log \log n \leq np  \leq n^{1/10}. \end{equation}
Recall that $\cg(n,p)$ has average degree $np$, however at the lower end of the range \eqref{sparsecon} it also has vertices of constant degree w.h.p. so, at the lower end, $\cg(n,p)$ is far from being regular. 

Let $\mathcal{C}:=\cc_n$ be the event $\mathcal{G}$  is connected and $\ex{\cdot \big|{\mathcal{C}}} $ be conditional expectation w.r.t. $\mathcal{G}(n,p)$. 
\begin{theorem} \label{exthm}
	Let $\mathcal{G}\sim  \mathcal{G}(n,p)$ satisfy \eqref{sparsecon}. Then for any $i,j\in V(\mathcal{G})$ where $i \neq j$,
	\[\ex{R(i,j)\big|{\mathcal{C}}} =   \frac{2\pm \BO{\epn} }{np} \quad \qquad \text{and} \qquad \quad \ex{h(i,j)\big|{\mathcal{C}}} =  n\, (1\pm \BO{\epn}). \] \end{theorem}
We obtain concentration for resistances and hitting times from Theorem \ref{resconc}.

\begin{theorem}\label{resconc2}For any $c>0$ let $\mathcal{G}\sim  \mathcal{G}(n,p)$ with $np = (c\pm \lo{1}) \log n$. Then for fixed $ i\neq j\in V$,
	\begin{enumerate}[(i)]
		\item \label{itm:co2} $\displaystyle{\pr{\left|R(i,j) - \frac{2}{np} \right| > \frac{10 }{c^{2}\log( n) \log\log( n) }} \leq e^{-\Omega\left(\frac{\log n }{(\log\log n)^2} \right)}.}$
	\end{enumerate}Further, if $np = \omega(\log n)$ and $np\leq n^{1/10}$ then
	\begin{enumerate}[(i)]\setcounter{enumi}{1}
		\item \label{itm:co3}  $\displaystyle{\pr{\sup\limits_{\{i,j\}\subseteq V}\left|R(i,j) - \frac{2}{np}\right|\! > \!\frac{7\sqrt{\log n}}{(np)^{3/2}}}= o\left(\frac{1}{n}\right)}$
		\item \label{itm:co4} $\displaystyle{\pr{\sup\limits_{\{i,j\}\subseteq V} \left|h(i,j) - n  \right| > 12n \sqrt{\frac{\log n}{np}} } = o\left(\frac{1}{n}\right)}.$\end{enumerate}
\end{theorem}
Observe Theorem \ref{resconc2} \eqref{itm:co4} gives concentration of $h(i,j)$ around $n$ for all pairs $i,j\in V$ when $np=\omega(\log n)$. For $np= \Theta(\log n )$ we prove concentration by the second moment method.  
\begin{theorem}\label{concentrationnew} Let $\mathcal{G}\sim  \mathcal{G}(n,p)$ satisfy \eqref{sparsecon}, $f(n):\mathbb{N}\rightarrow \mathbb{R}_+$. Then for a fixed $i,j \in V, i \neq j$,
	\[\pr{\Big|h(i,j)  -    n \Big| > n\sqrt{ f(n)\cdot \epn}  \;}= O\left(\frac{1}{f(n)}\right).\] 
\end{theorem}

In particular by choosing $f(n)=\log \log (np)$ above we have concentration for a fixed pair $i,j \in V$, not all pairs;  the following proposition shows that this is best possible.

\begin{proposition}\label{tight}Let $\cg \sim \cg(n,p)$. If $np = \log n + 100\log\log\log n $, then w.h.p. there exists $i,j \in V$ such that $R(i,j)\geq 1$ and $h(i,j) > n\log(n)/3$. For any $1<c<\infty $, if $np = c\log(n)$ then there is an $a>0$ and $i,j \in V$ such that w.h.p. $R(i,j)\geq (2+a)/np $ and $ h(i,j) > (1+a)n$.
\end{proposition}
Theorems \ref{exthm}-\ref{concentrationnew} are valid only for $np  \leq n^{1/10}$, however concentration and expectation for all of the aforementioned random variables has been determined for $np$ above this range. The original contribution of this paper is determining expectation and concentration close to the connectivity threshold $np=\log n$, see the literature review in Section \ref{bkgrd} for more details.

One consequence of applying Theorem \ref{resbdd} to $\cg(n,p)$ is that we can also show that there are many ways to select a edge-disjoint paths between the first neighbours of a pair of vertices. In particular for a graph $G$ let $paths_2(i,j,l)$ be the maximum number of paths of length at most $l$ between vertices $i$ and $j$ of $G$ that are vertex disjoint on $V\backslash \left(B_1(i)\cup B_1(j)\right)$.

\begin{theorem}\label{bolthom}Let $\mathcal{G}\sim  \mathcal{G}(n,p)$ where for any $c>0$, $c\log n \leq np \leq n^{1/10}$. Let $l:=\log n/\log (np) + 9$. Then for $i,j \in V$ where $ i\neq j$,
	\begin{enumerate}[(i)]
		\item \label{itm:bolthom1} $\displaystyle{\pr{paths_2(i,j,l) \neq  \min\{{d}_2(i),{d}_2(j) \}}\leq 5n^3p^4 + o\left(e^{-7\min \left\{np,\log n \right\}/2}\right)}$,
		\item \label{itm:bolthom2} $\displaystyle{\pr{\left|paths_2(i,j,l) -(np)^2\right|> 3(np)^{3/2}\sqrt{\log np}} = o\left(1/np\right)}$.
\end{enumerate}\end{theorem}

We also prove results for some other related indices which appear in the literature for $\cg(n,p)$. For a discussion of how our results extend previous work see Section \ref{bkgrd}.

Let $\pi(v) = d(v)/2|E| $ for $v \in V$ be the stationary distribution of the SRW on $G$ and define,
\begin{equation}\label{eq:kem}
H_j(G) := \sum\limits_{i \in V} \pi(i) h(i,j)\text{ for }j \in V,\qquad  T(G) := \sum\limits_{j \in V} \pi(j) h(i,j).
\end{equation}
The index $H_j(G)$ is known as the stationary hitting time to $j$ \cite{lowe2013hitting} and $T(G)$ is the random target time or Kemeny's constant \cite{aldousfill,peresmix}. Note that $T(G)$ is independent of the vertex $i$ in \eqref{eq:kem}, see \cite[Eq. 3.3]{lovasz1993random}, and the expected running time of Wilson's algorithm  \cite{wilson1996generating} on $G$ is $O(T(G))$.

\begin{theorem} \label{exthm2}
	Let $\mathcal{G}\sim  \mathcal{G}(n,p)$ satisfy \eqref{sparsecon}. Then for any $i\in V(\mathcal{G})$,
	\[\ex{H_i(\mathcal{G})\big|{\mathcal{C}}}  =n\, (1\pm \BO{\epn}) \quad \qquad \text{and} \qquad \quad \ex{T(\mathcal{G})\big|{\mathcal{C}}} =n\, (1\pm \BO{\epn}). \] 
\end{theorem}
 The Kirchoff index $K(G)$ and cover cost $cc_i(G)$ of a finite connected graph $G$ are defined by
\begin{equation}\label{Kirchoff}
K(G):=  \sum\limits_{\left\{i,j\right\}\subseteq  V}R(i,j), \qquad \text{ and }\qquad  cc_i(G) := \frac{1}{n-1} \sum\limits_{j \in V} h(i,j) \quad \text{for }i \in V.
\end{equation}The former is studied in the contexts of mathematical chemistry \cite{Ran} and sensor networks \cite{boumal2014concentration}, and the latter was introduced to bound the cover time \cite{cc,AgW}. By linearity of expectation: 
\begin{corollary}[of Theorem \ref{exthm}] \label{extcor}
	Let $\mathcal{G}\sim  \mathcal{G}(n,p)$ satisfy \eqref{sparsecon}. Then for any $i\in V(\mathcal{G})$,
	\[\ex{K(\mathcal{G})\big|{\mathcal{C}}}= \frac{n}{p} (1\pm \BO{\epn})\quad \qquad \text{and} \qquad \quad \ex{cc_i(\mathcal{G})\big|{\mathcal{C}}}=n\, (1\pm \BO{\epn}). \]
\end{corollary}

We prove concentration for these random variables on $\cg(n,p)$ by the  second moment method.  
\begin{theorem}\label{concentration}  Let $\mathcal{G}\sim  \mathcal{G}(n,p)$ satisfy \eqref{sparsecon} and let $f(n):\mathbb{N}\rightarrow \mathbb{R}_+$. Fix $i \in V$ and let $X $ be any of the random variables $K(\cg), \, H_i(\cg) , \,T(\cg),\, cc_i(\cg)$. Then
	\[\pr{\Big|X  -    \ex{X\big|\cc} \Big| > \ex{X\big|\cc}\sqrt{f(n)\cdot \epn }  \;}= O\left(\frac{1}{f(n)}\right).\] 
\end{theorem}

\subsection{Outline of the Paper}
Section \ref{ressec} contains our bounds on the effective resistance. In particular, in Subsection \ref{genbdd} we prove a general bound, Theorem \ref{resbdd}, which is based on the existence and structure of a desirable sub-graph $H$. In Subsection \ref{apppg} we describe a specialisation of this bound based on a specific family of sub-graphs defined by an exploration process which is well suited to $\cg(n,p)$. In Section \ref{strgnp} we prove preliminary results regarding the sub-graph of $\cg(n,p )$ described in Section \ref{apppg}, these results are needed to apply our bounds on effective resistance to $\cg(n,p)$. In Section \ref{ResThmSec} we apply the results of Sections \ref{ressec} and \ref{strgnp} to prove Theorems \ref{resconc} and \ref{resconc2}, which determine resistances in $\cg(n,p)$, and Theorem \ref{bolthom}, which concerns paths between second neighbours in $\cg(n,p)$. Finally in Section \ref{FinalSec} we combine results from the previous three sections to prove Theorems \ref{exthm}, \ref{concentrationnew}, \ref{exthm2} \& \ref{concentration} which are results on the expectation and concentration for hitting times of random walks and related indices on $\cg(n,p)$. In the remainder of this section we shall discuss some related work and how our work extends known results and also cover some preliminary material.

\subsection{Related Work}\label{bkgrd}

In \cite{jonasson1998cover} Jonasson studies the cover time, that is the expected time to visit all vertices from the worst starting vertex, for $\mathcal{G}(n,p)$. He bounds the cover time by showing effective resistances and hitting times on $\mathcal{G}(n,p)$ concentrate in the regimes where $\omega(\log n)=np\leq n^{1/3}$. Jonasson does not use spectral methods and instead bounds the effective resistance by finding a suitable flow. This is the approach we have also taken, using a refined analysis we extend Jonasson's results for hitting times to the range \eqref{sparsecon} and for effective resistance to the case $np=\Omega(\log n)$. It is worth noting that the cover time has since been determined for all connected $\cg(n,p)$ by Cooper \& Frieze \cite{CFcov} using the first visit Lemma and mixing time estimates.

Let $L=D-A$ be the graph Laplacian, where $A$ is the adjacency matrix and $D$ is the diagonal matrix with $D_{i,j}={d}(i)$ if $i=j$ and $D_{i,j}=0$ otherwise \cite{Ran,lovasz1993random}. Many previous results rely on exploiting connections between resistances or hitting times and spectral statistics of $L$ or other representations of the graph. In this paper we do not employ spectral methods; the results we achieve hold for $\mathcal{G}(n, p)$ close to the connectivity threshold where the minimum degree is $1$ w.h.p.\ and it is hard to obtain good enough estimates on the relevant spectral statistics.

Boumal \& Cheng \cite{boumal2014concentration} exploit an expression for the Kirchoff index $K(G)$ in terms of the trace of $L^{\dagger}(G)$, the Moore-Penrose pseudoinverse of $L(G)$ \cite{Ran}, to obtain expectation and concentration for $K(\mathcal{G})$ on $\mathcal{G}(n,p)$ with $np= \omega\left((\log n)^6\right)$. We will now outline a related expression for $K(G)$ and explain how this can also be used with spectral statistics to control $K(G)$. Let $\lambda_i$ be the eigenvalues of $L(G)$, where $G$ is a finite connected graph. Then by the matrix tree theorem \cite{AgW}:
\begin{equation}\label{Klap} K(G) = \sum\limits_{\lambda_i \neq 0} \frac{1}{\lambda_i}. \end{equation}
A theorem of Coja-Oghlan, \cite[Theorem 1.3]{coja2007laplacian}, states that if $\mathcal{G} \sim  \mathcal{G}(n,p)$ with $np \geq C_0\log n$ for sufficiently large $C_0$ the non-zero eigenvalues of $L(\mathcal{G})$ concentrate around the mean. Combining these estimates with \eqref{Klap} yields concentration for $K(\mathcal{G})$ and with extra work the leading order term of $\ex{K(\mathcal{G})\big|{\mathcal{C}}}$ can be determined when $np \geq C_0\log n$. It is of note however that Boumal \& Cheng obtain second order terms for $\ex{K(\mathcal{G})\big|{\mathcal{C}}}$, which is not possible with the latter method. Theorems \ref{exthm} and \ref{concentration} give expectation and concentration for $K(\cg)$ in the range \eqref{sparsecon}.

L{\"o}we \& Torres \cite{lowe2013hitting} obtain concentration results for $T(\mathcal{G}), H_i(\mathcal{G})$ and also the commute time $\kappa(i,j) = h(i,j)+h(j,i)$ on $\mathcal{G}(n,p)$. Again, the result comes from using expressions for these quantities in terms of the eigenvectors and eigenvalues of the transition matrix of the simple random walk, these expressions can be found in \cite{lovasz1993random}. L{\"o}we \& Torres then apply results from Erd\H{o}s et. al. \cite{erdHos2012spectral} to bound from above the reciprocal of the spectral gap. L{\"o}we \& Torres require $np = \omega\left((\log n)^{C_0}\right)$ for some $C_0>0$ sufficiently large as this is needed to apply the results in \cite{erdHos2012spectral}. Theorems \ref{exthm}, \ref{exthm2},  \ref{concentration} and \ref{concentrationnew} extend these results to the range \eqref{sparsecon}.

Von Luxburg, Radl  \& Hein \cite[Theorem 5]{von2010hitting} prove bounds on the difference of $h(i,j)/2|E|$ from $1/{d}(i)$ for non bipartite graphs by the reciprocal of the spectral gap and the minimum degree of $G$. They then apply these to various geometric random graphs. These bounds give the same result as Theorem \ref{resconc} \eqref{itm:co4} when applied to $\cg(n,p)$ with $np=\omega(\log n)$, however if $np= \BO{\log n}$ they will only give the hitting times up-to a constant. Theorem \ref{concentrationnew} provides concentration results for $h(i,j)$ recovering the leading constant in the extended range \eqref{sparsecon}.

On a different note, Bollob\'{a}s \& Thomason \cite[Theorem 7.4]{bollobasrandom} showed the threshold for having minimum degree $k(n)$ coincides with the threshold for having at least $k(n)$ vertex-disjoint paths between any two points. Theorem \ref{bolthom} can be thought of as a ``local first neighbourhood relaxation'' of this statement for two vertices as it roughly states that if you want to separate two vertices $x$ and $y$ and your not allowed to use any vertices from either $x$ or $y$'s first neighbourhoods then w.h.p. the next best option take the smaller of $x$ or $y$'s second neighbourhoods as a separator. Broder, Frieze, Suen \& Upfal \cite{BFSU} show that there are edge disjoint paths between any two sets of vertices in $\cg(n,p)$, provided that the sets are not too large and provide a polynomial time algorithm to find them. The restrictions on the sets are very modest however their results do not give bounds on the length of the paths found or exact bounds on their number.

\subsection{Futher Preliminaries}\label{prelim}
We use $X\sim \mathcal{L}$ to denote the random variable $X$ having law $\mathcal{L}$. For random variables $A,B$, we say that $B$ dominates $A$ if $\mathbb{P}\left[A>x\right]\leq \mathbb{P}\left[B>x\right]$ for every $x$ and we use the notation $B \succeq_1 A$, or $A\preceq_1 B$ in this case. If $A\preceq_1 B$ and $A,B\geq 0$ then $\mathbb{E}\left[ A^\alpha\right]\leq \mathbb{E}\left[ B^\alpha\right]$ for any $\alpha \geq 1$. Let $\text{Bin}(n,p)$ denote the binomial distribution over $n$ trials each of probability $p$. Some additional probabilistic notions and lemmas may be found in Appendix \ref{appen}.

Throughout we will be working on a finite simple connected graph $G=(V,E)$ with $|V|=n$ and $|E|=:m$. Let $d(i,j)$ be the graph distance between $i,j \in V$ and define the following \begin{equation}\label{nebhs}
\Gamma_{G,k}(i):= \left\{j \in V: d(i,j)=k \right\},\qquad {d}_{G,k}(i):= \left|\Gamma_k(i)\right|, \qquad B_{G,k}(i) := \bigcup\limits_{h=0}^{k}\Gamma_h(i),  
\end{equation}which are the $k^{th}$ neighbourhood of $i$, size of $k^{th}$ neighbourhood and the ball of radius $k$ centred at $i$ respectively. We drop the $G$ from the subscripts in \eqref{nebhs} when the graph is clear, and the subscript $1$ when referring to first neighbourhoods i.e. $\Gamma(x):=\Gamma_{1}(x)$ and  $d(x):=|\Gamma(x)| $.

The hitting times $h(i,j)$ can be far from symmetric, see the example of the lollipop graph \cite{lovasz1993random}. The commute time $\kappa(i,j)$ is the expected number of steps for a random walk from $i$ to reach $j$ and return back to $i$. The commute time $\kappa(i,j)$ is symmetric and related to hitting times and effective resistances by the commute time formula \cite{chandra1996electrical}
\begin{equation}\label{commute} \kappa(i,j):=h(i,j) + h(j,i) = 2 m\cdot  R(i,j). \end{equation}

\subsubsection*{\er Graphs} The Erd\H{o}s-R\'{e}nyi or Binomial random graph model $\mathcal{G}(n,p)$ is a probability distribution over simple $n$ vertex graphs. Any given $n$ vertex graph $G=(V,E)$ is sampled with probability
\[\pr{\mathcal{G}=G} = p^{|E(G)|}(1-p)^{\binom{n}{2}-|E(G)|}. \]This $\mathbb{P}$ is the product measure over edges of the complete graph $K_n$ where each edge occurs as an i.i.d.\ Bernoulli random variable with probability $0<p:=p(n)<1$. Throughout $\mathbb{E}$ will denote expectation with respect to $\mathbb{P}$. Another feature of \er graphs worth mentioning is that for each $u \in V$ the degree of $u$ is binomially distributed ${d}(u) \sim Bin(n-1,p)$ and the degrees are not independent. This model has received near constant attention in the literature since the original $\cg(n,m)$ model was studied by Erd\H{o}s \& R\'{e}nyi \cite{erdos1959random}. For more information consult one of the many books on random graphs \cite{bollobasrandom,Frbook,JSL}.

Observe that the effective resistance becomes a random variable when the graph is drawn from $\mathcal{G}(n,p)$. Since the effective resistance between two disconnected vertices is infinite we shall need to condition on the event $\mathcal{C}:=\mathcal{C}_n$ that $\mathcal{G}$ is connected. Let $\pru{\mathcal{C}}{\cdot} := \mathbb{P}\left(\cdot \left.\right|\mathcal{C}\right)$ and $\mathbb{E}_{\mathcal{C}}:= \mathbb{E}\left[\cdot \left.\right|\mathcal{C}\right]$ be the expectation with respect to $\mathbb{P}_{\mathcal{C}}$. The following theorem gives a bound on being disconnected above the $np=\log n$ connectivity threshold. 

\begin{theorem}[{\cite[Th.\,9,\,\S VII]{bollobas1998modern}}]
	Let $\mathcal{G}\sim  \mathcal{G}(n,p), \,np = \log n+\omega(n)$ where $\omega(n) \rightarrow  \infty$. Then 
	\begin{equation}\label{eq:probC}\pr{\mathcal{C}^c} \leq 4\cdot e^{-\omega( n)} .\end{equation}
\end{theorem}

\subsubsection*{Basics of Electrical Network Theory}
There is a rich connection between random walks on graphs and electrical networks, consult  either of the books \cite{doyle1984random,lyons2005probability} for a thorough introduction to the subject. Here we only intend to make the definition of $R(i,j)$ given in the introduction rigorous and cover the essentials.

An electrical network, $N:=(G,C)$, is a graph $G$ and an assignment of conductances $C:E(G)\rightarrow \mathbb{R}^+$ to the edges of $G$. Our graph $G$ is undirected and we define $\vec E(G):=\left\{\vec{xy}: xy\in E(G)\right\}$, this is the set of all possible oriented edges for which there is an edge in $G$. For some $i,j \in V(G)$, a flow from $i$ to $j$ is a function $\theta:\vec E(G)\rightarrow \mathbb{R}$ satisfying $\theta(\vec{xy}) = -\theta(\vec{yx})$ for every $xy \in E(G)$ as well as Kirchoff's node law for every vertex apart from $i$ and $j$, i.e.
\begin{equation*}
\sum\limits_{u \in \Gamma_1(v)} \theta(\vec{uv}) = 0 \qquad \text{for each} \,v \in V,\;v \neq i,j.
\end{equation*}
A flow from $i$ and $j$ is called a unit flow if in addition to the above it has strength 1, that is
\[\sum\limits_{u \in \Gamma_1(i)} \theta(\vec{iu}) = 1, \qquad \sum\limits_{u \in \Gamma_1(j)} \theta(\vec{uj}) = 1. \]
For the network $N=(G,C)$ we can then define the effective resistance $R_C(i,j)$ between two vertices $i,j\in V(G)$. First for a flow $\theta$ on $N$ let
\[ \mathcal{E}(\theta) = \sum\limits_{e \in \vec E}\frac{\theta(e)^2}{2C(e)},\] be the energy dissipated by $\theta$. Then for $i,j \in V(G)$, $R_C(i,j)$ can be defined as
\begin{equation}\label{resdef}
R_C(i,j):= \inf\left\{\mathcal{E}(\theta): \theta \text{ is a unit flow from $i$ to $j$} \right\}.                                                                                                                                                                                                                                                                             
\end{equation}
We refer to the flow minimising \eqref{resdef} as the current, thus the effective resistance is the energy dissipated by the current of strength $1$ from $i$ to $j$ in $N=(G,C)$. This current exists and is unique since we are working on a finite graph. Equivalently $R(i,j)$ is the reciprocal of the amount of flow when a unit potential difference is fixed between $i$ and $j$. 

The conductances $C$ define a reversible Markov chain \cite{lyons2005probability}. In this paper we fix $C(e)=1$ for all $e \in E(G)$ as this corresponds to a simple random walk, in this case we write $R(i,j)$ instead of $R_C(i,j)$. This $R(i,j)$ is the effective resistance in Equations \eqref{Kirchoff}, \eqref{eq:hit}, and \eqref{commute}. 

One very useful tool is Rayleigh's monotonicity law \cite[\S\;2.4 ]{lyons2005probability}: If $C,C':E(G)\rightarrow \mathcal{R}^+$ are conductances on the edge set $E(G)$ of a connected graph $G$ and $C(e)\leq C'(e)$ for all $e \in E(G)$ then for all pairs $\{i,j\}\subset V(G)$, we have $R_{C'}(i,j) \leq R_{C}(i,j)$.

\section{Bounds on Effective Resistance} \label{ressec}

The aim of this section is to obtain bounds on $R(u,v)$ for which the main contribution, when applied to graphs with good expansion, comes from the first neighbourhoods of $u$ and $v$.
\subsection{General Bound}\label{genbdd}

Recall that ${d}(x)$ denotes the size of the first neighbourhood of vertex $x\in V(G)$. Jonasson gives the following lower bound on effective resistance. 
\begin{lemma}[{\cite[Lem. 1.4]{jonasson1998cover}}]\label{lowjoh}  For any graph $G = (V,E)$ and $x,y \in V$, $x\neq y$ 
	\[
	R(x,y) \geq \frac{1}{{d}(x) +1} + \frac{1}{{d}(y) +1}.
	\]
\end{lemma}
We seek an upper bound where the dominant term looks roughly like the one in Lemma \ref{lowjoh}. To achieve this we shall define a sub-graph $H(x,y,r,k)$ which will allow us to route flow through the graph efficiently. Recall the definition \eqref{nebhs} of $\Gamma_{H,k}(x)$, $d_{H,k}(x)$ and $B_{H,k}(x)$.  
\begin{definition}\label{hdef}For a graph $G$ and $x,y \in V(G)$, $x\neq y$ define $H:=H(x,y,r,k) \subseteq G$ as follows. Firstly $B_{H,r}(x)$ (resp. $B_{H,r}(y)$) is a tree with a non-empty set of leaves all at distance $r$ from $x$ (resp. $y$). Secondly for each vertex $w \in \Gamma_{H,r}(x)\cup \Gamma_{H,r}(y)$ there exists a set $H_w$ such that 
	\begin{enumerate}[(i)]
		\item The vertex $w$ is connected to every vertex in $H_w$ by paths of lengths at most $k$.
		\item\label{prop2} For every distinct $w,z \in \Gamma_{H,r}(x)\cup \Gamma_{H,r}(y)$, the sets $H_w,H_z$ are disjoint and the paths connecting $w$ to $H_w$ and $z$ to $H_z$ are edge disjoint. 
		\item For every $(w,z) \in \Gamma_{H,r}(x)\times\Gamma_{H,r}(y) $ there is an edge from $H_w$ to $H_z$.   
	\end{enumerate}     
\end{definition} 

To better understand why defining $H(x,y,r,k)$ as above will help us control effective resistance $R(x,y)$ in graphs with good expansion we first recall definition \eqref{resdef} which states that $R(x,y)$ is determined by the sum over all edges of the square of the current through each edge. Heuristically one wishes to keep the max flow through an edge as small as possible, since this max is squared. If a graph has good expansion properties, that is every set has lots of edges leaving the set, then the smallest edge cuts separating $x$ from $y$ will be ``close'' to $x$ and $y$. Our approach is thus to find a flow from $x$ to $y$ which balances the flow as evenly as possible over edges close to $x$ and $y$ where the cuts are smaller and thus the amount of flow is greater. We divide the flow evenly among all descendants of $x$ (from the perspective of rooting $B_{H,r}(x)$ at $x$) and do the same for the flows in the reverse direction in $B_{H,r}(y)$. Then once we are clear of the $r$-neighbourhoods we use the abundance of paths to route the flow from $B_{H,r}(x)$ to $ B_{H,r}(y)$. The challenge is to then make sure all the paths meet up and that we have a valid flow, the structure of the desired sub-graph $H$ outlined in Definition \ref{hdef} will allow us to do this. 
  
We now bound $R(x,y)$ in terms of the neighbourhoods of vertices in $H(x,y,r,k)$. This bound is effective when a sub-graph $H$ can be found who's vertices close to $x$ and $y$ have degrees similar to what they were in the original graph. Let $\mathcal{P}_{i}(x):= \left\{x_0x_1\cdots x_i: x_0=x,\; x_{k-1}x_{k}\in E(H)\right\} $ be the set of all paths of length $i$ from $x$ in $H$, and $\mathbf{1}_{j=x}=1$ if $j=x$ and $0$ otherwise.

\begin{theorem}\label{resbdd} For a graph $G$,  $x,y \in V$ disjoint and $k,r \geq 1$, if $H(x,y,r,k) \subseteq  G$ then
		\begin{equation*}R\left(x,y\right)\leq \frac{1}{d_{H}(x)} + \frac{1}{d_{H}(y)} + \sum_{i=1}^{r}\sum_{pp_1\cdots p_i} \frac{1}{d_{H}(p)^2}\prod_{j=1}^{i-1}\frac{1 + \mathbf{1}_{j=r-1}\cdot (k+1 )  }{(d_{H}(p_j) -1)^2} 
	\end{equation*}where summation $pp_1\cdots p_i\in \mathcal{P}_i(x)\cup\mathcal{P}_i(y)$ is over  all paths of length $i$ from $x$ or $y$ in $H$.   
	
\end{theorem}
\begin{proof}
We will now describe a unit flow $\theta$ from $x$ to $y$ through the network $N=(H,C)$ where $C(e)=1$ for all $e \in E(H)$. This flow will be used to bound $R(x,y)$ from above by \eqref{resdef}.

To begin we assign a flow of $\theta(xx_1) =1/d_{H}(x)$, where $x_1\in \Gamma_H(x)$ is a neighbour of $x$. Likewise let $\theta(yy_1) = - 1/d_{H}(y)$, where $y_1\in \Gamma_H(y)$. Then, for each edge $x_{i-1}x_{i}$ where $x_i \in \Gamma_{H,i}(x) $ and $1\leq i\leq r$ we send the amount of flow entering $x_i$ divided by the number of edges to $x_i \in \Gamma_{H,i}(x) $. So inductively if the unique path from $x=x_0$ to some $x_i$ is $x_0,x_1,\dots, x_i$, then the flow through the (directed) edge  $x_{i-1}x_i$ is 

	\begin{equation*}\theta(x_{i-1}x_i ) = \frac{1}{d_{H}(x)(h(x_1) -1 )\cdots (d_{H}(x_{i-1}) - 1)} = \frac{1}{d_{H}(x)}\prod_{j=1}^{i-1}\frac{1}{d_{H}(x_j) -1},\end{equation*}   
Where we follow the convention that empty products are equal to $1$. 

We do the same with edges in the $r$ neighbourhood of $y$ but the flow is reversed. The total contribution to $\mathcal{E}(\theta)$ from the ball $B_{H,r}(x)$ of radius $r$ around $x$ is then given by
\begin{equation}\label{rnbh}\sum_{i=1}^r\sum_{xx_1\cdots x_i \in \mathcal{P}_i(x)} \frac{1}{d_{H}(x)^2}\prod_{j=1}^{i-1}\frac{1}{(d_{H}(x_j) -1)^2} 
 \end{equation}
and likewise for the contribution to $\mathcal{E}(\theta)$ from the edges in $B_{H,r}(y)$. 

We now describe the flow across an edge from $H_w$ to $H_z$ where $(w,z )\in \Gamma_{H,r}(x)\times\Gamma_{H,r}(y)$. Indeed for each such edge $e_{w,z}$ we assign a flow 
\[\theta(e_{w,z}) = \theta(x_{r-1}w)\cdot \theta(zy_{r-1}) = \frac{1}{d_{H}(x)\cdots (d_{H}(x_{r-1})-1 ) }\cdot \frac{1}{d_{H}(y)\cdots (d_{H}(y_{r-1})-1 ) },   \]
where $xx_1\cdots x_{r-1}w \in \mathcal{P}_{r}(x)$ (resp. $yy_1\cdots y_{r-1}z \in \mathcal{P}_{r}(y)$) is the unique path of length $r$ from $x$ (resp. $y$) to $w$ (resp. $z$) in $H$. The reason for assigning this flow is that if we sum over all the vertices in $\Gamma_{H,r}(y)$ we obtain
\begin{equation}\label{flowok}\sum_{z\in \Gamma_{H,r}(y)}\theta(e_{w,z}) = \theta(x_{r-1}w)\sum_{\mathcal{P}_r(y)}\frac{1}{d_{H}(y)\cdots (h(y_{r-1})-1 ) } =\theta(x_{r-1}w),  \end{equation} which is precisely the flow leaving through $w \in \Gamma_{H,r}(x) $. Thus the contribution to $\mathcal{E}(\theta)$ by the flow through these edges is precisely
\begin{equation}\label{flowcross}
\sum_{xx_1\cdots x_r \in \mathcal{P}_r(x)} \frac{1}{d_{H}(x)^2}\prod_{j=1}^{r-1}\frac{1}{(d_{H}(x_j) -1)^2} = \sum_{yy_1\cdots y_r \in \mathcal{P}_r(y)} \frac{1}{d_{H}(y)^2}\prod_{j=1}^{r-1}\frac{1}{(d_{H}(y_k) -1)^2}
\end{equation}

We are not concerned with how flow is rooted from $w$ to the relevant vertices of $H_w$ but note that since there is a path from $w$ to each vertex in $H_w$ with an edge to some $H_z$ where $z\in \Gamma_{H,r}(y) $ constructing a flow is possible. We now bound the contribution from these paths. 

\begin{cla}
	The contribution to the $\mathcal{E}(\theta)$ by the flow through the paths from $w\in \Gamma_{H,r}(x)$ to $H_w$ is at most $k\cdot \theta(x_{r-1}w)^2$, where $xx_1\cdots x_{r-1}w\in \mathcal{P}_r(x)$. The analogous bound holds for $z\in \Gamma_{H,r}(y) $.	
\end{cla}

\begin{poc}We can assume that all paths to vertices in $H_w$ have length $k$ otherwise we can subdivide edges on these paths, only increasing total resistance by Rayleigh's monotonicity law. Consider the set $S_i$ of edges in the union of all paths to $H_w$ with furthest endpoint from $w$ at distance $1\leq \ell \leq k$ from $w$, this edge set separates $w$ from $H_w$. Recalling property \eqref{prop2} of Definition \ref{hdef}, the only flow through these edges is that from $w$ to $H_w$. Thus the combined flow through $S_i$ is $\theta(x_{r-1}w)$ since this is the amount of flow entering at $w$ and leaving $H_w$, as shown by \eqref{flowok}. Thus since the contribution to $\mathcal{E}(\theta)$ by the edges of $S_i$ is the sum of the squares of the flows through each edge of $S_i$ we see that this cannot exceed $\theta(x_{r-1}w)^2 $ by convexity. The result follows by summing the contributions from the $k$ such edge sets $S_i$.	
\end{poc} 

Thus the contribution to $\mathcal{E}(\theta)$ from all edges in these paths is at most 
\begin{equation}
\label{Flowpaths}
\sum_{xx_1\cdots x_r \in \mathcal{P}_r(x)} \frac{k}{d_{H}(x)^2}\prod_{j=1}^{r-1}\frac{1}{(d_{H}(x_j) -1)^2} + \sum_{yy_1\cdots y_r \in \mathcal{P}_r(y)} \frac{k}{d_{H}(y)^2}\prod_{j=1}^{r-1}\frac{1}{(d_{H}(y_j) -1)^2}.
\end{equation}The result follows by summing the contributions \eqref{rnbh}, \eqref{flowcross} and \eqref{Flowpaths} to $\mathcal{E}(\theta)$.\end{proof}

\subsection{Application to $\mathcal{G}(n,p)$}\label{apppg}To apply Theorem \ref{resbdd} to $\cg(n,p)$ one must describe a suitable $H(x,y,r,k)$, this is achieved using the modified breadth-first search (MBFS) algorithm. The inputs to the MBFS algorithm are a graph $G$ and a subset $I_0=\{u,v\}\subseteq V(G)$, the outputs are sets $I_i,S_i\subseteq V(G)$  and $E_i \subseteq E(G)$ indexed by the iteration of the algorithm.  The algorithm is similar to one used in \cite[\S 11.5]{alon2008probabilistic} to explore the giant component of an \er graph. However, the MBFS algorithm differs from other variations on breadth-first search algorithms used in the literature  as it starts from two distinct vertices. More importantly it also differs by removing clashes, where a clash is a vertex with more than one parent in the previous generation as exposed by a breadth-first search from two root vertices. In what follows all graphs are on a common labelled vertex set $V:=[n]$.\medskip

\noindent{\bf Modified Breadth-First Search Algorithm, $\mbfs(G,I_0)$:} To begin set $S_0:= V\backslash I_0$, and $I_i=E_i=\emptyset$ for all $i \geq 1$. Then generate the sets $S_i$ and update the sets $ I_{i}$ and $E_i$ for $i\geq 1$ iteratively by the following procedure:
\begin{enumerate}[{\textup{Step}}~1:]
	\item Set $S_i=S_{i-1}$. For each $w\in S_{i}$ check all pairs $\{w,w'\}$ where $w' \in I_{i-1}$ and,  
\end{enumerate} 
	\begin{itemize}
		\item if there exists $w'\in I_{i-1}$ such that  $ww'\in E(G)$ then remove $w$ from $S_{i}$,
		\item if there is a unique $w'\in I_{i-1}$ such that $ww'\in E(G)$ then add $w$ to $I_{i}$ and add $ww'$ to $ E_i$.  
	\end{itemize} 
\begin{enumerate}[{\textup{Step}}~2:]
	\item If $S_{i}\neq \emptyset $ and $ I_{i}\neq \emptyset $ then advance $i$ to $i+1$ and return to \textup{Step} 1. Otherwise end.   
\end{enumerate}

The set $I_i$ contains the ``active'' vertices in the $i^{th}$ iteration and $S_i$ is the set of vertices that have not been used in the first $i$ iterations and $E_i$ is the set $\left\{xy \in E(G): x \in I_{i-1}, y \in I_{i} \right\}$ of edges ``accepted'' by the algorithm. Notice that $S_0 \supseteq S_1 \supseteq S_2 \dots$ and the sets $\{I_i\}_{i\geq 0}$ are all disjoint. A vertex in $S_i$ will not be included in either $I_{i+1}$ or $S_{i+1}$ if it has two or more neighbourhoods in $I_i$, in this instance it is just ignored by the algorithm. If instead those vertices in $S_i$ with edges to more than one vertex in $I_i$ were added to $I_{i+1}$ then this procedure would describe a standard breadth-first search starting from two root vertices. Notice also that in Step 1 the order in which we consider the vertices of $S_i$ and then the edges between $S_i$ and $I_i$ is unimportant. 

For each pair of vertices $I_0\subseteq V$ the $\mbfs$ algorithm provides a filtration \begin{equation}\label{eq:filt} 
\mathfrak{F}_i:= \mathfrak{F}_i(I_0),\end{equation} where $\mathfrak{F}_0 \subseteq \mathfrak{F}_1 \subseteq \cdots$, on the set of labelled graphs on $V$. Roughly speaking $\mathfrak{F}_i(I_0)$ only sees graphs that are the distinguishable by $\mbfs$ run up to step $i\geq 0$ from initial set $I_0$. To make this precise we must first describe an equivalence relation on graphs. Let $u,v \in V$ and $G,F$ be graphs on $V$. We say $G\cong_{k}^{\{u,v \}} F$ if  the same $k$-sequence of sets $\left\{ S_i, I_i, E_i \right\}_{1 \leq i \leq k}$ is output when $\mbfs(G, \{u,v \} )$ and $\mbfs(F, \{u,v \} )$ are run for $k$ iterations. Let $I_0=\{u,v \}\subseteq V$ and define $\mathfrak{F}_i(I_0)$ to be the $\sigma$-algebra where the atoms are the equivalence classes of $\cong_{i}^{\{u,v\}}$.

Let $x\in I_k$ where $I_k$ is produced by running $\mbfs(G,I_0)$ for some given $I_0$. We shall now define $\Gamma_i^*(x)$, the $\mbfs$ neighbourhood of $x$, let $\Gamma_0^*(x):=x$ and for $i\geq 1$ \begin{equation}\label{eq:gamdef} 
\Gamma_i^*(x) := \left\{ y \in I_{k+i} : \begin{array}{c} \text{ there exists } x=x_0, x_1, \dots , x_i=y \text{ where } \\ x_{j-1}x_j \in E_{k+j}\text{ for all }j=1,\dots, i\end{array}\right\}, 
\end{equation} and let ${d}_i^*(x) = |\Gamma_i^* (x)|$. Equivalently, we can also define $\Gamma_i^*(x)$ for $i\geq 1$ inductively 
\[\Gamma_i^*(x):=\left\{ z \in I_{k+i}: \text{ there exists } y \in \Gamma_{i-1}^*(x)\text{ and }yz \in E_i  \right\} .\] 
To try and further clarify \eqref{eq:gamdef} we define the following sets $S_k(x)$ which are the vertices in $S_k$ that will not cause any clashes when the $\Gamma^*$-neighbourhood of $x$ is explored, \begin{equation}\label{eq:WS}
S_{k}(x):= S_{k}\Big\backslash \left(\bigcup \limits_{z \in I_{k},\; z\neq x}\Gamma_1(z)\right).
\end{equation}
We can then also define the neighbourhood $\Gamma_i^*(x)$ inductively as follows
\begin{equation}\label{second*def} \Gamma_{i}^*(x) = \bigcup_{y \in \Gamma_{i-1}^*(x)} \Gamma_1(y) \cap S_{k+i}(y). \end{equation}              

We define the pruned neighbourhood $\Phi(x)$ of $x \in I_1$ by \begin{equation}\label{eq:phi}
\Phi(x) := \Gamma_1^*(x)\backslash \left\{y: {d}^*(y) \leq D  \right\},\qquad \text{and let}\qquad\varphi(x):=\left|\Phi(x) \right| ,                                                   \end{equation}where 
\begin{equation}\label{Ddef}
D:=D(n,p) =\begin{cases}\max\left\{\lceil\frac{50}{c} \rceil, 50\right\} &\text{if }np = (c\pm o(1)) \log (n) \text{ where } c>0\\
0 &\text{if }np = \omega \left(\log (n) \right).
\end{cases}                                                                                      \end{equation} This is the $\mbfs$ neighbourhood of $x$ with all the neighbours who have less than $D$ ``$\mbfs$-children'' removed. The choice of $D$ is related to concentration for binomial random variables and our choice for this will be apparent during the proof of Lemma \ref{neighlow}. One may think of $D$ as an atypically small value for the degree of a vertex. 

Define the pruned neighbourhoods $\Psi_1(w)$ of $w\in I_0$ by
\begin{equation}\label{eq:psi}\Psi_1(w) := \Gamma_1^*(w)\backslash \left\{y: \Phi(y) =\emptyset  \right\},\qquad \text{and let}\qquad \psi(w):=\left|\Psi_1(w) \right|. \end{equation} Set $\Psi_0(u)=\{u\}$ and let the pruned second neighbourhood $\Psi_2(w)$ of $w \in I_0$ be given by
\begin{equation}\label{eq:psi2}
\Psi_2(w) := \bigcup\limits_{x\in \Psi_1(w)}\Phi(x) = \bigcup\limits_{x\in \Gamma_1^*(w)}\Phi(x), \qquad \text{also let}\qquad \psi_2(w):=\left|\Psi_2(w) \right|. 
\end{equation} For $\mbfs(G,\{u,v\})$ define $\Psi_i$, the pruned version of $I_i$ for $i=0,1,2$, by
\begin{equation*}
\Psi_i:= \Psi_i(u)\cup  \Psi_i(v),\quad i=0,1,2. 
\end{equation*}
We prune the first neighbourhoods of vertices $x \in I_1$ to obtain $\Phi(x)$ so that later on when we consider the trees induced by the union up to $i$ of the $\Gamma^*$-neighbourhoods of $y \in \Phi(x)$ we can get good control over the growth rate of the trees. We prune the first neighbourhoods of vertices $w \in I_0$ as above so that we can send flow from our source vertex $w$ to its pruned neighbourhood $\Psi_1(w)$ without having to worry about it getting stuck in any ``dead ends''.  

Recall \eqref{eq:filt}, the definition of the filtration $\mathfrak{F}_k(I_0)$. Observe that if $x \in I_k$ then $\Gamma_1^*(x)$ is $\mathfrak{F}_{k+1}$ measurable. It is worth noting however that if $y\in I_1$ then $\Phi(y)$ is $\mathfrak{F}_3$ measurable and not necessary $\mathfrak{F}_2$ measurable since $\Phi(y)$ is determined by vertices at distances $2$ and $3$ from $I_0$. A consequence of this is that for $w \in I_0$, $\Psi_1(w),\Psi_2(w)$ are both $\mathfrak{F}_3$ measurable as they are both determined by the $\Phi$-neighbourhoods of vertices in $\Gamma_1^*(w)$.

\begin{definition}[The Set $A_{u,v}^{n,k}$]\label{A} For integers $n,k\geq 0$ let $A_{u,v}^{n,k}$ be the set of $n$-vertex graphs on $V$, where $u,v \in V$, such that for every pair $ (x,y) \in \Psi_2(u)\times \Psi_2(v)$ the neighbourhoods $\Gamma_k^*(x)$ and $\Gamma_k^*(y)$ are non-empty and there is at least one edge $ij \in E(G)$ where $i \in \Gamma_k^*(x)$, $j \in \Gamma_k^*(y)$. 
\end{definition}

We relate the structure of $G\in A_{u,v}^{n,k}$ to Theorem \ref{resbdd} to give a bound on $R(u,v)$.

\begin{corollary}[of Theorem \ref{resbdd}]\label{Gnpbdd} Run $\mbfs(G,\{i,j\})$ and suppose $G\in A_{i,j}^{n,k}$. Then
	\begin{align*}
	R\left(i,j\right) &\leq \frac{1}{\psi(i)} + \frac{1}{\psi(j)} +\sum_{a \in \Psi_1(i)}\frac{k+2}{\psi(i)^2\cdot \varphi(a)} + \sum_{b \in \Psi_1(j)}\frac{k+2}{\psi(j)^2\cdot \varphi(b)}\\
& \leq \frac{1}{\psi(i)}\left(1+ \sup_{x\in \Psi_1(i)}\frac{k+2}{\varphi(x)}\right) + \frac{1}{\psi(j)} \left( 1+ \sup_{y\in \Psi_1(j)}\frac{k+2}{\varphi(y)}\right). \end{align*}	
\end{corollary}
\begin{proof}
	If we can find a some suitable sub-graph $H(i,j,2,k)$, from Definition \ref{hdef}, encoded by the property $A_{i,j}^{n,k}$ then the result follows by Theorem \ref{resbdd}. The only thing that can go wrong with $G\in A_{i,j}^{n,k}$ according to Definition \ref{A} is that if one of the neighbourhoods $\Psi_1(i),\Psi_1(j),\Psi_2(i)$ or $\Psi_2(j) $ are empty. In this case one of the terms $\psi(i)$,$\psi(i)$, $ \varphi(a) $ or $ \varphi(b) $ on the RHS of the inequality will be $0$, we define $1/0$ to be infinity and so the inequality holds vacuously. 
\end{proof}
We encode Definition \ref{A} as the following event for $\cg \sim \cg(n,p)$,
\begin{equation}\label{calA}\mathcal{A}_{u,v}  := \{\text{ exists }k \leq \log n /(2\log np) +2\text{ such that }\mathcal{G} \in  A_{u,v}^{n,k}\},\end{equation} 
and say that $\cg$ satisfies the strong path property if this holds. 
\begin{figure}       
	\begin{center}
	
	\definecolor{ffqqqq}{rgb}{1.0,0.0,0.0}
	\definecolor{ffxfqq}{rgb}{1.0,0.5,0.0}
	\definecolor{qqqqff}{rgb}{0.0,0.0,1.0}
	\begin{tikzpicture}[line cap=round,line join=round,>=triangle 45,x=0.87cm,y=0.87cm]
	\clip(-1,-0.85) rectangle (14.1,6.6);
	\draw [line width=1.32pt,color=qqqqff] (-0.8533410836607904,2.8430038215593485)-- (0.8985076927995811,4.078060956872299);
	\draw [line width=1.32pt,color=qqqqff] (0.9698009171060231,2.794782919356338)-- (-0.8533410836607904,2.8430038215593485);
	\draw [line width=1.32pt,color=qqqqff] (-0.8533410836607904,2.8430038215593485)-- (0.7559212441866971,1.7253845547597038);
	\draw [line width=1.32pt,color=ffxfqq] (0.8985076927995811,4.078060956872299)-- (2.986658916339206,5.263003821559348);
	\draw [line width=1.32pt,color=ffxfqq] (3.0666589163392053,4.563003821559349)-- (0.8985076927995811,4.078060956872299);
	\draw [line width=1.32pt,color=ffxfqq] (0.8985076927995811,4.078060956872299)-- (3.0373044219928347,3.8285346717997513);
	\draw [line width=1.32pt,color=ffxfqq] (3.0373044219928347,3.222542265194992)-- (0.9698009171060231,2.794782919356338);
	\draw [line width=1.32pt,color=ffxfqq] (3.0666589163392053,1.5830038215593483)-- (0.9698009171060231,2.794782919356338);
	\draw [line width=1.32pt,color=ffxfqq] (3.026658916339206,0.9630038215593484)-- (0.7559212441866971,1.7253845547597038);
	\draw [line width=1.32pt,color=ffxfqq] (0.7559212441866971,1.7253845547597038)-- (2.9303645855331726,0.2638734564776367);
	\draw [line width=1.32pt,color=ffxfqq] (3.0373044219928347,2.4383167978241267)-- (0.9698009171060231,2.794782919356338);
	\draw [line width=1.32pt,color=qqqqff] (13.798997748790955,2.6431244389904722)-- (11.69830000343381,3.8284927765482064);
	\draw [line width=1.32pt,color=qqqqff] (13.798997748790955,2.6431244389904722)-- (11.69830000343381,2.6165079633386874);
	\draw [line width=1.32pt,color=qqqqff] (13.798997748790955,2.6431244389904722)-- (11.840886452046693,1.5114629865888318);
	\draw [line width=1.32pt,color=ffxfqq] (9.936563473392043,0.326458015416191)-- (11.840886452046693,1.5114629865888318);
	\draw [line width=1.32pt,color=ffxfqq] (9.936563473392043,1.006458015416191)-- (11.840886452046693,1.5114629865888318);
	\draw [line width=1.32pt,color=ffxfqq] (11.69830000343381,2.6165079633386874)-- (9.9677705777903,2.1863154353287437);
	\draw [line width=1.32pt,color=ffxfqq] (11.69830000343381,2.6165079633386874)-- (9.946549396998664,2.907835582244414);
	\draw [line width=1.32pt,color=ffxfqq] (11.69830000343381,3.8284927765482064)-- (9.945837602255024,4.082762086017134);
	\draw (-1.1,3.4) node[anchor=north west] {$u$};
	\draw [rotate around={86.03704580186712:(0.8852198144366249,2.839623242954415)},line width=1.02pt,dash pattern=on .67pt off .67pt] (0.8852198144366249,2.839623242954415) ellipse (1.4231398809853604cm and 0.352964157162801cm);
	\draw [rotate around={-88.57236332543971:(11.727834791054017,2.689022587744238)},line width=1.02pt,dash pattern=on .67pt off .67pt] (11.727834791054017,2.689022587744238) ellipse (1.5576236092646378cm and 0.29040178988699744cm);
	\draw (13.6,3.2) node[anchor=north west] {$v$};
	\draw [line width=1.32pt,color=ffxfqq] (9.945837602255024,4.751753379538426)-- (11.69830000343381,3.8284927765482064);
	\draw [line width=1.02pt,dash pattern=on 2pt off 2pt,color=ffxfqq] (5.592970580188116,5.468064142835809) circle (0.24005471158687333cm);
	\draw [line width=1.02pt,dash pattern=on 2pt off 2pt,color=ffxfqq] (5.611688332053448,4.681918564491881) circle (0.24005471158687333cm);
	\draw [line width=1.02pt,dash pattern=on 2pt off 2pt,color=ffxfqq] (5.630406083918779,3.951926241743948) circle (0.24005471158687192cm);
	\draw [line width=1.02pt,dash pattern=on 2pt off 2pt,color=ffxfqq] (5.630406083918779,3.165780663400019) circle (0.24005471158687333cm);
	\draw [line width=1.02pt,dash pattern=on 2pt off 2pt,color=ffxfqq] (5.630406083918779,2.4170705887867543) circle (0.24005471158687333cm);
	\draw [line width=1.02pt,dash pattern=on 2pt off 2pt,color=ffxfqq] (5.611688332053448,1.630925010442826) circle (0.24005471158687333cm);
	\draw [line width=1.02pt,dash pattern=on 2pt off 2pt,color=ffxfqq] (5.592970580188116,0.863497183964229) circle (0.24005471158687341cm);
	\draw [line width=1.02pt,dash pattern=on 2pt off 2pt,color=ffxfqq] (5.574252828322784,-0.01623715370635753) circle (0.2400547115868732cm);
	\draw [line width=1.02pt,dash pattern=on 2pt off 2pt,color=ffxfqq] (7.1839794887412936,5.449346390970478) circle (0.24005471158687333cm);
	\draw [line width=1.02pt,dash pattern=on 2pt off 2pt,color=ffxfqq] (7.165261736875962,4.513458797703897) circle (0.24005471158687333cm);
	\draw [line width=1.02pt,dash pattern=on 2pt off 2pt,color=ffxfqq] (7.25885049620262,3.371675933918667) circle (0.240054711586874cm);
	\draw [line width=1.02pt,dash pattern=on 2pt off 2pt,color=ffxfqq] (7.2775682480679516,2.3983528369214224) circle (0.24005471158687333cm);
	\draw [line width=1.02pt,dash pattern=on 2pt off 2pt,color=ffxfqq] (7.315003751798614,1.1629812138095352) circle (0.2400547115868737cm);
	\draw [line width=1.02pt,dash pattern=on 2pt off 2pt,color=ffxfqq] (7.315003751798614,0.3768356354656069) circle (0.2400547115868733cm);
	\draw (5.574252828322784,5.505499646566473)-- (7.1839794887412936,5.449346390970478);
	\draw (5.574252828322784,5.505499646566473)-- (7.165261736875962,4.569612053299891);
	\draw (5.574252828322784,4.513458797703897)-- (7.240132744337288,3.3529581820533356);
	\draw (5.798865850706763,3.165780663400019)-- (7.09887593910803,4.441876131589434);
	\draw (5.720429241945538,4.804251294245557)-- (7.09887593910803,5.280312782440855);
	\draw (5.819904776792315,4.5981948292058)-- (6.9638734275302605,4.534246271090015);
	\draw (5.805693986099919,3.9231822713169455)-- (7.283616218109189,4.392138364166045);
	\draw (5.684902265214546,4.079500968933312)-- (6.992295008915054,5.543212410250198);
	\draw (5.677796869868347,3.859233713201159)-- (7.084665148415633,3.3121182715438766);
	\draw (5.585426730367767,3.3050128761976785)-- (7.198351473954808,3.525280131929831);
	\draw (7.312037799493982,2.3386791091146857)-- (5.727534637291736,3.0563240390807316);
	\draw (5.6138483117525615,2.2320981789217087)-- (7.240132744337288,3.3529581820533356);
	\draw (5.656480683829751,2.6371057136550218)-- (7.09887593910803,2.3173629230760904);
	\draw (7.2338784506858,2.566051760193037)-- (5.599637521060164,3.1557995739275104);
	\draw (7.447040311071753,2.5021032020772505)-- (5.571215939675371,4.022657806163724);
	\draw (7.297827008801586,3.148694178581312)-- (5.571215939675371,5.351366735902839);
	\draw (5.443318823443799,5.379788317287633)-- (7.2338784506858,2.566051760193037);
	\draw (7.3973025436483635,2.22499278357551)-- (5.6493752884835535,1.741825900034014);
	\draw (5.5783213350215695,2.4665762253462584)-- (7.340459380878777,4.633721805936793);
	\draw (5.500161986213387,2.296046737037495)-- (7.354670171571173,5.415315294018625);
	\draw (7.251699249727316,1.2989406029783652)-- (5.6939031849732356,1.5349703097592884);
	\draw (7.421640638609579,1.062910896197442)-- (5.750550314600656,2.4507655720692707);
	\draw (7.166728555286183,1.0440285196549681)-- (5.542844172633446,2.9700309269873015);
	\draw (7.421640638609579,1.062910896197442)-- (5.514520607819735,3.8669438127548097);
	\draw (7.478287768236999,1.2706170381646544)-- (5.580608925718393,4.669444815809949);
	\draw (7.251699249727316,1.2989406029783652)-- (5.574252828322784,5.505499646566473);
	\draw (5.627814867074577,0.8835283190439404)-- (7.166728555286183,1.0440285196549681);
	\draw (5.627814867074577,0.8835283190439404)-- (7.185610931828657,2.2052946770171102);
	\draw (5.505079419548498,0.9590578252138359)-- (7.297827008801586,3.148694178581312);
	\draw (5.457873478192314,0.7607928715178603)-- (7.14784617874371,4.716650757166133);
	\draw (5.712785561515709,0.7324693067041496)-- (6.992295008915054,5.543212410250198);
	\draw (7.261917025269307,0.40330533543717534)-- (5.457873478192314,0.7607928715178603);
	\draw (7.261917025269307,0.40330533543717534)-- (5.6939031849732356,1.5349703097592884);
	\draw (7.261917025269307,0.40330533543717534)-- (5.5783213350215695,2.4665762253462584);
	\draw (7.261917025269307,0.40330533543717534)-- (5.599637521060164,3.1557995739275104);
	\draw (7.261917025269307,0.40330533543717534)-- (5.514520607819735,3.8669438127548097);
	\draw (7.455107535733684,0.27911143585293197)-- (5.580608925718393,4.669444815809949);
	\draw (7.261917025269307,0.40330533543717534)-- (5.443318823443799,5.379788317287633);
	\draw (5.716392941554292,-0.12106890725185222)-- (7.455107535733684,0.27911143585293197);
	\draw (5.716392941554292,-0.12106890725185222)-- (7.478287768236999,1.2706170381646544);
	\draw (5.509403108913889,0.09972024756458044)-- (7.185610931828657,2.2052946770171102);
	\draw (5.509403108913889,-0.16246687377993335)-- (7.198351473954808,3.525280131929831);
	\draw (5.509403108913889,-0.16246687377993335)-- (7.14784617874371,4.716650757166133);
	\draw (5.509403108913889,-0.16246687377993335)-- (6.992295008915054,5.543212410250198);
	\draw (7.1839794887412936,5.449346390970478)-- (5.489308121022042,1.6110193015580003);
	\draw (5.505079419548498,0.9590578252138359)-- (7.1839794887412936,5.449346390970478);
	\draw [line width=1.32pt,dash pattern=on .67pt off 2pt,color=ffxfqq] (7.158984061488307,5.748371919953735)-- (9.945837602255024,4.751753379538426);
	\draw [line width=1.32pt,dash pattern=on .67pt off 2pt,color=ffxfqq] (9.945837602255024,4.751753379538426)-- (7.196661619917357,5.149546120930404);
	\draw [line width=1.32pt,dash pattern=on .67pt off 2pt,color=ffxfqq] (7.1868339716616,4.812750757257824)-- (9.945837602255024,4.082762086017134);
	\draw [line width=1.32pt,dash pattern=on .67pt off 2pt,color=ffxfqq] (9.945837602255024,4.082762086017134)-- (7.160080644956087,4.213435140880315);
	\draw [line width=1.32pt,dash pattern=on .67pt off 2pt,color=ffxfqq] (7.256403001942905,3.671734341798076)-- (9.946549396998664,2.907835582244414);
	\draw [line width=1.32pt,dash pattern=on .67pt off 2pt,color=ffxfqq] (9.946549396998664,2.907835582244414)-- (7.2704990195662385,3.07183372496415);
	\draw [line width=1.32pt,dash pattern=on .67pt off 2pt,color=ffxfqq] (7.2783633868022335,2.6984201729006485)-- (9.9677705777903,2.1863154353287437);
	\draw [line width=1.32pt,dash pattern=on .67pt off 2pt,color=ffxfqq] (9.9677705777903,2.1863154353287437)-- (7.307010699570038,2.0997323747678323);
	\draw [line width=1.32pt,dash pattern=on .67pt off 2pt,color=ffxfqq] (7.3095917497186855,1.4630007941718235)-- (9.936563473392043,1.006458015416191);
	\draw [line width=1.32pt,dash pattern=on .67pt off 2pt,color=ffxfqq] (9.936563473392043,1.006458015416191)-- (7.324718797801302,0.8630701332399064);
	\draw [line width=1.32pt,dash pattern=on .67pt off 2pt,color=ffxfqq] (7.340292207811563,0.6758365256979899)-- (9.936563473392043,0.326458015416191);
	\draw [line width=1.32pt,dash pattern=on .67pt off 2pt,color=ffxfqq] (9.936563473392043,0.326458015416191)-- (7.301920762996229,0.0770525909629648);
	\draw [line width=1.32pt,dash pattern=on .67pt off 2pt,color=ffxfqq] (2.986658916339206,5.263003821559348)-- (5.554818205339592,5.765697196545732);
	\draw [line width=1.32pt,dash pattern=on .67pt off 2pt,color=ffxfqq] (5.577853432529106,5.168376788725005)-- (2.986658916339206,5.263003821559348);
	\draw [line width=1.32pt,dash pattern=on .67pt off 2pt,color=ffxfqq] (3.0666589163392053,4.563003821559349)-- (5.589004509959853,4.981128328006556);
	\draw [line width=1.32pt,dash pattern=on .67pt off 2pt,color=ffxfqq] (5.573885607346851,4.3842408992698605)-- (3.0666589163392053,4.563003821559349);
	\draw [line width=1.32pt,dash pattern=on .67pt off 2pt,color=ffxfqq] (3.0373044219928347,3.8285346717997513)-- (5.589424373002746,4.249182928359523);
	\draw [line width=1.32pt,dash pattern=on .67pt off 2pt,color=ffxfqq] (5.6229845958248115,3.651949642853725)-- (3.0373044219928347,3.8285346717997513);
	\draw [line width=1.32pt,dash pattern=on .67pt off 2pt,color=ffxfqq] (3.0373044219928347,3.222542265194992)-- (5.624626263362888,3.465793383199727);
	\draw [line width=1.32pt,dash pattern=on .67pt off 2pt,color=ffxfqq] (5.6307988986364474,2.865712531030251)-- (3.0373044219928347,3.222542265194992);
	\draw [line width=1.32pt,dash pattern=on .67pt off 2pt,color=ffxfqq] (3.0373044219928347,2.4383167978241267)-- (5.595555605093814,2.7151083087031216);
	\draw [line width=1.32pt,dash pattern=on .67pt off 2pt,color=ffxfqq] (3.0373044219928347,2.4383167978241267)-- (5.555280296558267,2.126558705310812);
	\draw [line width=1.32pt,dash pattern=on .67pt off 2pt,color=ffxfqq] (5.587222016741779,1.9299942963655807)-- (3.0666589163392053,1.5830038215593483);
	\draw [line width=1.32pt,dash pattern=on .67pt off 2pt,color=ffxfqq] (3.0666589163392053,1.5830038215593483)-- (5.574677818741494,1.3331478113968578);
	\draw [line width=1.32pt,dash pattern=on .67pt off 2pt,color=ffxfqq] (5.537030778625074,1.1583052180616125)-- (3.026658916339206,0.9630038215593484);
	\draw [line width=1.32pt,dash pattern=on .67pt off 2pt,color=ffxfqq] (3.026658916339206,0.9630038215593484)-- (5.562967426133262,0.5649325358324766);
	\draw [line width=1.32pt,dash pattern=on .67pt off 2pt,color=ffxfqq] (5.492309362162306,0.27242581018212425)-- (2.9303645855331726,0.2638734564776367);
	\draw [line width=1.32pt,dash pattern=on .67pt off 2pt,color=ffxfqq] (2.9303645855331726,0.2638734564776367)-- (5.526507180965266,-0.31248264569853607);
	\draw [line width=1.32pt,color=ffqqqq] (10.77417544787252,1.9104400850374579)-- (11.69830000343381,2.6165079633386874);
	\draw [line width=1.32pt,color=ffqqqq] (11.840886452046693,1.5114629865888318)-- (10.77417544787252,1.9104400850374579);
	\draw [rotate around={-89.8237058501583:(2.9782240425901865,2.702008524587666)},line width=1.02pt,dash pattern=on 2pt off 2pt] (2.9782240425901865,2.702008524587666) ellipse (2.391422469449911cm and 0.49532665169487367cm);
	\draw [line width=1.02pt,color=ffqqqq] (2,3.3534803788687984)-- (1.02,2.85);
	\draw (0.4,5.3) node[anchor=north west] {$\Psi_1(u)$};
	\draw (2.4,6.2) node[anchor=north west] {$\Psi_2(u)$};
	\draw (9.5,6.1) node[anchor=north west] {$\Psi_2(v)$};
	\draw (11.1,5.3) node[anchor=north west] {$\Psi_1(v)$};
	\draw (9.35,4) node[anchor=north west] {$\Phi_1(t)$};
	\draw (11.51,3.8) node[anchor=north west] {$t$};
	\draw (10.5,1.9) node[anchor=north west] {$w$};
	\draw (1.75,3.85) node[anchor=north west] {$z$};
	\draw [line width=1.02pt,dash pattern=on 2pt off 2pt] (7.1839794887412936,5.449346390970478) circle (0.410012648404499cm);
	\draw (6.7,6.6) node[anchor=north west] {$\Gamma_k^*(x)$};
	\draw (9.7,4.7) node[anchor=north west] {$x$};
	\draw [rotate around={89.72845805431513:(9.940170912279008,2.6549411046209435)},line width=1.02pt,dash pattern=on 2pt off 2pt] (9.940170912279008,2.6549411046209435) ellipse (2.3323998829242067cm and 0.48768216103044254cm);
	\draw [rotate around={90.0:(9.945837602255024,4.417257732777783)},line width=1.02pt,dash pattern=on 2pt off 2pt] (9.945837602255024,4.417257732777783) ellipse (0.4259014824949661cm and 0.3313369533605295cm);
	\begin{scriptsize}
	\draw [fill=black] (-0.8533410836607904,2.8430038215593485) circle (1.67pt);
	\draw [fill=black] (0.8985076927995811,4.078060956872299) circle (1.67pt);
	\draw [fill=black] (0.9698009171060231,2.794782919356338) circle (1.67pt);
	\draw [fill=black] (0.7559212441866971,1.7253845547597038) circle (1.67pt);
	\draw [fill=black] (2.986658916339206,5.263003821559348) circle (1.67pt);
	\draw [fill=black] (3.0666589163392053,4.563003821559349) circle (1.67pt);
	\draw [fill=black] (3.0373044219928347,3.8285346717997513) circle (1.67pt);
	\draw [fill=black] (3.0373044219928347,3.222542265194992) circle (1.67pt);
	\draw [fill=black] (3.0666589163392053,1.5830038215593483) circle (1.67pt);
	\draw [fill=black] (3.026658916339206,0.9630038215593484) circle (1.67pt);
	\draw [fill=black] (2.9303645855331726,0.2638734564776367) circle (1.67pt);
	\draw [fill=black] (9.945837602255024,4.082762086017134) circle (1.67pt);
	\draw [fill=black] (9.946549396998664,2.907835582244414) circle (1.67pt);
	\draw [fill=black] (9.9677705777903,2.1863154353287437) circle (1.67pt);
	\draw [fill=black] (9.936563473392043,1.006458015416191) circle (1.67pt);
	\draw [fill=black] (9.936563473392043,0.326458015416191) circle (1.67pt);
	\draw [fill=black] (11.69830000343381,3.8284927765482064) circle (1.67pt);
	\draw [fill=black] (11.69830000343381,2.6165079633386874) circle (1.67pt);
	\draw [fill=black] (11.840886452046693,1.5114629865888318) circle (1.67pt);
	\draw [fill=black] (13.798997748790955,2.6431244389904722) circle (1.67pt);
	\draw [fill=black] (3.0373044219928347,2.4383167978241267) circle (1.67pt);
	\draw [fill=black] (9.945837602255024,4.751753379538426) circle (1.67pt);
	\draw [fill=qqqqff] (5.574252828322784,5.505499646566473) circle (1pt);
	\draw [fill=qqqqff] (7.1839794887412936,5.449346390970478) circle (1pt);
	\draw [fill=qqqqff] (7.165261736875962,4.569612053299891) circle (1pt);
	\draw [fill=qqqqff] (5.574252828322784,4.513458797703897) circle (1pt);
	\draw [fill=qqqqff] (7.240132744337288,3.3529581820533356) circle (1pt);
	\draw [fill=qqqqff] (5.798865850706763,3.165780663400019) circle (1pt);
	\draw [fill=qqqqff] (7.09887593910803,4.441876131589434) circle (1pt);
	\draw [fill=qqqqff] (5.720429241945538,4.804251294245557) circle (1pt);
	\draw [fill=qqqqff] (7.09887593910803,5.280312782440855) circle (1pt);
	\draw [fill=qqqqff] (5.819904776792315,4.5981948292058) circle (1pt);
	\draw [fill=qqqqff] (6.9638734275302605,4.534246271090015) circle (1pt);
	\draw [fill=qqqqff] (5.805693986099919,3.9231822713169455) circle (1pt);
	\draw [fill=qqqqff] (7.283616218109189,4.392138364166045) circle (1pt);
	\draw [fill=qqqqff] (5.684902265214546,4.079500968933312) circle (1pt);
	\draw [fill=qqqqff] (6.992295008915054,5.543212410250198) circle (1pt);
	\draw [fill=qqqqff] (5.677796869868347,3.859233713201159) circle (1pt);
	\draw [fill=qqqqff] (7.084665148415633,3.3121182715438766) circle (1pt);
	\draw [fill=qqqqff] (5.585426730367767,3.3050128761976785) circle (1pt);
	\draw [fill=qqqqff] (7.198351473954808,3.525280131929831) circle (1pt);
	\draw [fill=qqqqff] (7.312037799493982,2.3386791091146857) circle (1pt);
	\draw [fill=qqqqff] (5.727534637291736,3.0563240390807316) circle (1pt);
	\draw [fill=qqqqff] (5.6138483117525615,2.2320981789217087) circle (1pt);
	\draw [fill=qqqqff] (5.656480683829751,2.6371057136550218) circle (1pt);
	\draw [fill=qqqqff] (7.09887593910803,2.3173629230760904) circle (1pt);
	\draw [fill=qqqqff] (7.2338784506858,2.566051760193037) circle (1pt);
	\draw [fill=qqqqff] (5.599637521060164,3.1557995739275104) circle (1pt);
	\draw [fill=qqqqff] (7.447040311071753,2.5021032020772505) circle (1pt);
	\draw [fill=qqqqff] (5.571215939675371,4.022657806163724) circle (1pt);
	\draw [fill=qqqqff] (7.297827008801586,3.148694178581312) circle (1pt);
	\draw [fill=qqqqff] (5.571215939675371,5.351366735902839) circle (1pt);
	\draw [fill=qqqqff] (5.443318823443799,5.379788317287633) circle (1pt);
	\draw [fill=qqqqff] (7.3973025436483635,2.22499278357551) circle (1pt);
	\draw [fill=qqqqff] (5.6493752884835535,1.741825900034014) circle (1pt);
	\draw [fill=qqqqff] (5.5783213350215695,2.4665762253462584) circle (1pt);
	\draw [fill=qqqqff] (7.340459380878777,4.633721805936793) circle (1pt);
	\draw [fill=qqqqff] (5.500161986213387,2.296046737037495) circle (1pt);
	\draw [fill=qqqqff] (7.354670171571173,5.415315294018625) circle (1pt);
	\draw [fill=qqqqff] (7.251699249727316,1.2989406029783652) circle (1pt);
	\draw [fill=qqqqff] (5.6939031849732356,1.5349703097592884) circle (1pt);
	\draw [fill=qqqqff] (7.421640638609579,1.062910896197442) circle (1pt);
	\draw [fill=qqqqff] (5.750550314600656,2.4507655720692707) circle (1pt);
	\draw [fill=qqqqff] (7.166728555286183,1.0440285196549681) circle (1pt);
	\draw [fill=qqqqff] (5.542844172633446,2.9700309269873015) circle (1pt);
	\draw [fill=qqqqff] (5.514520607819735,3.8669438127548097) circle (1pt);
	\draw [fill=qqqqff] (7.478287768236999,1.2706170381646544) circle (1pt);
	\draw [fill=qqqqff] (5.580608925718393,4.669444815809949) circle (1pt);
	\draw [fill=qqqqff] (5.627814867074577,0.8835283190439404) circle (1pt);
	\draw [fill=qqqqff] (7.185610931828657,2.2052946770171102) circle (1pt);
	\draw [fill=qqqqff] (5.505079419548498,0.9590578252138359) circle (1pt);
	\draw [fill=qqqqff] (5.457873478192314,0.7607928715178603) circle (1pt);
	\draw [fill=qqqqff] (7.14784617874371,4.716650757166133) circle (1pt);
	\draw [fill=qqqqff] (5.712785561515709,0.7324693067041496) circle (1pt);
	\draw [fill=qqqqff] (7.261917025269307,0.40330533543717534) circle (1pt);
	\draw [fill=qqqqff] (7.455107535733684,0.27911143585293197) circle (1pt);
	\draw [fill=qqqqff] (5.716392941554292,-0.12106890725185222) circle (1pt);
	\draw [fill=qqqqff] (5.509403108913889,0.09972024756458044) circle (1pt);
	\draw [fill=qqqqff] (5.509403108913889,-0.16246687377993335) circle (1pt);
	\draw [fill=qqqqff] (5.489308121022042,1.6110193015580003) circle (1pt);
	\draw [fill=ffqqqq] (10.77417544787252,1.9104400850374579) circle (1.67pt);
	\draw [fill=ffqqqq] (2,3.3534803788687984) circle (1.67pt);
	\end{scriptsize}
	\end{tikzpicture}
		
\end{center}
	\caption{Illustration of $G\in A_{u,v}^{n,k}$, see Definition \ref{A}. Note: vertex $z$ is not in $\Psi_2(u)$ as it is connected to less than $D$ vertices in $I_3$ and $w$ is not in $ I_2$ as it has more than one parent in $I_1$.}\label{Apic}
\end{figure}

Notice that in Definition \ref{A} either $ \Psi_1(u)$ or $\Psi_1(v)$ may be empty, thus we also define the following sets $B_{w}^{u,v}$ for $w \in \{u,v \}$ using the output of $\mbfs(G,\{u,v\})$
\begin{equation*}
B_{w}^{u,v} :=\left\{G:  \Psi_1(w)\neq \emptyset \right\},\quad \text{and let}\quad B_{u,v} :=B_{u}^{u,v} \cap  B_{v}^{u,v}.
\end{equation*}
Similarly to how we defined $\ca_{u,v}$ define the events 
\begin{equation}\label{calB}\mathcal{B}_{w}^{u,v}: = \left\{\mathcal{G} \in B_{w}^{u,v} \right\}, \qquad \mathcal{B}_{u,v} = \mathcal{B}_{u}^{u,v} \cap \mathcal{B}_{v}^{u,v}.  \end{equation}

\section{The Strong Path Property for $\mathcal{G}(n,p)$}\label{strgnp}
In this section we prove some results needed to successfully apply Corollary \ref{Gnpbdd} to $\cg(n,p)$ in the sparsely connected range \eqref{sparsecon}, in particular we show that $\ca_{u,v}$ holds w.h.p.. To apply the bound on effective resistance in terms of the reciprocals of $\psi$ and $\varphi$ we couple them to $d$ and $d^*$. Lemmas \ref{dom}, \ref{betterprob} and \ref{sdom} will help us achieve this.

\begin{lemma}\label{dom}
	Let $\mathcal{G}\sim  \mathcal{G}(n,p)$, $I_0:=\{u,v \}\subset V$ and $i,k\geq 0$. Run $\mbfs(\mathcal{G},I_0)$. \begin{enumerate}[(i)]
		\item \label{itm:gam4} Then $|S_1|\sim Bin \left(n-2,(1-p)^2\right)$ and $|I_1|\sim Bin \left(n-2,2p(1-p)\right)$.
		\item \label{itm:gam1} Conditioning on $\left\{ x \in I_k\right\}$ and $\left|S_{k}(x)\right|$, then\[{d}^*(x) \sim  Bin(\left|S_{k}(x)\right|,p).\]
		\item \label{itm:gam2}Conditioning on $\left\{ x \in I_k\right\},|S_{k+i}|,|I_{k+i}|$ and ${d}_i^*(x)$, then\[{d}_{i+1}^*(x) \sim  Bin\left(|S_{k+i}|,{d}_{i}^*(x)\cdot p \cdot (1-p)^{|I_{k+i}|-1}\right).\]  
		\end{enumerate}          
	
\end{lemma}

\begin{proof} \textit{Item \eqref{itm:gam4}}: a vertex in $S_0$ is in $S_1$ if it is not connected to either vertex in $I_0$. This happens independently w.p.\ $(1-p)^2$ for each of the $n-2$ vertices in $S_0$ thus $|S_1|\sim Bin \left(n-2,(1-p)^2\right).$ 
	
	A vertex in $S_0$ is in $I_1$ if it is connected to exactly one vertex in $I_0$. This happens independently with probability $2p(1-p)$ for each of the $n-2$ vertices in $S_0$ thus $|I_1|\sim Bin \left(n-2,2p(1-p)\right).$\medskip 
	
	\textit{Item \eqref{itm:gam1}}: recall the definitions of $\Gamma_1^*(x)$ and $S_k(x)$ for $x \in I_k$, given by \eqref{eq:gamdef} and \eqref{eq:WS} respectively. Observe the following relation:
	\[
	\Gamma_1^*(x) =\left(\Gamma_1(x)\cap S_k \right) \backslash \bigcup_{y \in I_{k},\;y\neq x}\Gamma_1(y) = \Gamma_1(x)\cap S_{k}(x).
	\] Since we completely remove the vertices if they clash, and the edges of $\mathcal{G}$ are independent, the order MBFS explores the neighbourhoods of each $y \in I_k$ is unimportant. Assume that we have explored the neighbourhood of every $y \in I_k$ with $y\neq x$. We then know which vertices in the neutral set $S_k$ will not clash if included in $\Gamma_1(x)$ and these are the vertices in $S_{k}(x)$. Since edges occur independently with probability $p$, conditioning on $|S_{k}(x)|$ yields $ {d}^*(x) \sim  Bin(\left|S_{k}(x)\right|,p).$ \medskip 
	
	\textit{Item \eqref{itm:gam2}}: for a vertex $v\in S_{k+i}$ we have $v\in \Gamma_{i+1}^*(x)$ when there is exactly one edge $yv\in E(\mathcal{G})$ where $y \in \Gamma_{i}^*(x)$ and there is no edge of the form $y'v \in E$ where $y' \in I_{k+i}$ and $y'\neq y$. Conditioning on the sizes of $I_{k+i}$ and $\Gamma_i^*(x)$ we see that each $v \in S_{k+i}$ is a member of $\Gamma_{i+1}^*(x)$ with probability ${d}_{i}^*(x)\cdot p \cdot (1-p)^{|I_{k+i}|-1}$. These events are independent, thus \begin{align*}{d}_{i+1}^*(x) &\sim  Bin\left(|S_{k+i}|,{d}_{i}^*(x)\cdot p \cdot (1-p)^{|I_{k+i}|-1}\right), \end{align*}conditional on $|S_{k+i}|,|I_{k+i}|$ and ${d}_i^*(x)$. \end{proof}

Let $x \in I_k$. Choosing $i=0$ in Lemma \ref{dom} \eqref{itm:gam2} gives ${d}^*(x) \sim  Bin\left(|S_k|,p(1-p)^{|I_k|-1} \right) $ conditional on $|S_k|$ and $|I_k|$ whereas Lemma \ref{dom} \eqref{itm:gam1} gives ${d}^*(x) \sim Bin \left(|S_k(x)|,p \right)$ conditional on $|S_k(x)|$. To relate \eqref{itm:gam1} to \eqref{itm:gam2} observe that conditional on $ |S_k|$ and $|I_k|$, $|S_k(x)| \sim  Bin\left(|S_k|,(1-p)^{|I_k|-1}\right).$ Item \eqref{itm:gam2} then follows as if $X \sim B(n, p)$ and, conditional on $X$, $Y\sim B(X, q)$, then $Y$ is a simple binomial variable with distribution $Y \sim B(n, pq)$.

The next two lemmas provide tail estimates on the sizes of $\Gamma_i$ and $\Gamma^*_i$. We prove them by induction where the inductive step comes from applying Chernoff bounds to the binomial distributions described in Lemma \ref{dom}. For Lemma \ref{neighup} this induction shows that w.h.p. the sequence ${d}(u), {d}_2(u),\dots$ is bounded above by the sequence $a_1np, a_2 (np)^2,\dots$ where the $a_i$ satisfy a recurrence relation. This recurrence can later be solved to give bounds on the sequence $a_i$ based on the exceptional probability desired. This strategy is inspired by \cite{Chung01thediameter}.

\begin{lemma}\label{neighup} Let  $\mathcal{G}\sim  \mathcal{G}(n,p)$ where $np = \omega\left(1\right)$. Then for $u \in V$ and $\alpha \in \mathbb{R}$,  $\alpha \geq 6$,
	\[\pr{\left|B_i(u)\right| > \alpha  (np)^i}= o\left(e^{-\alpha np/3}\right).\]
\end{lemma}
\begin{proof}Let $\ce_i:=\{ {d}_{i}(u)  \leq a_i(np)^i\}$ and $\mathcal{H}_i:=\bigcap\limits_{j=0}^i \ce_i$. We shall show
	\[\pr{\mathcal{H}_i^c} \leq  \sum\limits_{j=0}^i e^{-j}\exp\left(-\frac{\lambda^2}{2(1 + \lambda/3\sqrt{np})}\right),\]by induction on $i\geq 0$, where $a_i \geq 0$ is given by the recurrence \[a_{i+1} = a_{i} + \frac{\lambda_i\sqrt{a_{i}}}{(np)^{(i+1)/2}}, \; a_0=1,  \]and $\lambda_i =\sqrt{3i + \lambda^2 }$ for some $\lambda $ specified later. For the base case ${d}_0(u) = 1 = a_0$. Now observe
	\begin{align}\label{reccurance}
	a_i(np)^{i}np + \lambda_i \sqrt{a_i(np)^{i+1}}
	=(np)^{i+1} \!\left( \! a_i + \lambda_i\frac{\sqrt{a_i}}{(np)^{(i+1)/2}}\!\right)= a_{i+1} (np)^{i+1}.
	\end{align}
	Conditional on ${d}_i(u)$ we have  $ {d}_{i+1}(u) \preceq_1 Bin\left({d}_{i}(u)\cdot n,p\right)$. Thus by \eqref{reccurance} above
	\begin{align*}
	\pr{\left(\ce_{i+1}\right)^c\cap \ch_i}&:=\pr{\{{d}_{i+1}(u)> a_{i+1} (np)^{i+1}\} \cap \ch_i} \\
	&\leq \ex{\pr{Bin\left({d}_{i}(u)\cdot n,p\right)> a_i(np)^{i}np + \lambda_i \sqrt{a_i(np)^{i+1}}\;\bigg|{d}_i(u)}\mathbf{1}_{\mathcal{H}_i}}.\end{align*}Now by the Chernoff bounds, Lemma \ref{chb}, we have\begin{align}
	\pr{\left(\ce_{i+1}\right)^c\cap \ch_i}&\leq \ex{\exp\left(-\frac{\lambda_i^2 a_i(np)^{i+1}}{2\left(a_i(np)^{i+1} + \lambda_i \sqrt{a_i(np)^{i+1}}/3\right)}\right)\mathbf{1}_{\mathcal{H}_i}}\notag \\
	&= \ex{\exp\left(-\frac{3i+\lambda^2}{2\left(1 + \lambda_i /3\sqrt{a_i(np)^{i+1}}\right)}\right) \mathbf{1}_{\mathcal{H}_i}}\leq e^{-i}e^{-\frac{\lambda^2}{2\left(1 + \lambda /3\sqrt{np}\right)}}\pr{\ch_i },\label{dsigdsg} \end{align}for $n$ large enough since $a_i\geq 1$ and $ np =\omega(1)$ thus $\lambda_i /\sqrt{a_i(np)^{i+1}} \leq \lambda /\sqrt{np} $. Now observe that $\ch_{i+1} \subseteq \ch_{i}$ and $\ch_{i}$ is the disjoint union of $\ch_{i+1}$ and $ \left(\ce_{i+1}\right)^c\cap \ch_i$. Hence by \eqref{dsigdsg} we have
	\begin{align*}
	\pr{\ch_{i+1}} &= \pr{\ch_i} - \pr{\left(\ce_{i+1}\right)^c\cap \ch_i} = \left(1-  e^{-i}\exp\left(-\frac{\lambda^2}{2\left(1 + \lambda /3\sqrt{np}\right)}\right)\right)\pr{\ch_i }.  \intertext{If we continue iteratively and recall $\ch_0=\{{d}_0(u) \leq 1 \}$, thus $\pr{\ch_0}=1$, then we have} 
	\pr{\ch_{i+1}}&= \prod\limits_{j=0}^{i}\left(1-  e^{-j}\cdot e^{-\frac{\lambda^2}{2\left(1 + \lambda /3\sqrt{np}\right)}}\right)\pr{\ch_0 }\geq  1-  \sum\limits_{j=0}^{i}e^{-j}\cdot e^{-\frac{\lambda^2}{2\left(1 + \lambda /3\sqrt{np}\right)}} .\end{align*}Let $\lambda=k\sqrt{np}$ for any $k\geq 3$ and observe that
	\begin{align}\label{phcomp}
	\pr{\left(\ch_i\right)^c}&\leq \sum\limits_{j=0}^{i} e^{-j}\cdot  e^{-\frac{\lambda^2}{2\left(1 + \lambda /3\sqrt{np}\right)}}=\BO{\exp\left(-\frac{k^2 np  }{2(1 + k/3)}\right)}=O\left(e^{-3knp/4}\right),
	\end{align} where the last equality follows since $\frac{k^2 }{2(1 + k/3)} \geq \frac{3k}{4} $ provided $k\geq 3$, and $np = \omega\left(1\right)$.
	
	 We will show that $a_i \leq 2k$ for all $i$. Since $a_0=1\leq 2k$ assume $a_i \leq 2k$, then by \eqref{reccurance}
	\begin{align*}
	a_{i+1} &=a_i + \frac{\lambda_i\sqrt{a_i}}{(np)^{(i+1)/2}} = 1 + \frac{\lambda_0\sqrt{a_0}}{\sqrt{np}}+ \sum\limits_{j=1}^{i}\frac{\lambda_j\sqrt{a_j}}{(np)^{(j+1)/2}}. \intertext{Recall that $\lambda_i = \sqrt{3i+\lambda^2}$ and observe that $\lambda_0=\lambda=k\sqrt{np}$. Thus we have }
	a_{i+1}&= 1 + k + \sum\limits_{j=1}^{i}\frac{ \sqrt{3j+k^2np}\sqrt{2k}}{(np)^{(j+1)/2}} =1+k+O\!\left((np)^{-1/2}\right)\leq 2k,
	\end{align*}for large $n$. Finally, conditional on $\bigcap\limits_{j=0}^i\{{d}_j(u) \leq 2k(np)^i \}\subseteq \ch_i$ we have 
	\[|B_i(u)| = \sum\limits_{j=0}^i{d}_j(u) \leq \sum\limits_{j=0}^i 2k(np)^i \leq \alpha(np)^i,\] where $\alpha = (2+1/100)k$, now for this $\alpha$ we have $\pr{\left( \ch_i\right)^c} =o\left(e^{-\alpha np/3}\right) $ by \eqref{phcomp}.\end{proof}

\begin{lemma}\label{neighlow} Let $\mathcal{G}\sim  \mathcal{G}(n,p)$, and $i\in \mathbb{Z}$ satisfy $
	1\leq  i \leq \lfloor \log (n) /\log(np) \rfloor -3.$
	Let $\Psi_2$ be defined with respect to $\mbfs(\mathcal{G},\{u,v\})$ for $u,v \in V$ such that $u \neq v$.
	\begin{enumerate}[(i)] 
		\item \label{itm:bd1}Let $c>0$, $np\geq c\log n$. Then $\pr{{d}_i^*(y)< 15 (np)^{i-1}\Big| \;y \in \Psi_2} = o\left(e^{- 4np }\right).$
		\item \label{itm:bd2} If $np = \omega\left( \log n\right)$ then $\pr{{d}_i^*(y)<\tfrac{9}{10} (np)^{i}\Big|\; y \in \Psi_2} = o\left(n^{-K} \right)$ for any fixed $K\geq 0$.  
		\item \label{itm:bd3} If $np -\log n \rightarrow \infty$ then $\pru{\mathcal{C}}{|B_{j}(v)|< 15 (np)^{i-5}} = o\left(n^{-4} \right)$ for any $5 \leq j\leq i -2$.   
	\end{enumerate} 
\end{lemma}

\begin{proof} We will first set up the general framework for a neighbourhood growth bound and then apply this bound under different conditions to prove Items \eqref{itm:bd1}, \eqref{itm:bd2} and \eqref{itm:bd3}. 
	
	Run $\mbfs(\mathcal{G},\{u,v\})$ and let $y \in I_h$, $n_i := |S_{i+h}|$, $p_i:= p \cdot \left(1-p\right)^{|I_{i+h}|-1}$ and $r_i = \prod_{j=i_0}^{i} n_jp_j$. We wish to show that there exists some $i_0 \in \mathbb{Z},\; i_0 \geq 0$ such that for all $i\geq i_0$:
	\begin{equation}\label{recur}
	\pr{{d}_{i+1}^*(y)< a_{i+1} r_i} \leq  (i+1)\exp\left(-\lambda^2/2\right),    
	\end{equation}
	where $a_{i}\geq 0 $ satisfies $ a_{i+1} = a_{i} - \lambda\sqrt{a_{i}}/\sqrt{r_{i}}$, for some initial $  a_{i_0}$ we will find later. Observe
	\begin{align*}
	a_ir_{i-1} n_{i}p_{i} - \lambda \sqrt{a_ir_{i-1} n_{i}p_{i}}
	=\left( \! a_i - \lambda\frac{\sqrt{a_i}}{\sqrt{r_{i}}}\!\right)r_{i} = a_{i+1} r_{i}.
	\end{align*}
	Applying Lemma \ref{dom} \eqref{itm:gam2} and conditioning on $\mathfrak{F}_{i+h}$ yields ${d}_{i+1}^*(y)\sim Bin\left( n_i,{d}_{i}^*(y)p_i \right).$ 
	
		Let $\mathcal{H}_i:=\{ {d}_{i}^*(u) \geq a_ir_{i-1}\}\in \mathfrak{F}_{i+h}$ and assume $\pr{\ch_i^c}\leq ie^{-\lambda^2/2} $. Now by Lemma \ref{chb} \eqref{itm:chern1},

	\begin{align*}
	\pr{{d}_{i+1}^*(y)< a_{i+1} r_{i}} &= \ex{\pr{{d}_{i+1}^*(y)< a_ir_{i-1} n_{i}p_{i} - \lambda \sqrt{a_ir_{i-1} n_{i}p_{i}} \big|\mathfrak{F}_{i+h}} }\\
	&\leq \ex{\pr{Bin\left(n_{i},{d}_{i}^*(y) p_{i}\right)< a_ir_{i} - \lambda \sqrt{a_ir_{i}}\big|\mathfrak{F}_{i+h}} \mathbf{1}_{\mathcal{H}_i}} + \pr{\mathcal{H}_i^c}\\
	&\leq \exp\left(-\lambda^2a_ir_{i}/(2 a_ir_{i})\right) +i\exp\left(-\lambda^2/2\right) =(i+1)\exp\left(-\lambda^2/2\right) .
	\end{align*} The above always holds, however it may be vacuous as if $i$ is too large then $a_i$ may be negative. This can also happen for an incorrect choice of the starting time $i_0$ and initial value $a_{i_0}$. We address this in the application making sure to condition on events where everything is well defined. In this spirit let $\ell:=\lfloor \log(n) /\log(np) \rfloor -h-1$ and define the event \[\mathcal{I}:=\bigcap_{i=0}^\ell \left\{|I_{i+h}| \leq 26(np)^{i+h}\right\}\cap \left\{{d}_i^*(y)\leq 13(np)^i \right\}\cap \left\{\left|S_{i+h}\right|\geq n - 26(np)^{i+h}\right\}.\]Conditioning on the event $\mathcal{I}$ and the filtration $\mathfrak{F}_{i+h}$ for any $i\leq \ell$ ensures that $Bin\left( n_i,{d}_{i}^*(y)p_i \right)$ is a valid probability distribution and $n_ip_i=(1-o(1))np$. By Lemma \ref{neighup} with $\alpha=13$,
	\begin{align}\label{probD}
	\pr{\mathcal{I}^c} &\leq 2\sum\limits_{i=0}^{\ell+2} \pr{|B_{i}(u)| > 13(np)^{i}} =o\left(e^{ - 4np}\right).
	\end{align}\textit{Item \eqref{itm:bd1}}: recall from \eqref{eq:psi2} that if $y \in \Psi_2(u)\cup\Psi_2(v)\subseteq I_2$ then ${d}^*(y) >D$, defined at \eqref{Ddef}. Thus conditional on  $\mathcal{I}\cap\mathfrak{F}_3 $, ${d}_{2}^*(y)\succeq_1 Bin\left( n(1-\lo{1}),Dp(1-\lo{1}) \right)$. If we choose $\lambda = 3\sqrt{np}$, then since in this regime $D= \max\left\{\lceil\frac{50}{c} \rceil, 50\right\}$ applying Lemma \ref{chb} \eqref{itm:chern1} yields
	\begin{align*}
	\pr{{d}_{2}^*(y)< Dn_1p_1/2}&=\ex{\mathbb{P}\left({d}_{2}^*(y)< Dn_1p_1/2\big|\mathfrak{F}_3\right)\left(\mathbf{1}_{\mathcal{I}}+\mathbf{1}_{\mathcal{I}^c} \right)} \leq e^{-Dnp /10}+\pr{\mathcal{I}^c} \leq e^{-\lambda^2/2}.
	\end{align*}Take $i_0=1$ and $a_2=D/3$ since on $\mathcal{I}$ we have $D/2 n_1p_1 \geq Dnp/3$. Now $a_2 \geq \dots \geq a_i$ so on the event $\mathcal{I}$ we have the following for any $3\leq i\leq \lfloor \log(n) /\log(np) \rfloor -3$
	\begin{align*}
	a_i = a_{2} - \sum\limits_{k=2 }^{i-1} \frac{\lambda \sqrt{a_{k}}}{\sqrt{r_{k}}} \geq \frac{D}{3} - (3 +\lo{1})\sqrt{\frac{D}{3np} } \geq 16,
	\end{align*} since conditional on $\mathcal{I}$ we have $r_i = \prod_{j=i_0}^{i}n_jp_j\geq (1-\lo{1})(np)^{i}$ for any $1\leq i\leq \lfloor \log (n) /\log(np) \rfloor -3$. Notice also that ${d}^*(y)>D > 15(np)^0$ so by \eqref{recur}
	\begin{align*}
	\pr{{d}_{i+1}^*(y)< 15 (np)^{i}} &\leq \pr{{d}_{i+1}^*(y)< a_{i+1} r_{i}} + \pr{\mathcal{I}^c} \leq (i+1)e^{-\lambda^2/2 } +o\left(e^{- 4np }\right)=o\left(e^{ - 4np}\right).
	\end{align*}\textit{Item \eqref{itm:bd2}}: for this case on the event $\mathcal{I}$ we have $n_ip_i =(1-o(1))np =\omega(\log n)$ for every $0\leq i \leq \ell $, this is why $D=0$ when $np=\omega(\log n )$ in the definition \eqref{Ddef} as we do not need to rely on the fact that ${d}^*(y)$ greater than some constant to start the branching. 
	
	Fix $K>0$ and let $\lambda=\sqrt{3K\log n}$. As before conditioning on $\mathcal{I}\cap\mathfrak{F}_3 $ ensures that ${d}^*(y)\sim Bin\left( n_0,p_0\right)\succeq_1 Bin\left( n(1-\lo{1}),p(1-\lo{1}) \right)$. By Lemma \ref{chb} \eqref{itm:chern1}, 
	\begin{align*}
	\ex{ \pr{{d}^*(y)< r_0 - (5/4)\lambda\sqrt{r_0} \big|\mathfrak{F}_3}\left(\mathbf{1}_{\mathcal{I}}+\mathbf{1}_{\mathcal{I}^c} \right)} &\leq e^{-25\lambda^2/32} + \pr{\mathcal{I}^c}\leq \exp\left(-\lambda^2/2\right).
	\end{align*}
	Take $i_0=0$, $a_1=19/20$ since on $\mathcal{I}$ we have $r_0 -(5/4)\lambda\sqrt{r_0}  \geq 19np/20$. Now $a_1 \geq \dots \geq a_i$ so on the event $\mathcal{I}$ we have the following for any  $2\leq i\leq \lfloor \log(n) /\log(np) \rfloor -3$ 
	\begin{align*}
	a_i = a_{1} - \sum\limits_{k=1 }^{i-1} \frac{\lambda \sqrt{a_{k}}}{\sqrt{r_{k}}} \geq \frac{19}{20} - (1 +\lo{1})\frac{\sqrt{19\cdot 3K \log n}}{\sqrt{20np}}= \frac{19}{20}-o(1) \geq \frac{9}{10}.
	\end{align*} Thus for any $1\leq i\leq \lfloor \log(n) /\log(np) \rfloor -3, \; K>0$ we have  
	\begin{align*}
	\pr{{d}_{i}^*(y)<9/10(np)^{i}} &\leq \pr{{d}_{i}^*(y)< a_{i} r_{i-1}} + \pr{\mathcal{I}^c}\leq (i+1)e^{-\lambda^2/2 } +e^{- 4np }\leq o\left(n^{-K} \right).
	\end{align*}\textit{Item \eqref{itm:bd3}}: We assume that $np= \theta(\log n)$ for larger $p$ the result follows from stochastic domination. Since $\mathcal{G} \in \mathcal{C}$ there exists a path $u:=u_0,u_1,\dots ,u_4$ with $u_{j-1}u_{j} \in E$ for each $1\leq j\leq 4$. Let $f(u_j)=\left|\left\{v \in V\backslash  \{u_0,\dots ,u_4\}\,:\,u_jv \in E\right\}\right|$. Then, for $D$ from \eqref{Ddef}, by Lemma \ref{chb} \eqref{itm:chern2},
	\begin{align*}
\pr{f(u_j)< D}&=\pr{Bin(n-\ell -1,p)<D} \leq e^{-(n-\ell-1)p}\cdot \left(\frac{e(n-\ell-1)p}{D}\right)^D=e^{-(1-o(1))np}.
\end{align*} 
	Let $\mathcal{E}$ be the event $\left\{{d}(u_{j_0})\geq D \text{ for some } 0 \leq j_0\leq 4 \right\}$. Then as $\{f(u_j)\}_{j=0}^4$ are i.i.d. we have  
	\begin{align}\label{jbd}
	\pru{\mathcal{C}}{\mathcal{E}^c } &\leq \pr{f(u_j)< D}^{5}/\pr{\mathcal{C}} \leq e^{-5(1-o(1))\log n} \leq o\left(n^{-4} \right) . \end{align}  On $\mathcal{E}$ there is some $u_{j_0} \in V$ with $d(u,u_{j_0})=j_0\leq4$  and ${d}(u_{j_0})>D$. We use the stochastic domination ${d}_i(u_{j_0}) \succeq_1 {d}_i^*(u_{j_0})$ to bound the growth of $|B_{i+j_0}(u)|$ from below by that of ${d}_{i}^*(u_{j_0})$, where $u_{j_0} \in I_{j_0}$ is defined with respect to $\mbfs(\cg, \{u,v \} )$ for some  $v \in V$.  Let $\lambda =3\sqrt{\log n}$ and recall $D= \max\left\{\lceil\frac{50}{c} \rceil, 50\right\} $. On $\mathcal{I}$, $r_{j_0+1}=(1-o(1))np $, thus by Lemma \ref{chb} \eqref{itm:chern1}:
	\begin{align*} 
	\ex{\mathbb{P}\left({d}_{j_0+2}^*(u)< dn_{j_0+1}p_{j_0+1}/2\big|\mathfrak{F}_{j_0+1}\right)\mathbf{1}_{\mathcal{I}\cap\mathcal{E}}} &\leq \ex{e^{-dr_{j_0+1} /8}\mathbf{1}_{\mathcal{I}\cap\mathcal{E}}} \leq e^{-\lambda^2/2}.
	\end{align*}
	Take $i_0 =j_0+1$ and $a_{j_0+2}=d/3$ since on $\mathcal{I}\cap\mathcal{E}$ we have $dn_{j_0+1}p_{j_0+1}/2 \geq dnp/3$. Now $a_{i_0} \geq \dots \geq a_i$ and on the event $\mathcal{I}\cap\mathcal{E}$ we have $r_i =(1-o(1)) (np)^{i-j_0}$. Thus  we have the following for any $\varepsilon>0$ and $j_0+3\leq i\leq \lfloor \log(n) /\log(np) \rfloor -j_0-1$:
	\begin{align*}
	a_i = a_{j_0+2} - \sum\limits_{k=j_0+2 }^{i-1} \frac{\lambda \sqrt{a_{k}}}{\sqrt{r_{k}}} \geq \frac{d}{3} - (3 +\lo{1})\sqrt{\frac{d\log n}{3(np)^2} }\geq 16.
	\end{align*} Notice also ${d}_{j_0+1}^*(y)>d > 15(np)^0$. Thus for any $4\leq i\leq \lfloor \log (n) /\log(np) \rfloor -5$:
	\begin{align*}
	\pru{\mathcal{C}}{|B_{i+1}(u)|< 15 (np)^{i-5}} & \leq  \pru{\mathcal{C}}{{d}_{i+1}(u)< 15 (np)^{i-j_0-1}\big|\mathcal{I},\mathcal{E}}   + \pru{\mathcal{C}}{\mathcal{I}^c}+\pru{\mathcal{C}}{\mathcal{E}^c} \\ 
	&\leq  \pr{{d}_{i+1}^*(y)< a_{i+1} r_{i}\big|\mathcal{I},\mathcal{E}}/\pr{\cc} + \pr{\mathcal{I}^c}/\pr{\mathcal{C}}+ \pru{\mathcal{C}}{\mathcal{E}^c}, \end{align*}which is at most $2(i+1)e^{-\lambda^2/2 } +o\left( e^{-4np}\right)+o\left(n^{-4} \right)= o\left(n^{-4} \right)$ by the bounds on $\pr{\cc}$, $\pr{\mathcal{I}}$ and $\pr{\ce^c}$ given by \eqref{eq:probC}, \eqref{probD} and \eqref{jbd} respectively.
	\end{proof}

The next lemma allows us to couple the complex $\Psi_1$ and $\Phi$ neighbourhood distributions to the far more simple $\Gamma^*$ and $\Gamma$-neighbourhood distributions. 
\begin{lemma}\label{sdom}
	Let $\mathcal{G}\sim  \mathcal{G}(n,p)$ where $c\log n  \leq np \leq o(n^{1/3})$ for any $c>0$. Let $I_1$ and the $\varphi , \;\psi , \;\psi_2$ and ${d}^*$-distributions be defined with respect to $\mbfs(\mathcal{G},\{u,v\})$, $u,v \in V$. Then  \begin{enumerate}[(i)]
		\item \label{itm:d1} $\displaystyle{\pr{\varphi(x)\neq {d}^*(x)\big|x \in I_1}=e^{-(1-o(1))np},}$
		\item \label{itm:d2}  $\displaystyle{\pr{\psi(u)\neq {d}^*(u)}=  e^{-(1-o(1))np}},$
		\item \label{itm:d3}  $\displaystyle{\pr{\psi(u)\neq {d}(u) \text{ or }\psi(v)\neq {d}(v)  }\leq 2np^2+ e^{-(1-o(1))np},}$
		\item \label{itm:d5}  $\displaystyle{\pr{\psi_2(u)\neq {d}_2^*(u) \text{ or }\psi_2(v)\neq {d}_2^*(v)  }=  e^{-(1-o(1))np}},$
		\item \label{itm:d4} $\displaystyle{\pr{\psi_2(u)\neq {d}_2(u) \text{ or }\psi_2(v)\neq {d}_2(v)  }\leq  4n^3p^4 +e^{-(1-o(1))np}+  O\!\left(n^2p^3 \right).}$
	\end{enumerate} 
\end{lemma}

\begin{proof}

		\textit{Item \eqref{itm:d1}}: run  $\mbfs(\mathcal{G},\{u,v\})$ and let $x \in I_1$. By the definition \eqref{eq:phi} of $\psi(x)$, if ${d}^*(\tilde{x})>D$ for all $\tilde{x} \in \Gamma_1^*(x)$ then $\varphi(x)={d}^*(x)$, where $D$ is bounded \eqref{Ddef}. Hence for $x \in I_1$,
	\begin{align}\label{uniton}
	\pr{\varphi(x)\neq{d}^*(x)\big|\mathfrak{F}_2}&= \pr{{d}^*(\tilde{x})\leq D \text{ for some }\tilde{x}\in \Gamma_1^*(x)\big|\mathfrak{F}_2} \leq \sum\limits_{\tilde{x} \in \Gamma_1^*(x)}\pr{{d}^*(\tilde{x})\leq D\big|\mathfrak{F}_2}.
	\intertext{If $\tilde{x}\in \Gamma_1^*(x), x \in I_1$ then $\tilde{x}\in I_2$. Knowing the parent of $\tilde{x}$ does not affect the ${d}^*$-distribution conditioned on $\{\tilde{x} \in I_2\}$, so by Lemma \ref{dom} \eqref{itm:gam2} as $|S_2|,|I_2|\in \mathfrak{F}_2$ we have }
	\pr{\varphi(x)\neq{d}^*(x)\big|\mathfrak{F}_2}&\leq{d}^*(x) \pr{Bin\left(|S_2|,p(1-p)^{|I_2|-1}\right)\leq D\Big||S_2|,|I_2|}.\notag\end{align}Let $\mathcal{E}_x^j:=\left\{|I_{j}|\leq 12 (np)^j \right\}\cap \{{d}^*(x)\leq 6np \} \cap\{|S_{j}|\geq n- 12(np)^j \}$ for $x \in I_{j-1}$ and observe
	\begin{equation}\label{probej}\pr{\left(\mathcal{E}_x^j\right)^c}\leq 2\pr{|B_j(u)|>6(np)^j}+ \pr{|B_1(x)|>6np}=  o(e^{-2np})  ,\end{equation} by Lemma \ref{neighup} with $\alpha=6$. Now by Lemma \ref{chb}\eqref{itm:chern2} and Bernoulli's inequality \eqref{bern}
	\begin{align*}
		\pr{\varphi(x)={d}^*(x)\big|\mathfrak{F}_2}\mathbf{1}_{\mathcal{E}_x^1} &\leq  {d}^*(x) e^{-|S_2|p(1-p)^{|I_2|}}\left(\frac{e|S_2|p(1-p)^{|I_2|}}{D}\right)^D\mathbf{1}_{\mathcal{E}_x^1}   \leq e^{-(1-o(1))np},
	\end{align*}the result follows by taking expectation of the above and using \eqref{probej} with $j=1$. 
	\medskip
		
\noindent\textit{Item \eqref{itm:d2}}: for $\tilde{u} \in I_1 $ the distribution of ${d}^*(\tilde{u})$ conditioned on $|S_1|,|I_1|$ is known by \ref{dom} \eqref{itm:gam2}. Thus using the bound $ (1-p)^n \leq \exp(-np)$ we obtain the following for $\tilde{u} \in I_1 $, 
\begin{align}\label{nondiff}
\pr{{d}^*(\tilde{u})= 0\big|\mathfrak{F}_1}&= \pr{Bin\left(|S_1|,p(1-p)^{|I_1|-1}\right)=0\Big|\mathfrak{F}_1} \leq \exp\left(-|S_1|p(1-p)^{|I_1|}\right).
\end{align}
Recall the definition \eqref{eq:psi} of $\Psi_1(u)$. If $\tilde{u}\in \Gamma_1^*(u)$ then $\tilde{u}\in I_1$ and knowing the parent of $\tilde{u}$ does not affect the ${d}^*$-distribution conditioned $\{\tilde{u} \in I_1\}$. Thus, similarly to \eqref{uniton}, for $\tilde{u}\in \Gamma_1^*(u)$,
\begin{align}\label{exp6}
\pr{\psi(u)\neq {d}^*(u)\big|\mathfrak{F}_1} &= \pr{\varphi(\tilde{u})= 0 \text{ for some } \tilde{u}\in \Gamma_1^*(u)\big|\mathfrak{F}_1}\leq d^*(u) \cdot \pr{\varphi(\tilde{u})= 0\big|\mathfrak{F}_1}.\notag \intertext{Now using the coupling inequality \eqref{prk}, yields the following for $\tilde{u} \in I_1 $ }
\pr{\psi(u)\neq {d}^*(u)\big|\mathfrak{F}_1}&\leq{d}^*(u)\pr{{d}^*(\tilde{u})= 0\big|\mathfrak{F}_1} + {d}^*(u)\pr{\varphi(\tilde{u})\neq {d}^*(\tilde{u})\big|\mathfrak{F}_1}.
\end{align}
Recall $\mathcal{E}_u^1$ from \eqref{probej} and observe $\pr{{d}^*(\tilde{u})= 0\big|\mathfrak{F}_1}\mathbf{1}_{\mathcal{E}_x^1}\leq e^{-(1-o(1))np} $ by \eqref{nondiff}. Thus  
\begin{align*}\pr{\psi(u)\neq {d}^*(u)} &\leq e^{-(1-o(1))np} + 6np\pr{\varphi(\tilde{u})\neq {d}^*(\tilde{u})\big|\tilde{u}\in I_1} + o\left(e^{-2np}\right)= e^{-(1-o(1))np},
\end{align*}
by taking the expectation of \eqref{exp6} conditioned on $\mathcal{E}_u^1$ then using the bound on $\pr{{d}^*(\tilde{u})= 0\big|\mathfrak{F}_1}$ above and bounds on $\pr{\varphi(\tilde{u})\neq {d}^*(\tilde{u})\big|\tilde{u}\in I_1}$ and $\pr{(\ce_u^1)^c}  $ from Item (i) and \eqref{probej} respectively.\medskip 

	\noindent \textit{Item \eqref{itm:d3}}: let $I_0=\{u,v\}$ and $\ch:=\{{d}(u)= {d}^*(u), \;{d}(v)={d}^*(v)\}$. By Item \eqref{itm:d2} 
	\begin{align}\label{lasttag}
	\pr{\psi(u)\neq {d}(u) \text{ or } \psi(v)\neq {d}(v) }\notag &\leq \pr{\{\psi(u)\neq {d}(u) \text{ or } \psi(v)\neq {d}(v)\}\cap \ch } + \pr{\ch^c}\notag \\
	& \leq \pr{\psi(u)\neq {d}^*(u)\text{ or }\psi(v)\neq{d}^*(u) } + \pr{\ch^c}\notag \\ &\leq 2e^{-(1-o(1))np} + \pr{\ch^c}.
	\end{align}To calculate $\pr{\ch^c}$ in the above recall the definition $\eqref{eq:gamdef}$ of ${d}^*(u)$ and observe 
	\begin{align}
	\pr{\ch^c}&= \pr{\{uv \in  E \} \cup \{ xu \in E \text{ and } xv \in E \text{ for some }x \in V\backslash I_0\}}\notag \\
	&\leq \pr{uv \in  E} + {\sum}_{x \in V\backslash I_0}\pr{xu \in E \text{ and } xv \in E } = p + (n-2)p^2 .  \label{pce}
	\end{align}Finally, combining \eqref{lasttag} and \eqref{pce} yields the bound \begin{align*}\pr{\psi(u)\neq {d}(u) \text{ or } \psi(v)\neq {d}(v) }&\leq  2e^{-(1-o(1))np}+ p + (n-2)p^2 \leq 2np^2+ e^{-(1-o(1))np} .\end{align*}
	
	\noindent \textit{Item \eqref{itm:d5}}: Define the following events 
	\[\cj:= {\bigcap}_{x \in I_1} \left\{\varphi(x) = {d}^*(x) \right\}, \qquad \cd= \left\{{d}(u),{d}(v) \leq (1+9/\min\{c,1\})np \right\}. \] 
	Notice $\{\psi_2(u)\neq {d}_2^*(u)\text{ or }\psi_2(v)\neq{d}_2^*(u) \}= \{\psi(u)\neq {d}^*(u)\text{ or }\psi(v)\neq{d}^*(u) \}\cup \cj^c $. Now   
	\begin{equation}\label{J}\pr{\cj^c}\leq \ex{|I_1|\pr{\varphi(x) = {d}^*(x) \big|\mathfrak{F}_1}(\mathbf{1}_{\cd}+\mathbf{1}_{\cd^c})} \leq O\!\left(npe^{-(1-o(1))np}\right) + n\pr{\cd^c},  \end{equation} 
	which is $o(e^{-np(1-o(1))})$ by Item \eqref{itm:d1} and since $\pr{\cd^c} = o\left(e^{-np}/n\right)$ by Lemma \ref{chb} \eqref{itm:chern1}.\medskip

	\noindent \textit{Item \eqref{itm:d4}}: let $I_0=\{u,v\}$ and $\cl:=\{{d}_2(u)=  {d}_2^*(u),  \;{d}_2(v)= {d}_2^*(v)\}$. Then\begin{equation} \label{clset}\cl:= \left(\bigcap_{x \in {d}(u), y \in {d}(v)}\!\!\!\!\!\left\{xy \notin E \right\}\!\!\right) \cap \left(\bigcap_{z \in S_1}\!\!\left\{ |\{x \in I_1: xz \in E\}|\leq 1\right\}\!\!\right)\cap \ch,  \end{equation} by the definition $\eqref{eq:gamdef}$ of ${d}_2^*(u)$. Observe that by the Bernoulli inequality \eqref{bern},
	\begin{align*}
	&	\pr{ |\{x \in I_1: xz \in E\}|>1\big|\mathfrak{F}_1} = 1-  {\sum}_{a =0,1}\pr{ |\{x \in I_1: xz \in E\}|=a\big|\mathfrak{F}_1}\\ 
	&\qquad= 1 - (1-p)^{|I_1|} - |I_1|p(1-p)^{|I_1|-1}\leq 1 - (1-|I_1|p) -|I_1|p(1-|I_1|p)= \left(|I_1|p\right)^2.\end{align*}
	By  $\eqref{clset}$, the above estimate on $\pr{ |\{x \in I_1: xz \in E\}|>1\big|\mathfrak{F}_1}$ and $\ch \in \mathfrak{F}_1$, we have 
	\begin{align*}
	\pr{\cl^c\big|\mathfrak{F}_1}&\leq  \sum\limits_{ x \in {d}(u), y \in {d}(v)}\pr{xy \in  E\big|\mathfrak{F}_1}+\sum\limits_{z \in S_1}\pr{ |\{x \in I_1: xz \in E\}|>1\big|\mathfrak{F}_1} +
	\pr{\ch^c\big|\mathfrak{F}_1}\\
	&\leq {d}(u){d}(v)p + |S_1|\left(|I_1|p\right)^2+  \pr{\ch^c\big|\mathfrak{F}_1}.  
	\end{align*}
	Then by the bound on $\pr{\cl^c\big|\mathfrak{F}_1}$ above, the tower property and Cauchy Schwartz inequality,
	\begin{align}\label{ell}
	\pr{\cl^c} &\leq  p\sqrt{\ex{{d}(u)^2}\ex{{d}(v)^2}} + p^2\sqrt{\ex{|S_1|^2}\ex{|I_1|^4}}   +  \pr{\ch^c}=4n^3p^4 +  O\!\left(n^2p^3 \right),  \end{align} the last equality holds by \eqref{pce} and as $|S_1|\sim Bin \left(n-2,(1-p)^2\right), \;|I_1|\sim Bin \left(n-2,2p(1-p)\right)$ by Lemma \ref{dom}, then applying the bound on moments of binomial r.v.'s from \eqref{dmom}.

	Finally let $\cf:= \left\{\psi_2(u)= {d}_2(u),\; \psi_2(v)= {d}_2(v) \right\}$. Then by the definitions \eqref{eq:psi} and \eqref{eq:psi2}  of the vertex sets $\Psi_1(u)$ and $\Psi_2(u)$ we have $\cf:= \left\{\psi(u)= {d}(u),\; \psi(v)= {d}(v) \right\} \cap \cj \cap \cl$.
	Thus
	\begin{align*}
	\pr{\cf^c}& \leq \pr{\psi(u)\neq  {d}(u)\text{ or } \psi(v)\neq  {d}(u) }+ \pr{\cj^c} + \pr{\cl^c}\leq  4n^3p^4 +e^{-(1-o(1))np}+  O\!\left(n^2p^3 \right),
	\end{align*}by Item \eqref{itm:d3}, \eqref{ell} and \eqref{J}.  
\end{proof}

Recall from Definition \eqref{calA} that $\mathcal{A}_{u,v}$ is the event that $\cg(n,p) $ satisfies the strong path property for $u,v \in V$ and some $k\leq \log(np)/\left(2\log n\right) +2$. Recall also Definition \eqref{calB} of $\mathcal{B}_{u,v}$ which is the event the pruned first neighbourhoods $\Psi_1(u),\Psi_1(v)$ are both non-empty.

\begin{lemma}\label{setlem} Let $\mathcal{G}\sim  \mathcal{G}(n,p)$ where $c\log n\leq  np < n^{1/10}$, $c>0$. Then for $u,v \in V, \, u\neq v$ ,\[\pr{\left( \mathcal{A}_{u,v} \right)^c}=  o\left(e^{-7\min \left\{np,\log n \right\}/2}\right)\qquad \text{and} \qquad  \pr{(\mathcal{B}_{u,v})^c}= e^{-(1-o(1))np}.\]
	
\end{lemma}
\begin{proof}
	
	Run $\mbfs(\mathcal{G},\{u,v\})$, $u,v \in V$. For $k\geq 0$ let $\mathcal{T}:=\mathcal{T}_{u,v,k}=\mathcal{T}_1\cap\mathcal{T}_2 $ where
	\[\mathcal{T}_1:=\left\{ \left|S_{k+2}\right|\geq n - n^{5/6} \right\},\quad \mathcal{T}_2:=\left\{|\Gamma_{k}^*(x)\times \Gamma_{k}^*(y)| \geq 4n, \text{for all } (x,y) \in \\ \Psi_2(u)\times\Psi_2(v)\right\}.\]
	On the event $\mathcal{T}_1$ when $\mbfs(\mathcal{G},\{u,v\})$ has run for $k+2$ iterations there is still a lot of the graph yet to explore and the algorithm will run for at least one more iteration.    
	The $k$ in the definition of $\mathcal{T}$ will be the one occurring in the description of $\mathcal{A}_{u,v} $. Set the value of $k$ to be 
	\begin{equation}\label{kchoice}
	k:=k(n,p)=\begin{cases}
	\left\lceil\log \left( \frac{4n}{(15)^2}  \right) /\left(2\log(np) \right) \right\rceil +1 & \text{if } np=c\log n \text{ where }c>0\\
	\left\lceil\log \left( \frac{400n}{81}  \right) /\left(2\log(np)\right) \right\rceil &\text{if }np=\omega\left(\log n \right).
	\end{cases}
	\end{equation} Thus $k\leq \log(np)/\left(2\log n\right) +2$ for large $n$. It remains to show that for $k$ given by \eqref{kchoice}:\[\pr{\cg \notin {A}_{u,v}^{n,k}} = o\left(e^{-7\min \left\{np,\log n \right\}/2}\right).\]

	Let $\mathcal{R}:=\mathcal{R}_{u,v}$ be the event $\left\{\left|\Psi_2(u)\times\Psi_2(v)\right|\leq (12(np)^2)^2 \right\}.$ Since  $\psi_2(u) \leq{d}_2(u)$ for any $u \in V$, we have $	\pr{\mathcal{R}^c}= o\left(e^{-4np}\right)$ by Lemma \ref{neighup} with $\alpha=12$. Thus by the tower property
	\begin{align}
	\pr{\mathcal{T}^c}&\leq \pr{\mathcal{T}_1^c }+\pr{\mathcal{R}^c}+ \ex{\pr{\mathcal{T}_2^c \big| \mathfrak{F}_3}\mathbf{1}_{\mathcal{R}}}  \leq 2\pr{|B_{k+2}(u)|> n^{5/6}/2}+2\pr{|B_{2}(u)|> 12(np)^2}\notag\\
	&\qquad+ 2\ex{\psi_2(u)\psi_2(v)\mathbf{1}_{\mathcal{R}}\pr{{d}_{k}^*(w) < 2n^{1/2} \big| \{w \in  \Psi_2\},\mathfrak{F}_3}}\notag,
	\intertext{as $\{{d}^*(x)\cdot {d}^*(y) < k\}\subseteq \{{d}^*(x) \text{ or }{d}^*(y)< \sqrt{ k}\}$. Provided $np\leq n^{1/10}$ the choice of $k$ given by \eqref{kchoice} satisfies the conditions of Lemma \ref{neighlow}. Thus by Lemmas \ref{neighup}, \ref{neighlow} \eqref{itm:bd1} and \ref{neighlow} \eqref{itm:bd2} we have} \pr{\mathcal{T}^c}&\leq o\left(e^{-9np/2}\right)+ \BO{
		np}^4\cdot \pr{{d}_{k}^*(w) < 2n^{1/2} \big| w \in  \Psi_2} =o\left(e^{-7\min \left\{np,\log n \right\}/2}\right),\label{prrnt}\end{align} 
	where the bound $\pr{{d}_{k}^*(w) < 2n^{1/2} \big| w \in  \Psi_2}\leq e^{-4\min\{np,\log n\}}$ above covers the different values of $np$ and comes from amalgamating Lemmas \ref{neighlow} \eqref{itm:bd1} and \eqref{itm:bd2}, with $K=4$ in the latter. 
	
	Let $\mathcal{L}_{x,y}$ be the following event indexed by $(x,y) \in \Psi_2(u)\times\Psi_2(v)$,  \[\mathcal{L}_{x,y}:=\{x'y'\notin E, \text{for every pair } (x',y') \in \Gamma_k^*(x)\times\Gamma_k^*(y)\}.\] This is independent of $\mathfrak{F}_{k+2}$ as each $x'y'$ has not been checked up to iteration $k+2$, thus
	\begin{align}
	\pr{\mathcal{L}_{x,y}\big|\mathfrak{F}_{k+2}}\mathbf{1}_{\mathcal{T}}&= \pr{x'y'\notin E }^{{d}_k^*(x){d}_k^*(y)}\mathbf{1}_{\mathcal{T}} \leq (1-p)^{4n}  \leq e^{-4np }. \label{conbd}
	\end{align}
	Recall Definition \ref{A} of the strong path property $A_{u,v}^{n,k}$ which we can express as
	\begin{equation*}
	\left\{\cg \notin A_{u,v}^{n,k}\right\} = \bigcup\limits_{(x,y) \in  \Psi_2(u)\times\Psi_2(v) }\left\{\Gamma_k^*(x)= \emptyset\right\}\cup
	\left\{\Gamma_k^*(y)= \emptyset\right\}\cup\mathcal{L}_{x,y}.
	\end{equation*}Observe that for each $i,j\geq 0$ the random variables $\{{d}_j^*(w)\}_{w \in I_i}$ are identically distributed. Recall also that $\Psi_1(u),\Psi_1(v),\mathcal{R} \in \mathfrak{F}_3$. Let $ \mathfrak{P}$ denote $\pr{\left\{\cg \notin A_{u,v}^{n,k}\right\}\cap \mathcal{R}\cap \mathcal{T}}$. By the union bound, tower property and since $\psi_1(u)\psi_1(v)\leq 12^2(np)^4 $ on $\mathcal{R}$, we have 
	\begin{align*}
	\mathfrak{P}&\leq \ex{ \sum\limits_{(x,y) \in  \Psi_2(u)\times\Psi_2(v) }\ex{\left(\mathbf{1}_{\mathcal{L}_{x,y}\cup\{{d}_{k}^*(x)=0\}\cup\{{d}_{k}^*(y)=0\}}\right) \mathbf{1}_{\mathcal{R}}\mathbf{1}_{\mathcal{T}}\Big|\mathfrak{F}_3 } }
	\\& \leq\ex{ 12^2(np)^4\mathbf{1}_{\mathcal{R}}\ex{\left(\mathbf{1}_{\mathcal{L}_{x,y}}+\mathbf{1}_{\{{d}_{k}^*(x)=0\}}+\mathbf{1}_{\{{d}_{k}^*(y)=0\}}\right) \mathbf{1}_{\mathcal{T}}\Big|\mathfrak{F}_3 } }.\intertext{Now since $x,y \in \Psi_2$ and ${d}_j^*(x),{d}_j^*(y)$ are identically distributed for any $j\geq 0$:}
	\mathfrak{P} &\leq \BO{np}^4\cdot\left(\ex{\ex{\mathbf{1}_{\mathcal{L}_{x,y}}\mathbf{1}_{\mathcal{T}}\Big|\mathfrak{F}_{k+2} } }+ 2\pr{{d}_{k}^*(w)=0\big| \{w \in \Psi_2\}}\right).\intertext{By Lemma \eqref{neighlow} \eqref{itm:bd1}, \eqref{conbd}   and since $\mathcal{T} \in \mathfrak{F}_{k+2}$ we have }  
	\mathfrak{P}&\leq  \BO{np}^4\cdot \left(\ex{\pr{\mathcal{L}_{x,y}\Big|\mathfrak{F}_{k+2} }\mathbf{1}_{\mathcal{T}} }+ 2e^{-4\min\{np,\log n\}}\right)= o\left(e^{-7\min \left\{np,\log n \right\}/2}\right).\end{align*}Recall $	\pr{\mathcal{R}^c}= o\left(e^{-4np}\right)$, so by \eqref{prrnt} and the bound on $\pr{\left\{\cg \notin A_{u,v}^{n,k}\right\}\cap \mathcal{R}\cap \mathcal{T}} $ above,
	\begin{align*}
	\pr{\cg \notin A_{u,v}^{n,k}}&\leq \pr{\left\{\cg \notin A_{u,v}^{n,k}\right\}\cap \mathcal{R}\cap \mathcal{T}}  + \pr{\left(\mathcal{R}\cap \mathcal{T}\right)^c} \leq o\left(e^{-7\min \left\{np,\log n \right\}/2}\right). 
	\end{align*}For $ \pr{(\mathcal{B}_{u,v})^c}$ we apply the coupling inequality \eqref{prk} to the $\psi$ and ${d}^*$-distributions:
	\begin{align*}
	\pr{(\mathcal{B}_{u,v})^c} &\leq  \pr{\psi_1(u)=0 }  +\pr{\psi_1(v)=0 }\leq  2\pr{{d}^*(u)=0 } + 2\pr{\psi_1(u)\neq {d}^*(u)}. \notag \intertext{Then since $\pr{\psi_1(u)\neq {d}^*(u)}$ is known by Lemma \ref{sdom} we have}
	\pr{(\mathcal{B}_{u,v})^c}&\leq 2\pr{{d}^*(u)=0 \left. \right|{d}^*(v)\leq 6np} + 2\pr{{d}^{\phantom{*}}(v) >6np} +2e^{-(1-o(1))np}.\notag \\&\leq 2(1-p)^{n-6np-1}  + o\left(e^{-2np}\right) +e^{-(1-o(1))np} = e^{-(1-o(1))np} \notag,
	\end{align*}by applying Lemma \ref{dom} \eqref{itm:gam1} to the first term and Lemma \ref{neighup} \eqref{itm:bd1} with $\alpha =6$ to the second.\end{proof}

The following crude but resilient bound on $R(i,j)$ is useful when conditioning on $ \ca_{i,j}^c$. 

\begin{lemma}\label{basicres} Let $\mathcal{G}\sim  \mathcal{G}(n,p)$ be such that $np -   \log n  \rightarrow \infty $. Then for $i,j \in V$,
	\[\pru{\mathcal{C}}{R(i,j) > 3\log n /\log(np)}=o(n^{-4}). \] 
\end{lemma}
\begin{proof}
	Since $G\in \mathcal{C}$ the effective resistance between two points is bounded from above by the graph distance. Let $\mathcal{J}_{i,j}:=\left\{|B_k(i)|\cdot |B_k(j)| \geq 4 n\right\}$ where $k:= \left\lceil \log\left(\frac{4n}{15^2}\right)/(2 \log np)\right\rceil +5.$ Using Lemma \ref{neighlow} \eqref{itm:bd3} to bound $\pru{\mathcal{C}}{\mathcal{J}_{i,j}^c}$, since $5 \leq k \leq   \lfloor \log (n) /\log(np) \rfloor -5$ when $n$ large:
	\begin{align*}
	\pru{\mathcal{C}}{R(i,j) > 2k+1}&\leq \pru{\mathcal{C}}{d(i,j)>2k+1\Big|\mathcal{J}_{i,j} } + \pru{\mathcal{C}}{\mathcal{J}_{i,j}^c}\\
	&\leq \pr{xy \notin E,\; \forall  (x,y) \in B_k(i)\times B_k(j), B_k(i)\cap B_k(j)= \emptyset \Big| \mathcal{J}_{i,j} }/\pr{\mathcal{C}} \\ &\qquad\qquad\quad +  2\pru{\mathcal{C}}{|B_k(j)|< 2 \sqrt{n}} \leq 2(1-p)^{4n } + 2\,o\left(n^{-4} \right)=o\left(n^{-4} \right).
	\end{align*}  The result follows since $2k+1 = 2\left(\left\lceil \log\left(\frac{4n}{15^2}\right)/(2 \log np)\right\rceil +5\right) +1 \leq \frac{3 \log n}{\log(np)}$ for large $n$.\end{proof}

\section{Proof of Theorems \ref{resconc}, \ref{resconc2} \& \ref{bolthom}}\label{ResThmSec}
Most of our main theorems result from Corollary \ref{Gnpbdd}, Lemma \ref{supbddlem} below simplifies this application.   

\begin{lemma}\label{supbddlem} Let $\mathcal{G}\sim  \mathcal{G}(n,p)$ where $\log n + \log\log\log n \leq np < o(n^{1/3})$. Let $\alpha\geq 1$ and $\Psi_1(u),\Psi_1(v)$ be defined with respect to $\mbfs(\mathcal{G},\{u,v\}), \; u,v \in V$. Then
	\begin{enumerate}[(i)]\item \label{itm:superr} $\displaystyle \mathbb{E}_{\mathcal{C}}\!\left[\left(\sup_{x\in \Psi_1(u)}\frac{\mathbf{1}_{\mathcal{B}_{u}^{u,v}}}{\varphi(x)} \right)^{\alpha}\right]^{1/\alpha}=  O\!\left(\frac{1}{np}\right)$.
	\item \label{itm:simres2} If $c\log n \leq np \leq n^{1/10}$, for any fixed $c>0$, then \[ \pr{R\left(u,v\right) >\left(\frac{1}{\psi(u)} + \frac{1}{\psi(v)} \right)\!\!\left(1 +  9\epn  \right) }= o\left(e^{-np/4} \right)+o\left(n^{-7/2} \right).\]
	\end{enumerate}

\end{lemma}

\begin{proof}\noindent\textit{Item \eqref{itm:superr}}: Let 	$\mathfrak{P}_u:= \pr{\inf_{x\in \Psi_1(u)}\varphi(x) < K_p}$ and $ K_p:= \left(1 - \sqrt{3}/2\right) np(1-66np^2)$. 
	
	Recall $\Psi_1(u) \subset I_1$ for $u \in I_0$ and observe that,
	\begin{align} 
	\mathfrak{P}_u&\leq \ex{\pr{\inf_{x\in I_1}{d}^*(x) < K_p\big|\mathfrak{F}_1}}  +  \ex{\pr{\varphi(x)\neq {d}^*(x) \text{ for some }x \in I_1 \big|\mathfrak{F}_1}}, \notag \intertext{by the tower property. Applying the union bound since $I_1 \in \mathfrak{F}_1$  yields } 
	\mathfrak{P}_u  &\leq \ex{|I_1|\pr{ {d}^*(x) < K_p \big|x\in I_1,\mathfrak{F}_1}}  +  \ex{|I_1|\pr{\varphi(x)\neq {d}^*(x) \big|x \in I_1,\mathfrak{F}_1}}.\notag
	\intertext{Let $a:= 13/\min\{c,1\}$ where $c>0$ is the largest real number such that  $np \geq  c \log n$. Separate the expectations into parts $\left\{ |I_1| \leq anp\right\}$ and $\left\{ |I_1| >anp\right\}$ to give}
	\mathfrak{P}_u &\leq anp \ex{\pr{ {d}^*(x) < K_p \big|x\in I_1,\mathfrak{F}_1}}  +  anp\pr{\varphi(x)\neq{d}^*(x) \big|x \in I_1} + 2n\pr{|I_1| >anp}.\notag\intertext{Since  ${d}^*(x)\sim Bin(|S_1(x)|,p)$ by Lemma \ref{dom}, $S_1(x) \in \mathfrak{F}_2$, and by Lemma \ref{sdom} \eqref{itm:d1} we have }
	\mathfrak{P}_u  &\leq anp\ex{\pr{Bin(|S_1(x)|,p) < K_p\big|\mathfrak{F}_2 }} + a(np)e^{-(1-o(1))np}+ 4 n\pr{{d}(u) >a np/2} \notag.
	\notag\intertext{Applying Lemma \ref{chb} to the first term and Lemma \ref{neighup} with $\alpha=a$  to the last yields}
	\mathfrak{P}_u &\leq anp\ex{e^{-\left(|S_1(x)|p-K_p\right)^2/\left(2|S_1(X)|p\right) }} +  anp\cdot e^{-(1-o(1))np}+ 4 n\cdot o\left(e^{-anp/6 }\right).\notag  \intertext{Once again by separating the expectation into the two disjoint parts $\left\{ |S_1(x)| \leq n-12(np)^2\right\}$ and $\left\{ |S_1(x)| >n-12(np)^2\right\}$ the applying Lemma \ref{neighup} with $\alpha=6$ we have the following}
	\mathfrak{P}_u &\leq anp\cdot  e^{- np/3 } + 2\pr{|B_2(u)| > 6(np)^2}+ o\left(e^{-np }\right) = o\left(e^{-np/4}\right).	\label{Ptag}
	\end{align}
	Recall $\sup_{x\in \Psi_1(u)}\mathbf{1}_{\mathcal{B}_{u}^{u,v}}/\varphi(x)<1/D$, see \eqref{eq:phi} \& \eqref{eq:psi}. Bernoulli's inequality \eqref{bern} provides
	\begin{align}\label{expectationbdd}
	\mathbb{E}_{\mathcal{C}}\!\left[\left(\sup_{x\in \Psi_1(u)}\frac{\mathbf{1}_{\mathcal{B}_{u}^{u,v}}}{\varphi(x)} \right)^{\alpha}\right]^{1/\alpha}&\leq \left(\frac{1}{(K_p)^\alpha} + \frac{1}{D^\alpha\pr{\cc}}\pr{\inf_{x\in \Psi_1(u)}\varphi(x) < K_p} \right)^{1/\alpha}\notag \\&\leq \frac{1}{(K_p)}\left(1 + (K_p)^\alpha e^{-np/4}/D^\alpha\pr{\cc} \right)^{1/\alpha}=  O\!\left(\frac{1}{np}\right).
	\end{align}
	Note that the bound \eqref{Ptag} on $\mathfrak{P_u}$ holds for any $np \geq c \log n$, $c>0$ fixed. The restriction on $np$ to $np \geq \log n$ comes from \eqref{expectationbdd}, where we need $\pr{\cc}$ bounded below by a constant.\medskip 
	
	\noindent\textit{Item \eqref{itm:simres2}}: Observe that $K_p=\left(1 - \sqrt{3}/2\right) np(1-66np^2) \geq np/9 $ for large $n$ and $k\leq np \epn $ in the definition of event $\mathcal{A}_{u,v}$ \eqref{calA}. Thus  conditional on $\{\varphi(x) \geq K_p \text{ for all }x \in \Psi_1\}\cap\mathcal{A}_{u,v}  $, 
	\[R\left(u,v\right) \leq  \left(1/\psi(u) + 1/\psi(v) \right)\left(1 + (k+2)/K_p \right) \leq \left(1/\psi(u) + 1/\psi(v) \right)\left(1+9\epn\right),\]by Corollary \ref{Gnpbdd}. The result follows since we have $\pr{\inf_{x\in \Psi_1(u)}\varphi(x) < K_p}= o\left(e^{-np/4} \right) $  by \eqref{Ptag} and $\pr{\left(\mathcal{A}_{u,v} \right)^c}=o\left(n^{-7/2}\right)$ by Lemma \ref{setlem}.
\end{proof}

\subsection{Proof of Theorems \ref{resconc} \& \ref{resconc2}}                                          

To begin let $i,j \in V$ and define the following three functions for ease of notation 
\[r_{i,j}:= 1/{d}(i) + 1/{d}(j),  \qquad f_{i,j} :=1/{d}(i)^2 +1/{d}(j)^2,  \qquad g_{i,j}:= \epn \cdot r_{i,j}.\]Define the four events $\ci, \ch$ $\cf$ and $\cl$ as follows, where event $\cl$ is w.r.t. $\mbfs(\cg, \{i,j \} )$,  
\begin{equation}\label{heeq}
\begin{alignedat}{2}
\ci &:= \bigcap_{\{i,j \}\subseteq V}\left\{\left|R(i,j) - \frac{2}{np}\right| \leq  \frac{7\sqrt{\log n}}{(np)^{3/2}} \right\}, \qquad \ch&:= \left\{\left|R(i,j) - r_{i,j}\right| \leq 9\left(2f_{i,j}+ g_{i,j}\right)\right\}, \\
 \cf&:= \left\{ R(i,j)\leq \left(\frac{1}{\psi(i)} + \frac{1}{\psi(j)} \right)\!\! \left(1+ 9\epn \right)\right\},\qquad \cl&:= \bigcap_{z\in \{i,j \}}\left\{ \frac{1}{d^*(z) } \leq \frac{1}{d(z) } + \frac{16}{d(z)^2  } \right\}.
\end{alignedat}
\end{equation}
\begin{lemma}\label{betterprob} Let $c\log n\leq np \leq n^{1/10}  $ where $c>0$. Then $\pr{\cl^c}= o(1/n^3 ) + o(e^{-np/3})$ 
\end{lemma}

\begin{proof}Let $I_0=\{i,j \}$ and recall that $\Gamma^*(i)=\Gamma(i)\backslash(\{j \}\cap \Gamma(j))$, see \eqref{eq:gamdef} and \eqref{second*def}. Let $\ce_a = \{ d^*(v)\leq d(v) -a\}$. For each $z \in \{j \}\cap \Gamma(y)$, provided $z\neq i$, we have $iz \in E$ independently with probability $p$. Thus  $\pr{\ce_a \mid ij\notin E,  d(j)=k} = \pr{Bin(k,p) \geq  a }$ and similarly $  \pr{\ce_a \mid ij\in E, d(j)=k} = \pr{Bin(k-1,p) \geq a -1}$. Thus we have 
	\[\pr{\ce_a} \leq \pr{Bin(k,p) \geq  a -1} + \pr{d(x)\geq k}.\] Let $k= 3a\max\{\log n, np  \} $ and apply Lemma \ref{chb} \eqref{itm:chern2} and \eqref{itm:chern1} respectively to give\begin{align}\label{split}\pr{\ce_a}  &\leq e^{-kp}\left(\frac{ekp}{a-1}\right)^{a-1} +e^{-\frac{k^2}{2(np+k/3)}}\leq \left(\frac{2e(np)^2}{n} \right)^{a-1} + n^{-3a^2/2} =  n^{-a/2}.  \end{align}
	Conditional on the event $\ce_j^c\cap \{d(v) \geq 2a \} $ we have 
	\[\frac{1}{d^*(x) } \leq \frac{1}{d(x) - j }\leq \frac{1}{d(x) } + \frac{j}{d(x)^2 -jd(x)  } \leq \frac{1}{d(x) } + \frac{2j}{d(x)^2 }. \] If we let $a=8$ then $ \pr{d(v)< 16}\leq o(e^{-np/3})$ by Lemma \ref{chb} and $\pr{\ce_8} = o(1/n^3) $  by \eqref{split}. 
\end{proof}

\begin{proof}[Proof of Theorem \ref{resconc}] To begin, by Lemma \ref{lowjoh} we have
	\begin{align*}
	R(i,j) - r_{i,j}&\geq -\left(\frac{1}{{d}(i)^2+{d}(i)} +\frac{1}{{d}(j)^2+{d}(j)}\right)> -\left(\frac{1}{{d}(i)^2} +\frac{1}{{d}(j)^2}\right)=-f_{i,j}. 
	\end{align*} Let $\cb$ be the event $\{\psi(i)  =  {d}^*(i),\psi(j)  =  {d}^*(j)  \}$. Conditional on $\cb\cap\cl\cap \cf$ we have
	\[R(i,j) - r_{i,j} \leq \frac{16}{d(i)^2} + \frac{16}{d(j)^2}+ \left(\frac{1}{d(i)} + \frac{1}{d(j)}+ \frac{16}{d(i)^2} + \frac{16}{d(j)^2} \right)9\eps_n \leq 9(2 f_{i,j}+g_{i,j}).\] 	
	Bounding $\pr{\cb^c}$, $\pr{\cl^c}$ and $\pr{\cf^c}$ using Lemmas \ref{sdom} \eqref{itm:d2}, \ref{betterprob},  and \ref{ubredux} \eqref{itm:simres2} respectively:\begin{align*}\pr{\ch^c}  &\leq  2 e^{-(1-o(1))np} +o(1/n^3 ) + o(e^{-np/3})+  o\!\left(e^{-np/4} \right)+o\!\left(n^{-7/2} \right), \end{align*}which is $o\!\left(e^{-np/4} \right)+ o\!\left(n^{-3}\right) $ as required.\end{proof}

For $S\subseteq V$ and $\lambda:=\lambda(n)=o(np)$ let $ \ce(S,\lambda)$ be the event\begin{align}\label{eventelam}
\ce(S,\lambda):=\bigcap_{u\in S}\left\{\left|{d}(u)- np\right| \leq  \sqrt{\lambda np }\right\}, \end{align}
for which we have $\pr{\ce(S,\lambda)^c}  \leq 2|S|e^{-\lambda(n)/3} $ by the union bound and Lemma \ref{chb}.

\begin{proof}[Proof of Theorem \ref{resconc2}]
	\noindent \textit{Item \eqref{itm:co2}}: conditional on the event $\ce\left(\{i,j\},\epn^2 np/19\right) \cap \ch$ we have
	\begin{align}\label{triandge}
	\Big|R(i,j)-\frac{2}{np}\Big|&\leq \Big|R(i,j)-r_{i,j}\Big| +\Big|r_{i,j}- \frac{2}{np}\Big|\leq \frac{19\epn }{2np}+\frac{2\sqrt{\lambda(n)}}{(np)^{3/2}}\leq \frac{10\epn }{np},
	\end{align}
	since  $9\left(2f_{i,j}+ g_{i,j}\right) \leq 19\epn /\left(2np\right)$ on $\ce(\{i,j\},\lambda) \cap \ch$. Thus by Theorem \ref{resconc}: \begin{align*}
	\pr{\left|R(i,j) - 2/(np)\right| > 10\epn /(np)} &\leq \pr{\left(\ch\cap\ce\right)^c} \leq \pr{\ch^c}+ 4e^{-\lambda(n)/3}\leq e^{-\epn^2 np/60}.                                                                                                                                                        
	\end{align*}\noindent \textit{Item \eqref{itm:co3}}: Recall the definition of $\ch $ from \eqref{heeq} and notice we supressed dependence on $i,j$ that is $\ch :=\ch_{i,j}$. Similarly to \eqref{triandge}, conditional on $\left(\bigcap_{\{i,j\} \subseteq v }\ch_{i,j}\right)\cap \ce\left(V,9\log n\right) $ we have 
	\[\left|R(i,j)- \frac{2}{np}\right|\leq 2\cdot \frac{3\sqrt{\log n}}{(np)^{3/2}} +  \frac{19\epn }{2np} \leq \frac{7\sqrt{\log n}}{(np)^{3/2}}.\] 
	Recall event $\ci$ from \eqref{heeq}. The result now follows since by Theorem \ref{resconc} and \eqref{eventelam} we have  \[\pr{\ci^c}\leq  n^2\left( o\!\left(e^{-np/4} \right)+ o\!\left(n^{-3}\right)\right) + 2ne^{-3\log n} = o\left(\frac{1}{n}\right) .\]
	\noindent \textit{Item \eqref{itm:co4}}: Recall that $m:=|E|$ and let $\mathcal{M}$ be the event $\left\{\left|m- \binom{n}{2}p\right| \leq  3\sqrt{\log(n)\binom{n}{2}p }  \right\}.$

	Conditional on $\ce\left(V,9\log n \right) \cap \mathcal{M} \cap \mathcal{I}$ we have the following for any $\{i,j\} \subseteq V$,\begin{align*}
	|mR(i, j) -n| \leq 4n\sqrt{\frac{\log n  }{np}}, \quad \text{and} \quad  \left|\sum\limits_{u\in V} \frac{{d}(u)}{2}\left[R(j,u)-R(u, i)\right]\right|\leq 8n\sqrt{\frac{\log n  }{np}},
	\end{align*} thus $	\left| h(i, j) -n \right|\leq 12 \sqrt{\log(n) /np } $ by Tetali's formula \eqref{eq:hit} and the Triangle inequality. Now 
	\[\pr{\left(\ce\left(V,9\log n \right) \cap \mathcal{M} \cap \mathcal{I}\right)^c}=  o\left(1/n^3\right) +  o\left(1/n^3\right) +o\left(1/n\right) = o\left(1/n\right) \]  
	by \eqref{eventelam}, Lemma \ref{chb}, since $m\sim Bin\left(\binom{n}{2},p \right)$  and Theorem \ref{resconc2} \eqref{itm:co3} respectively. \end{proof}

\subsection{Proof of Theorem \ref{bolthom}}

Recall that $paths_2(i,j,l)$ is the maximum number of paths of length at most $l$ between vertices $i$ and $j$ that are vertex disjoint on $V\backslash \left(B_1(i)\cup B_1(j)\right)$ of a graph $G$.

\begin{proof}[Proof of Theorem \ref{bolthom}]\textit{Item \eqref{itm:bolthom1}}: For $i,j \in V$ let $\ce_{i,j}$ be the event that there is no path from $i$ to $j$ of length less than $4$. Then by over-counting the number of paths we have
	\begin{equation}\label{E}
	\pr{\ce_{i,j}^c  }\leq \sum_{l=1}^3 \pr{\text{path from }i\text{ to }j\text{ of length }l  }\leq p + (n-2)p^2 +\binom{n-2}{2}p^3 \leq n^2p^3.  
	\end{equation}
	Conditional on $\ce_{i,j}$ every path between $i$ and $j$ must pass through at least one vertex from each of ${d}_2(i)$ and ${d}_2(j)$, though these vertices may not be distinct. So there cannot be more than $\min\{{d}_2(i),{d}_2(j) \}$ paths between $i,j\in V$ which are vertex disjoint on $V^*:= V\backslash \left( B_1(i)\cup B_1(j)\right)$ since $\Gamma_2(i) \cup \Gamma_2(j) \subseteq V^*$. Thus conditional on $\ce_{i,j}$ for any $l\geq 0$ we have 
	\begin{equation} \label{pathsE} paths_2(i,j,l) \leq \min\{{d}_2(i),{d}_2(j) \}. \end{equation}
	To bound $paths_2(i,j,l)$ from below we construct $\min\left\{\psi_2(i), \psi_2(j) \right\}$ vertex disjoint paths between $i$ and $j$ conditional on $\ca_{i,j}$, Definition \ref{A}, then couple $\psi_2(i)$ to $d_2(i)$ and $\psi_2(j)$ to $d_2(j)$. 
	
	For the path construction condition on $\mathcal{A}_{i,j}$ and w.l.o.g. assume $\psi_2(i) \leq \psi_2(j)$. Take any subset $\Psi_2(j)^*\subseteq \Psi_2(j)$ with $\psi_2(i)$ elements and any bijection $M$ between $\Psi_2(i)$ and $\Psi_2(j)^*$. Given any pair $(x,y)$ in $M$, conditional on $\mathcal{A}_{i,j} $, there is some $k$ and some pair $(x_k,y_k) \in \Gamma_k^*(x)\times \Gamma_k^*(y)$ such that $x_ky_k \in E$. We define the path $P_{x,y}:=i,i_x,x,x_1,\dots,x_k,y_k,y_{k-1},\dots,y,j_y,j$, where $x,x_1,\dots,x_k$ is the unique path from $x$ to $x_k$ in the tree $T_k(x):= \cup_{i=0}^{k}\Gamma_i^*(x)$ and $i_x$ is the unique vertex in $\Gamma_1^*(i)$ connected to $x$. The equivalent descriptions hold for $y,y_1,\dots,y_k \in T_k(y)$ and $j_y \in \Gamma_1(j)$ with respect to $y$ and $j$. The paths $\{P_{x,y} \}_{(x,y) \in M}$ are all vertex disjoint on $V^*$ since the trees $\{T_{K}(u) \}_{u \in \Psi_2}$ are all vertex disjoint. Each path in $P_{i,j}$ has length $l:=2k+5$  where the $k$ is given by the event $\mathcal{A}_{i,j}$. Thus conditional on the event $\mathcal{A}_{i,j}$ we have 
	\begin{equation} \label{pathsA} paths_2(i,j,l) \geq \left| \{P_{x,y} \}_{(x,y) \in M}\right| = \min\left\{\psi_2(i), \psi_2(j) \right\}.\end{equation}  Exchanging the $\psi_2$ and ${d}_2$ distributions on the event $\{\psi_2(i) \neq {d}_2(i) \text{ or } \psi_2(j) \neq {d}_2(j) \}$ yields
	\begin{align*}
	\mathfrak{P}&:=\pr{paths_2(i,j,l)\neq  \min\{{d}_2(i),{d}_2(j)\}} \leq \pr{\psi_2(i) \neq {d}_2(i) \text{ or } \psi_2(j) \neq {d}_2(j) } \\ &\qquad+\pr{paths_2(i,j,l)<  \min\{\psi_2(i),\psi_2(j)\}} +\pr{\left\{paths_2(i,j,l)>  \min\{{d}_2(i),{d}_2(j)\}\right\}}.\intertext{Now by \eqref{pathsA} and \eqref{pathsE} we have the following}  \mathfrak{P}&\leq \pr{\psi_2(i) \neq {d}_2(i) \text{ or } \psi_2(j) \neq {d}_2(j) }  +  \pr{( \mathcal{A}_{i,j})^c} + \pr{\ce_{i,j}^c} \leq 5n^3p^4 + o\left(e^{-7\min \left\{np,\log n \right\}/2}\right), 
	\end{align*}by Lemma \ref{sdom} \eqref{itm:d4}, Lemma \ref{setlem} and \eqref{E} respectively. On the event $\mathcal{A}_{i,j}$ the strong path property is satisfied for some $k\leq \lfloor\tfrac{\log n}{2\log (np)}\rfloor +2$, thus  $l=2k+5 \leq \tfrac{\log n}{\log (np)} +9$.\medskip 	
	
	\noindent\textit{Item \eqref{itm:bolthom2}}: observe that ${d}_2(u)\sim Bin\left(n-1-{d}(u), 1-(1-p)^{{d}(u)}\right)$, conditional on ${d}(u)$ for any $u \in V$. Notice that $(1-p)^k \leq 1- kp + (kp)^2$ when $(kp)^i\geq (kp)^{i+1}$ for all $i$ by the Bernoulli inequality \eqref{bern}. Thus conditional on $\ce(\{i,j\},3\log(np))$, see \eqref{eventelam}, we have the following
	\begin{align*}
	Bin\left(n-2np, np^2 - 2p\sqrt{\log (np)np}\right)\preceq_1 {d}_2(i),{d}_2(j) \preceq_1 Bin\left(n, np^2 + p\sqrt{3\log (np)np}\right).  
	\end{align*}
	Let $\mathcal{R}_{i,j}$ be the event $\left\{ \left|\min\{{d}_2(i),{d}_2(j)\}-(np)^2\right|\leq  3(np)^{3/2}\sqrt{\log np}\right\}$. Observe that we have  
	\begin{align}\label{rrrtag}
	\pr{\mathcal{R}_{i,j}^c}\leq \pr{\mathcal{R}_{i,j}^c\big|\ce(\{i,j\},3\log(np))}  +\pr{\ce(\{i,j\},3\log(np))^c } = o\left(1/np\right),
	\end{align}by \eqref{eventelam} and applying  Chernoff bounds to ${d}_2(i)$ conditional on $\ce(\{i,j\},3\log(np))$. 	We now have 
	\begin{align*}& \pr{\left|paths_2(i,j,l) -(np)^2\right|> 3(np)^{3/2}\sqrt{\log np}} \leq  \pr{paths_2(i,j,l) \neq  \min\{{d}_2(i),{d}_2(j) \}}+\pr{\mathcal{R}_{i,j}^c} \\ &\qquad \leq 5n^3p^4 + o\left(e^{-7\min \left\{np,\log n \right\}/2}\right) + o\left(1/np\right)= o\left(1/np\right),    \end{align*} by Item \eqref{itm:bolthom1} and the bound on $\pr{\mathcal{R}_{i,j}^c}$ from \eqref{rrrtag}. \end{proof}

\section{Proof of Theorems \ref{exthm}, \ref{concentrationnew}, \ref{exthm2} \& \ref{concentration}}\label{FinalSec}
 Recall $\varepsilon_n : = \frac{\log  n}{ np\log (np)}$ from \eqref{edef}, that $m=|E|$ and Tetali's formula \eqref{eq:hit}, which is given by \begin{align*}
 h(i, j) &=  m R(i, j) + \sum\limits_{u\in V} \frac{{d}(u)}{2}\left[R(j,u)-R(u, i)\right].
 \end{align*}

Our results on hitting times and other random walk indices come from applying our bounds on resistance to Tetali's formula \eqref{eq:hit} to obtain moments hitting times. The following two Lemmas help us calculate the terms arising during these computations.
\begin{lemma}\label{ubredux}Let $\mathcal{G}\sim  \mathcal{G}(n,p)$ where $\log n + \log\log\log n \leq np < o(n^{1/3})$. Let $\alpha\geq 1$ and $\Psi_1(u),\Psi_1(v)$ be defined with respect to $\mbfs(\mathcal{G},\{u,v\}), \; u,v \in V$. Then
	\[ \mathbb{E}_\mathcal{C}\left[\frac{\mathbf{1}_{\mathcal{B}_{u}^{u,v}}}{\psi(u)^\alpha} \right]^{1/\alpha}= \frac{ 1 + \BO{\eps_n}}{np}.\]
\end{lemma}

\begin{proof}We restrict to the event $\cb_{u}^{u,v}$ to ensure the expectation is bounded, 
	\begin{align*}
	\mathfrak{E}&:= \mathbb{E}_\mathcal{C}\left[ \frac{\mathbf{1}_{\mathcal{B}_{u}^{u,v}}}{\psi(u)^\alpha}\right]= \sum\limits_{k=1}^{n}  \frac{1}{k^\alpha}\pru{\mathcal{C}}{\psi(u) =k } = \sum\limits_{k=1}^{n}  \frac{1}{k^\alpha}\frac{\pr{\{\psi(u) =k \}\cap \mathcal{C}}}{\pr{\mathcal{C}}}. \notag \intertext{Applying the coupling inequality \eqref{prk}, and then Lemma \ref{sdom} to bound $\pr{{d}^*(u) \neq \psi(u)}$ gives} \mathfrak{E} &\leq \sum\limits_{k=1}^{n}  \frac{1}{k^\alpha}\frac{\pr{{d}^*(u) =k} +\pr{{d}^*(u) \neq \psi(u)} }{\pr{\mathcal{C}}} =\sum\limits_{k=1}^{n}  \frac{1}{k^\alpha}\frac{\pr{{d}^*(u) =k} +e^{-(1-o(1))np} }{\pr{\mathcal{C}}}.\notag \intertext{ Let $\tilde{{d}}_1(v):= \left|\Gamma_1(v) \cap S_0  \right| \sim  Bin(n-2,p)$. By Lemma \ref{dom}  we have ${d}^*(u) \sim  Bin(n-2-h,p)$ conditional on $\{\tilde{{d}}_1(v)=h \}$. By the law of total expectation and the generalised harmonic series, }
	\mathfrak{E}&\leq \sum\limits_{k=1}^{n}  \frac{1}{k^\alpha}\frac{\sum\limits_{h=0}^{n-2}\pr{{d}^*(u) =k\Big|\tilde{{d}}_1(v) =h}\pr{\tilde{{d}}_1(v) =h}  }{\pr{\mathcal{C}}} +  O\!\left(\frac{(\log n )e^{-(1-o(1))np} }{\pr{\mathcal{C}}}\right).\notag\intertext{Now by writing out $\pr{{d}^*(u) =k\Big|\tilde{{d}}_1(v) =h}\pr{\tilde{{d}}_1(v) =h}$ explicitly we have  }
	\mathfrak{E}&\leq \sum\limits_{k=1}^{n}\frac{1}{k^\alpha} \frac{\sum\limits_{h=0}^{n-2} \binom{n-2-h}{k}p^k(1-p)^{n-2-h-k}\cdot\binom{n-2}{h}p^h(1-p)^{n-2-h}}{\pr{\mathcal{C}}}+  e^{-(1-o(1))np}\notag\\
	&= \sum\limits_{h=0}^{n-3}\binom{n-2}{h}\frac{p^h(1-p)^{n-2-h}}{\pr{\mathcal{C}}} \left(\sum\limits_{k=1}^{n-2-h}  \frac{1}{k^\alpha}\binom{n-2-h}{k}p^k(1-p)^{n-2-h-k}\right) +  e^{-(1-o(1))np}.\notag
	\intertext{Applying Proposition \ref{quo} to the bracketed sum above where we let $B_h$ be a random variable with distribution $Bin(n-h-3,p)$ yields}
	\mathfrak{E}&\leq \frac{np }{\pr{\mathcal{C}}} \sum\limits_{h=0}^{n-3}\binom{n-2}{h}p^h(1-p)^{n-2-h}\ex{\frac{1}{\left(B_h+1\right)^{\alpha+1}}}+e^{-(1-o(1))np}.\notag
	\intertext{The weight in front of the expectation term is the density of a $Bin(n-2,p)$ random variable. Split the sum at $t:=\sqrt{3np(\alpha+2)\log(np)}$ and bound the expectation to give}
	\mathfrak{E}&\leq \frac{np }{\pr{\mathcal{C}}} \left(\pr{Bin(n-2,p)\leq t}\ex{\frac{1}{\left(B_t+1\right)^{\alpha+1}}}+ \pr{Bin(n-2,p)> t}\right)+e^{-(1-o(1))np}.\notag\intertext{Bounding $\pr{Bin(n-2,p)> t}$ by Lemma \ref{chb} using Lemma \ref{bmoment} to calculate $\ex{\frac{1}{\left(B_t+1\right)^{\alpha+1}}}$:}
	\mathfrak{E}&\leq \frac{np }{\pr{\mathcal{C}}} \left[ \left(\frac{1}{\left((n-t-3)p\right)^{\alpha+1}} + O\!\left(\frac{1}{\left((n-t-3)p\right)^{\alpha+2}}\right) \right) +  o\left(\frac{1}{(np)^{\alpha+2}}  \right)\right]+ e^{-(1-o(1))np}\notag\\
	&=\frac{1}{\pr{\mathcal{C}}}\left(\frac{1}{(np)^\alpha} +  O\!\left(\frac{1}{(np)^{\alpha+1}} \right) \right) 
	\end{align*}Applying Bernoulli's inequality \eqref{bern} yields
	\begin{align*}
	\mathfrak{E}^{1/\alpha} &\leq \frac{1}{\pr{\mathcal{C}}}\left(\frac{1}{(np)^\alpha} +  O\!\left(\frac{1}{(np)^{\alpha+1}} \right) \right)^{1/\alpha} = \frac{1}{1-\pr{\mathcal{C}^c} np}\left(1+ O\!\left(\frac{1}{np } \right)\right)^{1/\alpha}= \frac{ 1 + \BO{\eps_n}}{np},	\end{align*}
	as \eqref{eq:probC} gives  $\pr{\cc^c}\leq O\!\left(\epn \right) $ whenever $np \geq \log n + \log\log\log n$.
\end{proof}

\begin{lemma}\label{keylem}
For any set $A\subset V$ of size $0\leq a\leq 3$ and any set of vertex pairs $B\subset \binom{V}{2}$ of size $0\leq b\leq 3$ then 
\[\exu{\mathcal{C}}{\left(\prod_{v \in A}{d}(v) \right) \left(\prod_{\{x,y \} \in B}R(x,y)\right)} =\frac{2^b}{(np)^{a-b}}\left(1 \pm  \BO{\epn} \right).    \]  
\end{lemma}

\begin{proof}
We shall prove the case $A= \{u,v,w \}$ and $B=\{(a_1,a_2),(b_1,b_2),(c_1,c_1)  \}  $, this is the ``largest'' case and the other cases are proved in exactly the same way. Let  \[\ce :=\mathcal{A}_{a_1,a_2}^{n}\cap\mathcal{A}_{b_1,b_2}^{n}\cap\mathcal{A}_{c_1,c_2}^{n} \cap\mathcal{B}_{a_1,a_2}\cap\mathcal{B}_{b_1,b_2}\cap\mathcal{B}_{c_1,c_2}.
\]
For ease of notation we define
\[ \mathsf{Deg}(A) := \prod_{v \in A}{d}(v) \qquad \text{and} \qquad \mathsf{Res}(B);= \prod_{\{x,y \} \in B}R(x,y).\] 
Recall the bound on $R(x,y)$ from Corollary \ref{Gnpbdd}, conditional on $\ca_{x,y}$, this yields \begin{align}
\mathsf{Res}(B)\mathbf{1}_{\ce}&\leq \prod\limits_{\{x,y\} \in B}\left( \!\frac{1}{\psi(x)} + \frac{1}{\psi(y)} + \frac{k+2}{\psi(x)}\sup_{a\in \Psi_1(x)}\frac{1}{\varphi(a)}  +  \frac{k+2}{\psi(y)}\sup_{b\in \Psi_1(y)}\frac{1}{\varphi(b)} \right)\mathbf{1}_{\ce}\label{combo} \\
&=\sum_{\substack{x\in\{a_1,a_2\}\\ y\in\{b_1,b_2\}\\ z\in \{c_1,c_2 \}}  }\left(a_{x,y,z} + \sum\limits_{\substack{f,g,h \in\\ \{x,y,z \}\\ f \neq g\neq h }}\left[ (k+2)\cdot b_{f,g,h} +  (k+2)^2\cdot c_{f,g,h}\right] +  (k+2)^3\cdot d_{x,y,z}\right) \notag
\end{align}
where the summands are given by  
\[\begin{alignedat}{2}
a_{x,y,z} & = \frac{\mathbf{1}_{\ce}}{\psi(x)\psi(y)\psi(z)}, \qquad &&b_{f,g,h}= \frac{\mathbf{1}_{\ce}}{\psi(f)\psi(g)\psi(h)}\sup_{a\in \Psi_1(h)}\frac{1}{\varphi(a)},\\
c_{f,g,h}& = \frac{\mathbf{1}_{\ce}}{\psi(f)}\prod_{w\in \{g,h \}}\frac{1}{\psi(w)}\sup_{a\in \Psi_1(w)}\frac{1}{\varphi(a)},  \qquad \qquad &&d_{x,y,z}= \mathbf{1}_{\ce}\prod_{w\in \{f,g,h \}}\frac{1}{\psi(w)}\sup_{a\in \Psi_1(w)}\frac{1}{\varphi(a)}. 
\end{alignedat}\]

By  H\"older's inequality \eqref{holder}, it follows that  $\exu{\cc}{\mathsf{Deg}(A) \cdot  a_{x,y,z}\cdot  \mathbf{1}_{\ce} }$ is at most \begin{align*}
& \exu{\cc}{{d}(u)^6}^{\frac{1}{6}}  \exu{\cc}{{d}(v)^6}^{\frac{1}{6}} \exu{\cc}{{d}(w)^6}^{\frac{1}{6}} \exu{\cc}{\frac{\mathbf{1}_{\ce}}{\psi(x)^6}}^{\frac{1}{6}}\exu{\cc}{\frac{\mathbf{1}_{\ce}}{\psi(y)^6}}^{\frac{1}{6}}\exu{\cc}{\frac{\mathbf{1}_{\ce}}{\psi(z)^6}}^{\frac{1}{6}}\\
&\qquad= \left((np)^6 + \BO{(np)^5} \right)^{\frac{1}{2}}\cdot \left(\frac{1+\BO{\epn}}{np}\right)^3 = 1 + \BO{\epn} ,
\end{align*}
where we applied \eqref{dmom} and Lemma \ref{ubredux} to the expectations, then Bernoulli's inequality \eqref{bern}. 

Similarly by H\"older's inequality \eqref{holder} and collecting similar terms
\begin{align*}
\exu{\cc}{\mathsf{Deg}(A) \cdot  b_{f,g,h}\cdot  \mathbf{1}_{\ce} } &\leq \exu{\cc}{{d}(u)^7}^{\frac{3}{7}}\exu{\cc}{\frac{\mathbf{1}_{\cc_1}}{\psi(f)^7}}^{\frac{3}{7}} \exu{\cc}{\sup_{c\in \Psi_1(h)}\frac{\mathbf{1}_{\cc_1}}{\varphi(c)^7}}^{\frac{1}{7}} \\
&=  \left((np)^7 + \BO{(np)^6} \right)^{\frac{3}{7}}\cdot \left(\frac{1+\BO{\epn}}{np}\right)^3 \cdot \BO{\frac{1}{np}}= \BO{\frac{1}{np}}.
\end{align*} 
where in addition we applied Lemma \ref{supbddlem}. By a near identical calculation we have
\begin{align*}
\exu{\cc}{\mathsf{Deg}(A) \cdot  c_{f,g,h}\cdot  \mathbf{1}_{\ce} } & = \BO{\frac{1}{(np)^2}}, \qquad 
\exu{\cc}{\mathsf{Deg}(A) \cdot  d_{x,y,z}\cdot  \mathbf{1}_{\ce} }  = \BO{\frac{1}{(np)^3}}.
\end{align*}
Now by linearity of expectation, \eqref{combo}, and  since $k= \BO{\log (n)/\log (np) } $, we have 
\begin{equation}\label{emain}\exu{\cc}{\mathsf{Deg}(A) \cdot \mathsf{Res}(B)\cdot  \mathbf{1}_{\ce} } = 2^3 + \BO{\epn} + \BO{\frac{k}{np} }  + \BO{\frac{k^2}{(np)^2} }  + \BO{\frac{k^3}{(np)^3} }= 2^3 + \BO{\epn}. \end{equation}

We shall now consider what happens on $\ce^c$, let $ \mathcal{M}$ be the event $ \bigcap_{u \in A}\left\{ {d}(u)\leq 8np\right\}$. By Chernoff bounds Lemma \ref{chb} and the bound \eqref{eq:probC} on $\pr{\cc}$ we have $\mathbb{P}_{\mathcal{\cc}}\left(\mathcal{M}^c\right)= o(1/n^{7}).$ Let $\mathcal{S}_{i,j}$ be the event $\left\{R(i,j) \leq  3\log n /\log(np)\right\}$ and recall $\pru{\cc}{\mathcal{S}_{i,j}^c}= o(n^{-4}) $ by  Lemma \ref{basicres}.

Observe that conditional on $\tilde{\ce}_1= \ce^c\cap\mathcal{M}\cap\prod_{\{x,y \} \in B}\mathcal{S}_{x,y}$ the following inequalities hold for all $v \in A$ and $\{x,y \}\in B$: ${d}(u) \leq 8np $ and  $R(x,y)\leq 3\log (n )/\log(np) $. Thus
\begin{equation}\label{em1}\exu{\mathcal{C}}{\mathsf{Deg}(A) \cdot \mathsf{Res}(B)\cdot  \mathbf{1}_{\tilde{\ce}_1 }} = \BO{(np)^3\cdot \frac{(\log n)^3}{\log(np)} }\cdot \pru{\cc}{\ce^c}  = o\left(1/n^{4/5}\right). \end{equation}
We shall now consider conditioning on the event $\tilde{\ce}_2= \ce^c\cap\mathcal{M}\cap\left(\prod_{\{x,y \} \in B}\mathcal{S}_{x,y}\right)^c$ where we instead use the worse case resistance bound $R(i,j)\leq n-1 $, this gives 
\begin{equation}\label{ec2}\exu{\mathcal{C}}{\mathsf{Deg}(A) \cdot \mathsf{Res}(B)\cdot  \mathbf{1}_{\tilde{\ce}_2 }} = \BO{(np)^3\cdot n^3}\cdot \pru{\cc}{\mathcal{S}_{x,y}^c}  = o\left(1/n^{4/5}\right). \end{equation} 
Finally we consider the event $\ce^c \cap \mathcal{M}^c$ and we observe that since $\mathbb{P}_{\mathcal{\cc}}\left(\mathcal{M}^c\right)= o(1/n^{7})$ we have  
\begin{equation}\label{ecmc}
\exu{\mathcal{C}}{\mathsf{Deg}(A) \cdot \mathsf{Res}(B)\cdot  \mathbf{1}_{\ce^c \cap \cm^c }} = \BO{n^6}\cdot \pru{\cc}{\mathcal{M}^c}  = o\left(1/n\right). \end{equation} 

The upper bound on $\exu{\mathcal{C}}{\mathsf{Deg}(A) \cdot \mathsf{Res}(B)}$ follows by combining \eqref{emain}, \eqref{em1}, \eqref{ec2} and \eqref{ecmc}.\medskip

We now consider the lower bound. Recall that $B=\{(a_1,a_2),(b_1,b_2) (c_1,c_2) \} $, in the case we are considering. Lemma \ref{lowjoh} states $R(x,y)\geq 1/({d}(x)+1)+ 1/({d}(y)+1)$, thus  
\begin{align}&\exu{\mathcal{C}}{\mathsf{Deg}(A) \cdot \mathsf{Res}(B)}\geq\sum_{
x,y,z}\exu{\mathcal{C}}{\frac{ \prod_{u\in A}{d}(u)}{\left( {d}(x)+1\right)\left( {d}(y)+1\right)\left( {d}(z)+1\right)}} \label{sumlowbdd},
\end{align}where the sum is over $(x,y,z)\in \{a_1,a_2 \}\times  \{b_1,b_2 \}\times \{c_1,c_2 \}$.  Let $ \mathcal{D}$ be the event given by
\[ \mathcal{D}:=\left( \prod_{u\in A} \left\{{d}(u) \geq np -a\sqrt{np}\right\}\right)\bigcap\left( \prod_{\{x,y\}\in B  }\left\{{d}(x) ,{d}(y) \leq np +a\sqrt{np} \right\}\right),\]
 where $a= 3\sqrt{\log\log n}$ if $np = O(\log n)$ and $a=3\sqrt{\log n}$ if $np = \omega(\log n)$. Then, 
\begin{align*}\exu{\mathcal{C}}{\frac{ \prod_{u\in A}{d}(u)}{\left( {d}(x)+1\right)\left( {d}(y)+1\right)\left( {d}(z)+1\right)}} & \geq \frac{\left(np -a\sqrt{np}\right)^3}{\left(np+a\sqrt{np}\right)^3}\pru{\cc}{\cd} =1 - O\!\left(\epn \right),
\end{align*}where the bound on $\pru{\cc}{\cd}$ is by Lemma \ref{chb}. The lower bound follows from \eqref{sumlowbdd}.\end{proof}

\subsection{Proof of Theorem \ref{exthm}}\label{PFEXT}
Equipped with Lemma \ref{keylem}, the proofs of the main ``moment theorems'' are straightforward. 

\begin{proof}[Proof of Theorem \ref{exthm}] Observe that $\exu{\cc}{R(i,j)}=\left(2\pm O\!\left(\varepsilon_n \right)\right)/np$
follows directly from Lemma \ref{keylem} with $A=\emptyset$ and $B=\{(i,j) \}$. For hitting times  we have the following by \eqref{eq:hit}:
\begin{align*}
\exu{\mathcal{C}}{h(i,j)}&= \exu{\mathcal{C}}{mR(i,j)}+ \frac{1}{2}\sum\limits_{u\in V} \left(\mathbb{E}_{\mathcal{C}} \!\left[ d(u)R(u, j)\right] -\mathbb{E}_{\mathcal{C}} \!\left[ d(u) R(u, i)\right]\right)=\exu{\mathcal{C}}{mR(i,j)}, 
\end{align*} when $i\neq j$, by symmetry. Thus, we have 
\[	\exu{\mathcal{C}}{h(i,j)} = \frac{1}{2}\sum_{u\in V} \exu{\mathcal{C}}{{d}(u)R(i,j)} = \frac{1}{2}\sum_{u\in V} (2\pm\BO{\eps_n}) =n\left(1\pm O\!\left(\varepsilon_n \right)\right), \] by Lemma \ref{keylem} with $A=\{u\}$ and $B=\{(i,j) \}$. \end{proof}

\subsection{Proof of Theorem \ref{exthm2}}   \label{PFEXT2}

Theorems \ref{concentrationnew} \& \ref{concentration} shall be proved by Chebychev's inequality, thus we need second moments.
\begin{lemma}\label{h^2}
	Let $\cg \sim \cg(,n,p)$ satisfy \eqref{sparsecon} and $i,j \in V(\cg)$ where $i\neq j$. Then $\exu{\cc}{h(i,j)^2}=\left(1\pm O\!\left(\varepsilon_n \right)\right)n^2,$ $\exu{\cc}{cc_i(\cg)^2}=\left(1\pm O\!\left(\varepsilon_n \right)\right)n^2$ and $\exu{\cc}{K(\mathcal{G})^2}=\left(1\pm O\!\left(\varepsilon_n \right)\right)n^2/p^2.$
\end{lemma}
\begin{proof} Let $g(a,b,c,d):=\exu{\mathcal{C}}{{d}(u){d}(v)R(a,b)R(c,d)}$. Using Tetali's formula \eqref{eq:hit} we can expand $\exu{\mathcal{C}}{  h(i, j)h(i,a)}$ to give the following for any $i,j,a \in V$:
	
	\begin{align}&\mathbb{E}_\mathcal{C}\left[ \left(\sum\limits_{u\in V\phantom{v}} \frac{{d}(u)}{2}\left( R(i,j) + R(j,u) - R(u,i)\right)\!\right) \cdot\left(\sum\limits_{v\in V} \frac{{d}(v)}{2}\left( R(i,a) + R(a,v) - R(v,i)\right)\right)\right] \label{eq:hittet}\notag \\
	&=\frac{1}{4} \sum\limits_{u,v \in V}\!\left( g(i,j,i,a)  +  \sum\limits_{\substack{(w,z) \in \\ \{(u,i),(j,a)\}}} \!\!\!g(i,w,v,z)  - \sum\limits_{w \in \{i,u\}} g(w,j,i,v)\right)\notag \\& \quad+  \frac{1}{4}\!\sum\limits_{u,v \in V}\sum\limits_{w \in \{i,v\}}\!\!\left(g(u,j,w,a)- g(w,a,i,u)\right)= \frac{1}{4} \!\sum_{u,v \in V}\!\! \exu{\mathcal{C}}{{d}(u){d}(v)R(i,j)R(i,a)}.
	\end{align} 
	To see the above, observe that $R(a,b)R(c,d)=0$ if and only $a=b$ or $c=d$. Thus only the first term, $g(i,j,i,a)$, will always be non-zero. All the other terms contain one or more input from $\{u,v\}$ so will be zero at different times. Of the eight other terms there are two positive and two negative terms containing one of $\{u,v\}$, then two positive and two negative terms containing both $u$ and  $v$ as inputs. Thus by symmetry when the sums are expanded everything apart from the first term $g(i,j,i,a)$ cancels. 
	
	Thus by \eqref{eq:hittet} and  Lemma \ref{keylem} with $A=\{u,v\}$ and $B=\{(i,j), (i,a) \}$ we have 
	\begin{equation}\label{hitsqu}\exu{\mathcal{C}}{  h(i, j)h(i,a)} = \frac{1}{4} \sum_{u,v \in V} (4\pm \BO{\eps_n} )=n^2\left(1\pm O\!\left(\varepsilon_n \right)\right).\end{equation} Now by the definition \eqref{Kirchoff} of $cc_i(G)$ and \eqref{hitsqu} we have,
	\begin{align*}
	\exu{\mathcal{C}}{ cc_i(\mathcal{G})^2}&= \frac{1}{(n-1)^2}\exu{\cc}{\left(\sum\limits_{j \in V}h(i,j)\right)^2}= \frac{1}{(n-1)^2}\sum\limits_{j,k \in V; j,k \neq i}\exu{\mathcal{C}}{ h(i, j)h(i,k)},   
	\end{align*}which is equal to $\left(1\pm O\!\left(\varepsilon_n \right)\right)n^2$. Finally observe that by \eqref{Kirchoff} we have
	\begin{equation*}\exu{\mathcal{C}}{K(\mathcal{G})^2} = \sum\limits_{\{i,j\}\subseteq V }\sum\limits_{\{w,z\}\subseteq V }\exu{\mathcal{C}}{R(i,j)R(w,z)} = \sum\limits_{\{i,j\}\subseteq V }\sum\limits_{\{w,z\}\subseteq V }\frac{4\pm\BO{\eps_n}}{(np)^2} = \frac{n^2}{p^2}\left(1\pm O\!\left(\varepsilon_n \right)\right), \end{equation*} where we applied Lemma \ref{keylem} with $A=\emptyset$ and $B=\{(i,j),(w,z) \}$. \end{proof}

\begin{proof}[Proof of Theorem \ref{exthm2}] Recall \eqref{eq:kem}, the definitions of $H_i(G)$ for $i \in V$ and $T(G)$:
	\[H_i(G) := \sum\limits_{j \in V} \frac{{d}(j)}{2m}  h(j,i),\qquad  T(G) := \sum\limits_{j \in V} \frac{{d}(j)}{2m} h(i,j),\] where $m:=|E|\sim  Bin\left(\binom{n}{2},p\right)$. To begin, let  $m^*\sim  Bin\left(\binom{n}{2}-1,p\right), \;k \in \mathbb{Z},\; k\geq1$.  Proposition \ref{quo}  and the fact that $\mathcal{C} \subset \{m \geq 1\}$ yields the following
	\begin{align*} \exu{\mathcal{C}}{\frac{1}{m^k}}\pr{\mathcal{C}}& = \ex{\frac{\mathbf{1}_{\mathcal{C}}}{m^k}} \leq \ex{\frac{\mathbf{1}_{\{m \geq 1\}}}{m^k}} = \ex{\frac{\binom{n}{2}p}{(m^*+1)^{k+1}}}= \frac{2^k}{n^{2k}p^k}\left(1 + \BO{\frac{1}{np}} \right),\end{align*}where in the last step we used Lemma \ref{bmoment}  to bound the expectation term. Observe that by \eqref{eq:probC}, $\pr{\cc^c}\leq O\left(\epn   \right) $ whenever $np \geq \log n + \log\log\log n$. Thus by the Bernoulli inequality \eqref{bern} for any given $a,k\in \mathbb{Z},\; a,k \geq 1$ we have  
	\begin{align}\label{eq:mcalc} \exu{\mathcal{C}}{\frac{1}{m^k}}^{1/a}&= \frac{2^{k/a}}{n^{2k/a}p^{k/a}}\left(1 + \frac{\pr{\mathcal{C}^c }}{\pr{\mathcal{C} }}+ \BO{\frac{1}{np}} \right)^{1/a}\leq \frac{2^{k/a}}{n^{2k/a}p^{k/a}}\pepn. \end{align}
	Using  H\"older's inequality to break the product of random variables in the expectation:
	\begin{align*}
	\exu{\mathcal{C}}{T(\mathcal{G})} &\leq (1/2)\sum\limits_{j \in V} \exu{\mathcal{C}}{{d}(j)^4}^{1/4} \exu{\mathcal{C}}{1/m^{4}}^{1/4} \exu{\mathcal{C}}{ h(i,j)^2}^{1/2}.\notag\intertext{Then applying \eqref{dmom}, \eqref{eq:mcalc} and the upper bound on $\exu{\mathcal{C}}{ h(i,j)^2}$ from Lemma \ref{h^2} yields}
	\exu{\mathcal{C}}{T(\mathcal{G})} &\leq (n/2)\left((np)^4+O\left((np)^3\right)\right)^{1/4}\cdot \left[(2+\BO{\epn})/n^2p\right]\cdot n\pepn= n \pepn.\end{align*} The same upper bounds for $\exu{\mathcal{C}}{H_i(\mathcal{G})}$ follows by identical steps. By \eqref{eq:hit} we have \begin{align*}
	T(G) &= \sum\limits_{j \in V} \frac{{d}(j)}{2m}  \left(mR(i, j) + \sum\limits_{u\in V} \frac{{d}(u)}{2}\left[R(u, j)-R(u, i)\right]\right) \intertext{for $G$ connected. Applying the effective resistance bound, Lemma \ref{lowjoh}, and reducing yields}
	T(G)  &\geq \frac{m}{({d}(i)+1)} -\frac{{d}(i)}{2({d}(i)+1)} +\!\!\sum\limits_{\substack{j \in V\\j\neq i} } \frac{{d}(j)}{2({d}(j)+1)} +\!\! \!\! \sum\limits_{\substack{j,u \in V\\ j\neq u}}\!\!\frac{{d}(j){d}(u)}{2m({d}(u)+1)}-\!\!\!\!\sum\limits_{j,u\in V}\! \!\frac{{d}(u){d}(j)}{4m}R(u, i) . \intertext{Applying ${d}(i)/({d}(i)+1)= 1 - 1/({d}(i)+1)$ and the bound ${d}(i)/({d}(i)+1) \leq 1$ yields}
	T(G) &\geq \frac{m}{{d}(i)+1} +\frac{3n}{2}-2-  \sum\limits_{u \in V}\frac{3}{2({d}(u)+1)}-\sum\limits_{u \in V}\frac{{d}(u)}{2}R(u,i).
	\end{align*}Again by a similar procedure we have the following for the stationary hitting time $H_i(G)$
	\begin{align*}H_i(G)&= \sum\limits_{j \in V} \frac{{d}(j)}{2m}   \left(mR( j,i) + \sum\limits_{u\in V} \frac{{d}(u)}{2}\left[R(u, i)-R(u, j)\right]\right)\geq \frac{n-1}{2} -\sum\limits_{j \in V} \frac{1}{2({d}(j)+1)} \\&+\frac{m-1}{({d}(i)+1)} -1+  \sum\limits_{\mathclap{u\in V, u\neq i}} \frac{{d}(u)}{2}\!\left(\!\frac{1}{{d}(i)+1}+\frac{1}{{d}(u)+1}\!\right)-\!\sum\limits_{j,u\in V} \!\frac{{d}(u){d}(j)}{4m}R(u, j) 
	\\&\geq n +\frac{2m-2}{{d}(i)+1} - \sum\limits_{u \in V} \frac{1}{{d}(u)+1} -\frac{7}{2} -\sum\limits_{\mathclap{j,u \in V}}\frac{{d}(u){d}(j)}{4m}R(u,j) .\end{align*}Let $ \mathcal{D}$ be the event $\{m \geq n^2p/2 -a\sqrt{n^2p/2}\}\cap \{{d}(j) \leq np +a\sqrt{np} \}$ where $a= 3\sqrt{\log\log n}$ if $np = O(\log n)$ and $a=3\sqrt{\log n}$ if $np = \omega(\log n)$. Now by Lemma \ref{chb} we obtain\[\pru{\cc}{\cd}  = \left(1- \exp\left(-a^2/2\right)/\pr{\cc} -\exp\left(-a^2/2(1+a/3\sqrt{np})\right)/\pr{\cc}  \right) = 1- o\left(1/np\right). \] By H\"older's inequality \eqref{holder}, $1\geq \mathbf{1}_{\cd}$ and the bound on $\pru{\cc}{\cd}$ in the line above we have
	\begin{align*}
	\exu{\mathcal{C}}{H_i(\mathcal{G})} &\geq n +  2\frac{\binom{n}{2}p -a\sqrt{\binom{n}{2}p}-1}{np +a\sqrt{np}+1}\pru{\cc}{\cd}  -n\cdot\exu{\mathcal{C}}{\frac{1}{{d}(u)+1}}-\frac{7}{2}\\ &\quad - (n/4)\exu{\mathcal{C}}{{d}(j)^4}^{1/4}\exu{\mathcal{C}}{1/m^4}^{1/4} \exu{\mathcal{C}}{{d}(u)^2R(u, j)^2}^{1/2} = n\mepn .\intertext{The last equality comes from applying estimates to the expectation terms which are given by  Lemma \ref{bmoment} in Appendix \ref{appen} and \eqref{dmom}, \eqref{eq:mcalc}, and Lemma \ref{keylem} respectively. Similarly we have}
	\exu{\mathcal{C}}{T(\mathcal{G})} &\geq   \frac{\binom{n}{2}p -2a\sqrt{\binom{n}{2}p}}{np +a\sqrt{np}}\pru{\cc}{\cd} +\frac{3n}{2} -2  -\frac{3n}{2}\exu{\mathcal{C}}{\frac{1}{{d}(u)+1}}- \frac{n}{2} \exu{\mathcal{C}}{{d}(u)^2R(u, i)^2}^{1/2},
	\end{align*}which also evaluates to $n\mepn$. \end{proof}

\subsection{Proof of Theorems \ref{concentrationnew} \& \ref{concentration}}  
\begin{lemma}
	Let $\cg \sim \cg(,n,p)$ satisfy \eqref{sparsecon}. Then $\exu{\cc}{H_i(\mathcal{G})^2},\exu{\cc}{T(\mathcal{G})^2}=n^2\left(1\pm O\!\left(\varepsilon_n \right)\right)$.
\end{lemma}
\begin{proof}We will first bound $ \exu{\mathcal{C}}{  h(i, j)^3}$ from above. Now similarly to Lemma \ref{h^2}, \begin{align}\label{eq:upexpbdd2}
	\exu{\mathcal{C}}{  h(i, j)^3}&= \frac{1}{8} \sum_{x,y,z \in V} \exu{\mathcal{C}}{{d}(x){d}(y){d}(z)R(i,j)^3} = \frac{1}{8} \sum_{x,y,z \in V} (8\pm \BO{\eps_n} ),
	\end{align} which equals $n^3\left(1\pm O\!\left(\varepsilon_n \right)\right)$ - where above we applied Tetali's formula \eqref{eq:hit}, cancelled terms by symmetry and then applied Lemma \ref{keylem} with $A=\{x,y,z\}$ and the multi-set $B=\{(i,j), (i,j),(i,j) \}$. 
	
	By the definition \eqref{eq:kem} of $T(\cg)$ and H\"{o}lder's inequality \eqref{holder} we have
	\begin{align*}
		\exu{\cc}{T(\cg)^2} &= \exu{\cc}{\left(\sum\limits_{j \in V}\frac{{d}(j)}{2m} h(i,j)\right)^{2}}= \exu{\cc}{\sum\limits_{j,k \in V} \frac{{d}(j){d}(k)}{(2m)^2}  h(i,j)h(x,y)}\\
		&\leq \sum\limits_{j,k\in V}\left(\exu{\cc}{{d}(j)^9}\exu{\cc}{{d}(k)^9}\exu{\cc}{1/(2m)^{18} }\right)^{1/9}\left(\exu{\cc}{ h(i,j)^3}\exu{\cc}{h(x,y)^3}\right)^{1/3}.
		\intertext{Applying the bounds \eqref{dmom}, \eqref{eq:mcalc} and \eqref{eq:upexpbdd2} respectively then Bernoulli's inequality \eqref{bern} gives}
		\exu{\cc}{T(\cg)^2}  &\leq \frac{n^2}{2^2}\left((np)^{9}+O\!\left((np)^{8}\right)\right)^{\frac{2}{9}} \left(\frac{2^{18}+\BO{\epn}}{n^{36}p^{18}}\right)^{\frac{1}{9}}	\!\left(n^6\!\pepn\right)^{\frac{1}{3}}=n^2\pepn.\end{align*}Then by Jensen's inequality and the lower bound on $\exu{\cc}{T(\cg)}$ proved earlier we have
	\begin{align*}
		\exu{\cc}{T(\cg)^2} &\geq \exu{\cc}{T(\cg)}^2 \geq \left(n\mepn\right)^2 = n^2\mepn.
	\end{align*}  The exact same calculations yield the same bounds for $\exu{\mathcal{C}}{H_i(\mathcal{G})^2}$.\end{proof}
We prove Theorems \ref{concentrationnew} \& \ref{concentration} (together) by Chebychev's inequality and our moment bounds.
\begin{proof}[Proof of Theorems \ref{concentrationnew} \& \ref{concentration}]
	Let $X \in \{h(i,j), \; H_i(\cg) ,\; T(\cg),\; cc_i\}$ where $i,j \in V$ and recall $\exu{\cc}{\cdot} = \ex{\cdot |\cc}$. We have the following for these $X$ by Theorem \ref{exthm}
	\[\var\left(X\big|\cc\right) = n^2\pepn  - \left( n\pepn \right)^2 = O\!\left(n^2\epn  \right).  \] We can also calculate the conditional variance of $K(\cg)$ by Theorem \ref{exthm}, this yields \begin{align*}\var\left(K(\cg)\big|\cc\right)&= (n^2/p^2)\pepn  - \left( n\pepn /p\right)^2= O\!\left( n\epn/p  \right).\end{align*} By the Chebyshev inequality \cite[Theorem 4.1.1]{alon2008probabilistic} for each of the above
	\begin{align*} \pr{\Big|X - \ex{X\big|\cc}\big| \geq \lambda(n)\sqrt{\var\left(X|\cc\right)}\,\big|\, \mathcal{C} } &\leq 1/\lambda(n)^2.\end{align*}
	For $X$ above we have $\var\left(X|\cc\right)= O\!\left(\ex{X\big|\cc}^2\epn  \right) $ by Theorem \ref{exthm}, thus there exists some $K$ independent of $n$ and $X$ such that $\sqrt{\var\left(X|\cc\right)}<\ex{X\big|\cc}\sqrt{K \epn },$ for large $n$. By choosing $\lambda(n) = \sqrt{f(n)/K}$ for any function $f(n)$ we have
	\begin{align*} \pr{\left|X - \ex{X\big|\cc}\right| > \ex{X\big|\cc}\sqrt{f(n)\epn }\;\Big|\,\cc} &\leq K/f(n) = O\!\left(1/f(n) \right). \end{align*}The result follows since $\pr{\ce} \leq \pr{\ce|\cc} + \pr{\cc^c}$, for any event $\ce$.\end{proof}

\subsection*{Acknowledgements}
I would like to thank my supervisors Agelos Georgakopoulos and David Croydon for their guidance and the reviewers who's comments have greatly improved the paper. This work was largely composed while I was part of the MASDOC DTC at the University of Warwick, supported by EPSRC grant No. EP/HO23364/1 and ERC Starting Grant no.\ 639046 (RGGC). It was completed while I supported by ERC Starting Grant no.\ 679660 (DYNAMIC MARCH).

\bibliographystyle{plain}

\appendix 

\section{Appendix}\label{appen}
We make frequent use of the following inequalities. Bernoulli: Let $x\geq -1$, then
\begin{equation}\label{bern}(1 + x)^r \leq 1 + rx\;\text{    for    }\;0 \leq r \leq 1 \quad  \text{    and    }\quad (1 + x)^r \geq 1 + rx\;\text{    for    }\; r \geq 1  .\end{equation}
H\"{o}lder: For $1\leq k\leq  n$ let $X_k$ be r.v.'s and $p_k\in [1,\infty)$   where $\sum _{k=1}^{n}1/p_{k}=1$ and $\ex{X_k^{p_k}}$ exists, then \begin{equation}\label{holder}
\ex{X_1\cdots X_{n}}\leq \ex{X_{1}^{p_{1}}}^{1/p_1}\cdots \ex{X_{n}^{p_{n}}}^{1/p_n}.
\end{equation}
Coupling: If $X,Y$ are real random variables on a probability space $(\Omega, \mathfrak{F}, \mathbb{P})$, then for any $B\subset \mathbb{R}$,
\begin{equation}\label{prk}
\Big|\pr{X\in B}-\pr{Y \in B}\Big|\leq \pr{X\neq Y}.  
\end{equation} 

\begin{lemma}[Chernoff bounds]\label{chb} If $X\sim  Bin(n,p)$, then for any $a>0$, $b<np$ and $c>np$
	\begin{enumerate}[(i)]
		\item \label{itm:chern1} $\mathbb{P}\left[X < np -a \right] \leq \exp\left(-\frac{a^2}{2np}\right), \quad \text{and}\quad  \mathbb{P}\left[X > np +a \right] \leq \exp\left(-\frac{a^2}{2(np + a/3)}\right),$
		\item \label{itm:chern2} $\mathbb{P}\left[X < b \right] \leq e^{-np} \left(\frac{enp }{b} \right)^b,\qquad \qquad \text{and}\quad \mathbb{P}\left[X > c \right] \leq e^{-np} \left(\frac{enp }{c} \right)^c.$
	\end{enumerate}
\end{lemma} 
\begin{proof}
	For (i) see \cite[Thm.\ 2.4]{Chung} and \cite[Thm.\ 2.15]{Chung} with $ \eps = 1-\frac{b}{np}$ for (ii). 
\end{proof}
We also have the following closed form for moments of binomial random variables,
\begin{theorem}[{\cite[Theorem 4.1]{Knob}}]\label{knobmo} Let $X \sim Bin(n,p)$, $n^{\underline{i}}:=n(n-1)\dots(n-i+1)$ and $S(d,i)$ be the Stirling partition number of $d$ items into $i$ subsets. Then for $d \geq 0$,
	\[\ex{X^d} = \sum\limits_{i=0}^dS(d,i)p^in^{\underline{i}},\qquad\text{where}\qquad S(d,i):= \frac{1}{i!}\sum\limits_{k=0}^i(-1)^{k+i}\binom{i}{k} k^d. \]
	
\end{theorem}
\noindent Let $X\sim Bin(n,p)$, $0<p:=p(n)<1$ and $d\geq 0$ fixed. Then by Theorem \ref{knobmo} we have
\begin{equation}\label{dmom}
\ex{X^d} = S(d,d)p^dn^{\underline{d}} \pm O\left(p^{d-1}n^{\underline{d-1}}\right) =(np)^d \pm O\left((np)^{d-1}\right).
\end{equation} 

\begin{proposition}\label{quo}
	Let $X\sim Bin\left(n,p\right),\;Y \sim  Bin\left(n-1,p\right), \, \alpha \in \mathbb{Z},\, \alpha\geq 1.$ Then \begin{align*}
	\mathbb{E}\left[ \frac{\mathbf{1}_{\{X\geq 1\}}}{X^{\alpha}}\right]&:=\sum\limits_{k=1}^n \frac{1}{k^\alpha} \binom{n}{k}p^k(1-p)^{n-k}= \sum\limits_{k=0}^{n-1} \frac{1}{(k+1)^\alpha} \binom{n}{k+1}p^{k+1}(1-p)^{n-1-k} \notag \\
	& = \sum\limits_{k=0}^{n-1} \frac{np}{(k+1)^{\alpha+1}} \binom{n-1}{k}p^k(1-p)^{(n-1)-k}=\mathbb{E}\left[ \frac{np}{\left(Y+1\right)^{\alpha+1}}\right].
	\end{align*}
\end{proposition}
\begin{lemma}\label{bmoment}
	Let $X_n \sim  Bin(n,p)$ for $p:=p(n)$ with $np \rightarrow \infty$, $a\in \mathbb{R}, \;b\in \mathbb{Z}, \,a,b >0$. Then
	\[\frac{1}{(a+ np)^{b}}  \leq \ex{\frac{1}{(a+ X_n)^{b}}} \leq \frac{1}{(a+ np)^{b}} + O\left(\frac{1}{(np)^{(b+1)} }\right).\]
\end{lemma}

\begin{proof}Let $f(x) :=f_{a,b}(x)= (a+x)^{-b}$ for any constants $a,b >0$. The lower bound follows from Jensen's inequality since $f(x)$ is convex for $a,b >0$. 
	
	Let $\mu_n = \ex{X_n}=np$. When $np \rightarrow \infty$ it is possible to find some $r:=r(n)$ such that $r= \omega(\sqrt{np \log(np)})$ and $ r = o(np)$. The Chernoff bound, Lemma \ref{chb} \eqref{itm:chern1}, then yields
	\[\pr{X_n \leq \mu_n -r} \leq \exp\left(-r^2/2\mu_n \right) = o(1/np).\] With this $r$ we can achieve the following a priori upper bound for any $b\geq 1$:
	\begin{equation}\label{apri}
	\ex{f(X_n)} \leq \frac{1}{a^b}\pr{X_n \leq \mu_n -r} + f(\mu_n-r)\pr{X_n > \mu_n -r} = (1+o(1))f(\mu_n).
	\end{equation}
	By Taylor's theorem there is some $\xi_n$ between $X_n$ and $\mu_n$ such that  
	\[f(X_n) = f(\mu_n) + f'(\mu_n)\left(X_n-\mu_n \right) + f''(\xi_n)\left(X_n-\mu_n \right)^2 . \]Using H\"{o}lder's inequality \eqref{holder} and the fact $f(x)$ is decreasing when $x>0$, we have
	\begin{align}\label{reverse}
	&\left(\ex{f(X_n)} - f(\mu_n)\right)^2 \leq \left(f'(\mu_n)\ex{X_n-\mu_n}+\ex{f''(\xi_n)\left(X_n-\mu_n\right)^2}\right)^2\notag \\
	&\leq \ex{f''(\xi_n)^2}\ex{\left(X_n-\mu_n\right)^4} \leq \ex{f''(X_n)^2\mathbf{1}_{\{X_n\leq \mu_n\}} }\ex{\left(X_n-\mu_n\right)^4}\\
	&\quad +\ex{f''(\mu_n)^2\mathbf{1}_{\{X_n> \mu_n\}} }\ex{\left(X_n-\mu_n\right)^4}\leq (2+o(1))f''(\mu_n)^2 \ex{\left(X_n-\mu_n\right)^4}.\notag
	\end{align}The last inequality follows by \eqref{apri} since $f''(\mu_n) = b\cdot(b+1)\cdot (a+\mu_n)^{-(b+2)}$. Observe 
	\begin{equation}\label{4moment}
	\ex{\left(X_n-\mu_n\right)^4} = np(1-p)\left(3p(n-2) -3p^2(n+2)+1\right) = O((np)^2), 
	\end{equation}this can be calculated using the binomial moment generating function or by Theorem \ref{knobmo}. Hence by \eqref{reverse}, \eqref{4moment} and $(f_{a,b}(x))^{''}=b(b+1)f_{a,(b+2)}(x)$, we have
	\[\ex{f(X_n)} \leq f(\mu_n) + \left( O\left((a+\mu_n)^{-2(b+2)} \right)\cdot O((np)^2) \right)^{1/2} = \frac{1}{(a+np)^b} + O \left(\frac{1}{(np)^{b+1}} \right).  \]\end{proof} 
Finally we shall prove Proposition \ref{tight} which shows tightness for the concentration results. 

\begin{proof}[Proof of Proposition \ref{tight}] Let $X_{d}$ be the number of vertices with degree $d$. For the first case:
	\begin{align*}\ex{X_1}&= n\cdot \binom{n-1}{1}p(1-p)^{n-2} = n^2p e^{-\log n  -O( \log\log\log n )}\geq  \frac{\log n }{(\log\log n)^{O({1})}}.  \end{align*} This implies that, for any fixed $t$, $\lim_{n\rightarrow \infty} \pr{|X_{1}|\geq t} = 1 $ by \cite[Thm 3.1]{bollobasrandom}. Thus w.h.p. there is at least one pair of vertices $i,j$ both with degree $1$ and so $R(i,j)\geq 1$ by Lemma \ref{lowjoh}. Since the number of edges $m$ is distributed $Bin\left(\binom{n}{2},p\right)$ there are $n^2p/2(1-\lo{1})$ edges w.h.p by Lemma \ref{chb}. Thus by the Commute time formula \eqref{commute} we have $\kappa(i,j) = 2m\cdot R(i,j)\geq (1-o(1)) n\log(n)$ and since $\kappa(i,j) = h(i,j) + h(j,i)$ at least one of $h(i,j)$ of $h(j,i)$ is greater than $ n\log(n)/3$ w.h.p.

	For the case $np = (c+o(1))\log(n)$ let $k:k(\varepsilon):= (1-\varepsilon )np$ for some $0< \varepsilon<1$ and observe  
	\begin{align*}
	\ex{X_k} &= n \binom{n-1}{k}p^k ( 1-p)^{n -1 -k  } \geq \frac{n}{\sqrt{2\pi k}} \left(\frac{e}{(1-\varepsilon)} \right)^{(1-\varepsilon)np } e^{-np}(1-o(1))\intertext{Recall $ -\log(1-t) \geq t +t^2/2 $ for $t<1$. In a similar vein to the proof of \cite[Theorem 2.2]{jonasson1998cover}:}
	\ex{X_k} & \geq  \frac{n}{3\sqrt{k}} e^{-\varepsilon np -(1- \varepsilon )\log(1-\varepsilon)np}\geq  \frac{n}{3\sqrt{ k}} e^{-\varepsilon np +(1- \varepsilon )(+\varepsilon +\varepsilon^2/2)np} \geq  \frac{n e^{-\frac{\varepsilon^2(1+\varepsilon)np}{2} }}{3\sqrt{ k}}
	\end{align*}So for any $0<\varepsilon <1$ satisfying $\frac{\varepsilon^2(1+\varepsilon)c}{2}<1 $ we have that $\ex{X_{k}} \rightarrow \infty$ (a concrete example would be $\varepsilon = \sqrt{1/(c+1)}$). Thus again by \cite[Thm 3.1]{bollobasrandom} there are at least two vertices $i,j$ with degree less than $(1-\varepsilon)np$ w.h.p. Thus, as before, $\kappa(i,j) = 2m\cdot R(i,j)\geq n^2p\frac{2}{(1-\varepsilon )np}= \frac{2n}{1-\varepsilon}.$ Thus one or both of  $h(i,j)$ or $h(j,i)$ must be greater than $(1+a)n$ for some $a>\tfrac{\varepsilon}{2(1 -\varepsilon)}>0$ w.h.p.. \end{proof}

\end{document}